\documentclass{amsart}
\usepackage[
top=30truemm,
bottom=30truemm,
left=30truemm,
right=30truemm]{geometry}
\usepackage{amsfonts, amssymb, amsmath, verbatim, amsthm, mathrsfs, amscd, enumerate, ascmac, fancyhdr, caption, hyperref, framed,ulem,subfigure}
\usepackage{bbm}
\usepackage{eucal}
\usepackage{tikz}
\usetikzlibrary{shapes.geometric}
\usetikzlibrary{arrows.meta,arrows}
\usepackage{graphicx} 
\usepackage{pgf, tikz}
\usepackage{pgfplots}
\usepackage{float}
\usepackage{graphicx,color}
\usepackage{tcolorbox}
\usepackage{mathtools}

\def\XXint#1#2#3{{\setbox0=\hbox{$#1{#2#3}{\int}$}
    \vcenter{\hbox{$#2#3$}}\kern-.5\wd0}}

\usepackage[framemethod=tikz]{mdframed}
\mdfsetup{linecolor=blue,backgroundcolor=gray!10,roundcorner=10pt,leftmargin=2pt,innerleftmargin=2pt,innerrightmargin=2pt}
\usepackage{url}

\usepackage{xcolor}
\usepackage{tikz}
\usetikzlibrary{positioning,intersections, calc, arrows,decorations.markings}

\pgfplotsset{compat=1.18}

\newtheorem{theorem}{Theorem}[section]

\newtheorem{claim}[theorem]{Claim}

\newtheorem{lemma}[theorem]{Lemma}
\newtheorem{proposition}[theorem]{Proposition}
\theoremstyle{definition}

\newtheorem{remark}[theorem]{Remark}
\newtheorem{question}[theorem]{Question}

\newtheorem{conjecture}[theorem]{Conjecture}
\numberwithin{equation}{section}

\newtheorem{job}{JOB}

\def\eqalign#1{\null\,\vcenter{\openup\jot\mathsurround\dimen12
  \ialign{\strut\hfil$\textstyle{##}$&$\textstyle{{}##}$\hfil
      \crcr#1\crcr}}\,}

\def\supp{{\rm supp\,}}
\def\hess{{\rm Hess\,}}
\def\diam{{\rm diam\,}}
\def\dist{{\rm dist\,}}
\def\cvx{{\rm cvx\,}}

\def\ST{{\rm ST}}
\def\MT{{\rm MT}}

\title[A phase--space approach to weighted extension inequalities]{A phase-space approach to weighted Fourier extension inequalities}

\author[Bennett]{Jonathan Bennett}
\address[Jonathan Bennett]{School of Mathematics, The Watson Building, University of Birmingham, Edgbaston, Birmingham, B15 2TT, England.}
\email{J.Bennett@bham.ac.uk}

\author[Gutierrez]{Susana Guti\'errez}
\address[Susana Guti\'errez]{School of Mathematics, The Watson Building, University of Birmingham, Edgbaston, Birmingham, B15 2TT, England.}
\email{S.Gutierrez@bham.ac.uk}

\author[Nakamura]{Shohei Nakamura}
\address[Shohei Nakamura]{Department of Mathematics, Graduate School of Science, Osaka University, Toyonaka, Osaka 560-0043, Japan. School of Mathematics, The Watson Building, University of Birmingham, Edgbaston, Birmingham, B15 2TT, England.}
\email{s.nakamura@bham.ac.uk}

\author[Oliveira]{Itamar Oliveira}
\address[Itamar Oliveira]{School of Mathematics, The Watson Building, University of Birmingham, Edgbaston, Birmingham, B15 2TT, England.}
\email{i.oliveira@bham.ac.uk, oliveira.itamar.w@gmail.com}

\begin{document}

\date{\today}


\maketitle
\begin{abstract}
The purpose of this paper is to expose and investigate natural phase-space formulations of two longstanding problems in the restriction theory of the Fourier transform. These problems, often referred to as the Stein and Mizohata--Takeuchi conjectures, assert that Fourier extension operators associated with rather general (codimension 1) submanifolds of Euclidean space, may be effectively controlled by the classical X-ray transform via weighted $L^2$ inequalities. Our phase-space formulations, which have their origins in recent work of Dendrinos, Mustata and Vitturi, expose close connections with a conjecture of Flandrin from time-frequency analysis, and rest on the identification of an explicit ``geometric" Wigner transform associated with an arbitrary (smooth strictly convex) submanifold $S$ of $\mathbb{R}^n$. Our main results are certain natural ``Sobolev variants" of the Stein and Mizohata--Takeuchi conjectures, and involve estimating the Sobolev norms of such Wigner transforms by geometric forms of classical bilinear fractional integrals. Our broad geometric framework allows us to explore the role of the curvature of the submanifold in these problems, and in particular we obtain bounds that are independent of any lower bound on the curvature; a feature that is uncommon in the wider restriction theory of the Fourier transform.  Finally, we provide a further illustration of the effectiveness of our 
analysis by establishing a form of Flandrin's conjecture in the plane with an $\varepsilon$-loss.
While our perspective comes primarily from Euclidean harmonic analysis, the procedure used for constructing phase-space representations of extension operators is well-known in optics.

\end{abstract}
\tableofcontents
\section{Introduction}\label{Sect:intro}
\subsection{Background: the Stein and Mizohata--Takeuchi problems}
A central objective of modern harmonic analysis is to reach an effective quantitative understanding of Fourier
transforms of measures supported on submanifolds of Euclidean space, such as the sphere or paraboloid. Problems of this type
are usually formulated in terms of \textit{Fourier extension operators}:
to a smooth codimension-1 submanifold $S$ of $\mathbb{R}^n$, equipped with surface measure $\mathrm{d}\sigma$, we associate the extension operator
\begin{equation}\label{def:ext}
\widehat{g\mathrm{d}\sigma}(x):=\int_S g(u)e^{-2\pi ix\cdot u}\mathrm{d}\sigma(u);
\end{equation}
here $g\in L^1(\mathrm{d}\sigma)$ and $x\in\mathbb{R}^n$. The extension operator \eqref{def:ext} is often referred to as an \textit{adjoint restriction operator}, as its adjoint restricts the $n$-dimensional Fourier transform of a function to the submanifold $S$.
The estimation of extension operators in various settings is known as (Fourier) \textit{restriction theory}. A key instance of this is the celebrated \textit{restriction conjecture}, which concerns bounds of the form $\|\widehat{g\mathrm{d}\sigma}\|_q\lesssim\|g\|_p$. Surprisingly many problems from
across mathematics call for such an understanding, from dispersive PDE to analytic number theory; see \cite{Stovall} for a recent survey. Such connections are often quite intimate, as hopefully this paper serves to illustrate -- in this case with regard to optics, or optical field propagation.

In this paper we look to estimate extension operators in the setting of $L^2$ norms with respect to general weight functions $w$. This setting has been the subject of some attention since the influential work of Stein and others in the 1970s in the closely related context of Bochner--Riesz summability. At its centre is a variant of a question posed by Stein in the 1978 Williamstown conference on harmonic analysis \cite{S} (see \cite{BCSV} for further historical context). In its global form, for a given $S$, this asks whether there is a constant $C<\infty$ for which
\begin{equation}\label{Steinvgen}
\int_{\mathbb{R}^n}|\widehat{g\mathrm{d}\sigma}(x)|^2w(x)\mathrm{d}x\leq C\int_S|g(u)|^2\sup_{v\in T_uS}Xw(N(u),v)\mathrm{d}\sigma(u)
\end{equation}
for all nonnegative weight functions $w$. Here $N:S\rightarrow\mathbb{S}^{n-1}$ is the \textit{Gauss map}, and $X$ denotes the classical \textit{X-ray transform}
\begin{equation}\label{Defx}
Xw(\omega,v):=\int_{-\infty}^\infty w(v+t\omega)\mathrm{d}t,
\end{equation}
where $\omega\in\mathbb{S}^{n-1}$ and $v\in\langle\omega\rangle^\perp$ together parametrise the Grassmanian manifold of lines $\ell=\ell(\omega,v):=\langle\omega\rangle+\{v\}$ in $\mathbb{R}^n$; here $T_uS=\langle N(u)\rangle^\perp$ denotes the tangent space of $S$ at the point $u\in S$. This is a natural inequality for a number of reasons, and it is instructive to begin by considering the simple case where $g$ is the indicator function of a small cap (the intersection of $S$ with a small ball in $\mathbb{R}^n$). The key observation is that $|\widehat{g\mathrm{d}\sigma}|^2$ is then bounded below on a neighbourhood of a line segment with direction normal to $S$, so that the left hand side of \eqref{Steinvgen} computes a variant of the X-ray transform of the weight $w$. The inequality \eqref{Steinvgen} therefore proposes that $|\widehat{g\mathrm{d}\sigma}|^2$ concentrates on lines, or families of lines, rather more generally. An affirmative answer to this question is easily given in the case that $S$ is contained in a hyperplane -- a fact that follows quickly from Plancherel's theorem. More substantial results in support of \eqref{Steinvgen} have been obtained for restricted classes of weights, notably when $S$ is the sphere $\mathbb{S}^{n-1}$ and the weights are radial \cite{CRV, BRV, BBC}; see also \cite{BNS} and the references there. Inequalities of this general type, where an operator is estimated with respect to a general weight function, are often referred to as Fefferman--Stein inequalities -- see \cite{Beltran} for a recent example. We shall also be interested in the simpler \textit{Mizohata--Takeuchi} inequality
\begin{equation}\label{MTvgen}
    \int_{\mathbb{R}^n}|\widehat{g\mathrm{d}\sigma}(x)|^2w(x)\mathrm{d}x\leq C\|g\|_{L^2(\mathrm{d}\sigma)}^2\sup_{(u,v)\in TS}Xw(N(u),v),
\end{equation}
where 
the supremum is restricted to $u\in\supp(g)$, as suggested by \eqref{Steinvgen}, and $TS$ denotes the tangent bundle of $S$.
This emerged independently of \eqref{Steinvgen} through work of Mizohata and Takeuchi on the well-posedness of Schr\"odinger equations in the 1980s. We refer to \cite{BRV} and the references there for further context.
\begin{remark}[The strength of \eqref{Steinvgen}]\label{Stein philosophy}
The original motive for establishing \eqref{Steinvgen}, or some appropriate variant of it, is that it would allow the restriction conjecture to follow (and almost immediately) from the \textit{Kakeya maximal function conjecture}, the Kakeya maximal function being a close relative of 
$$
\sup_{v\in T_uS}Xw(N(u),v),
$$
at least when $S$ is suitably curved.
We refer to \cite{BN} and the references there for further details and discussion. In the original setting proposed by Stein, this amounts to the implication of the Bochner--Riesz conjecture from the Nikodym (or Kakeya) maximal conjecture. There is a number of precedents for this sort of integro-geometric control of oscillatory integral operators -- see for example \cite{BCSV, BBen}.
\end{remark}
\begin{remark}[Failure of the global inequalities \eqref{Steinvgen} and \eqref{MTvgen}]\label{Rem:Cairo}
Very recently, and since earlier drafts of this paper, Cairo \cite{Cairo} has succeeded in constructing a counterexample to \eqref{MTvgen} (and thus \eqref{Steinvgen}) whenever $S$ is \textit{not} contained in a hyperplane. However, her subtle example does not exclude the possibility that the local variants
\begin{equation}\label{Steinvgenloc}
\int_{B(0,R)}|\widehat{g\mathrm{d}\sigma}(x)|^2w(x)\mathrm{d}x\lesssim R^\alpha\int_S|g(u)|^2\sup_{v\in T_uS}Xw(N(u),v)\mathrm{d}\sigma(u)
\end{equation}
and 
\begin{equation}\label{MTvgenloc}
    \int_{B(0,R)}|\widehat{g\mathrm{d}\sigma}(x)|^2w(x)\mathrm{d}x\lesssim R^\alpha\|g\|_{L^2(\mathrm{d}\sigma)}^2\sup_{(u,v)\in TS}Xw(N(u),v)
\end{equation}
of \eqref{Steinvgen} and \eqref{MTvgen} (resp.)
might hold for exponents $\alpha>0$; here $R$ denotes a large parameter. That \eqref{Steinvgenloc} (and thus \eqref{MTvgenloc}) holds for \textit{some} $\alpha>0$ is an elementary exercise, and we refer to \cite{CIW} for recent local results of this type. In order to be meaningful for general $w$ the inequalities \eqref{Steinvgen} and \eqref{MTvgen} may therefore be qualified with the additional assumption that $w$ is supported in a ball of fixed radius $R\gg 1$, accepting some growth in $R$ in the constant factors. We clarify that such considerations are not relevant to the results presented in this paper.
 \end{remark}

\begin{remark}[The role of curvature]\label{Remark:no curvature}
Somewhat unusually in the setting of Fourier extension estimates it appears that the above Stein and Mizohata--Takeuchi-type inequalities should not require that $S$ has nonvanishing curvature; we have already noted that \eqref{Steinvgen} is easily verified when $S$ is a \textit{hyperplane}. Related to this fact is the observation that \eqref{Steinvgen} and \eqref{MTvgen} are dilation invariant in the sense that their validity for a given $S$ and a given (dilation-invariant) class of weights, implies their validity for any isotropic dilate $kS$ of $S$, uniformly in $k>0$; this follows by a routine scaling argument. This scale-invariance is important in applications, as may be seen in the established setting of the sphere and radial weights -- see \cite{BRV}.
\end{remark}
\subsection{Phase-space formulations}
Recently in the setting of quadratic submanifolds, Dendrinos, Mustata and Vitturi \cite{DMV} observed that the Mizohata--Takeuchi inequality \eqref{MTvgen} may be reformulated in terms of the classical Wigner distribution, providing it with a natural phase-space interpretation.
The purpose of this paper is to establish and explore such phase-space formulations of the Stein and Mizohata--Takeuchi inequalities for quite general (codimension-1) submanifolds, exposing the role played by the underlying geometry. The starting point is the surprising observation that a rather general Fourier extension operator (in modulus square) has a natural and explicit \textit{phase-space representation}, 
namely,
\begin{equation}\label{PSrep}
|\widehat{g\mathrm{d}\sigma}|^2=X_S^*W_S(g,g);
\end{equation}
see the forthcoming Proposition \ref{Prop:genBGNOWigner}.
Here $W_S(g_1,g_2):TS\rightarrow\mathbb{R}$ is a certain geometric (or $S$-carried) Wigner transform, and $X_S$ is the pullback of the X-ray transform by the Gauss map; concretely, $$X_Sw(u,v):=Xw(N(u),v)$$ for $(u,v)\in TS$. Such phase-space representations have their origins in quantum mechanics in the case that $S$ is the paraboloid -- a perspective that we develop in Section \ref{Sect:para}. They are also well-known in optics, particularly when $S$ is the paraboloid or the sphere, and we develop this perspective in Section \ref{Sect:Helm}. As we shall see in the later sections, identifying a suitable Wigner transform $W_S$ \textit{explicitly} in terms of the geometry of a \textit{general} (strictly convex) submanifold $S$ requires some careful geometric analysis. This is one of the main achievements of this paper, and it is hoped that it will also find some interesting applications beyond harmonic analysis. From the point of view of harmonic analysis, our treatment of these surface-carried Wigner transforms naturally involves controlling associated surface-carried singular integral and maximal averaging operators, which we hope will be of some independent interest.

By duality the representation \eqref{PSrep} immediately gives rise to the phase-space integral formula
\begin{equation}\label{frep}
\int_{\mathbb{R}^n}|\widehat{g\mathrm{d}\sigma}(x)|^2w(x)\mathrm{d}x=\int_{TS}W_S(g,g)(u,v)X_Sw(u,v)\mathrm{d}v\mathrm{d}\sigma(u),
\end{equation}
leading to phase-space formulations of the Stein and Mizohata--Takeuchi problems.
Here the integral on the tangent bundle $TS$ is defined in the usual way, by first integrating with respect to Lebesgue measure on the tangent space $T_uS$, and then with respect to surface measure $\mathrm{d}\sigma(u)$ on $S$.
\begin{remark}[Connections with Flandrin's conjecture]\label{remark:flan}
The phase-space formulation of the Mizohata--Takeuchi problem has striking similarities with a conjecture of Flandrin \cite{Flan} and its variants \cite{Lerner} in the setting of the classical Wigner transform $W$. A recent form of this conjecture states that
\begin{equation}\label{shortflan}
\iint_KW(g,g)\lesssim\|g\|_2^2
\end{equation}
uniformly over all convex subsets $K$ of phase-space; this was originally formulated by Flandrin with constant $1$, although a counterexample to this stronger statement was constructed recently in \cite{Duy}.
The methods of this paper are also effective here, and we illustrate this in Section \ref{Sect:Flandrin}, establishing a form of this conjecture in the plane involving an $\varepsilon$-loss in the measure of $K$, and by establishing that the Flandrin-type conjecture \eqref{shortflan} implies the parabolic Mizohata--Takeuchi inequality under a simple convexity assumption on the weight function $w$.
\end{remark}
\begin{remark}[Connections to maximally-modulated singular integrals]\label{remark:maxmod}
The Flandrin-type conjecture \eqref{shortflan} in the plane (and thus \eqref{Steinvgen} and \eqref{MTvgen}) is also intimately connected to boundedness questions for the \textit{maximally-modulated bilinear Hilbert transform}
$$
H_*(f_1,f_2)(x):=\sup_{\lambda\in\mathbb{R}}\left|\int_{\mathbb{R}}f_1\left(x+\frac{y}{2}\right)f_2\left(x-\frac{y}{2}\right)e^{i\lambda y}\frac{dy}{y}\right|.
$$
We refer to Section \ref{Sect:Flandrin} for details.
\end{remark}
Evidently the phase-space formula \eqref{frep} goes some way to motivate the original inequalities \eqref{Steinvgen} and \eqref{MTvgen}. 
The first remark to make is that the most naive use of \eqref{frep} is easily seen to fail for any $S$ through the observation that the $L^1$ estimate \begin{equation}\label{L1bad}\int_{T_uS}|W_S(g,g)(u,v)|\mathrm{d}v\lesssim |g(u)|^2,
\end{equation} fails, despite $W_S(g,g)$ satisfying the marginal property
\begin{equation}\label{marge}\int_{T_u S}W_S(g,g)(u,v)\mathrm{d}v=|g(u)|^2
\end{equation} (possibly under some additional minor regularity assumption on $S$); see Section \ref{app} for details, along with the sense in which such pointwise identities hold.
Of course if $X_Sw(u,v)$ is \textit{independent of $v$}, then the failure of \eqref{L1bad} is of no consequence, and \eqref{Steinvgen} follows quickly from an application of Fubini's theorem and \eqref{marge}. 

Our explicit phase-space representation \eqref{PSrep} requires rather little of the submanifold $S$. The main assumption is that $S$ is smooth and strictly convex in the sense that its shape operator is strictly positive definite at all points. On a technical level we also assume that its set of unit normals $N(S)$ is geodesically convex (that is, the intersection of $N(S)$ with any great circle is connected), along with a mild additional differentiability hypothesis (see Remark \ref{Remark:differentiability}), which we expect to be automatic from the smoothness of $S$.

For the purposes of our phase-space approach to the Stein \eqref{Steinvgen} and Mizohata--Takeuchi inequalities \eqref{MTvgen}, it will be convenient to restrict further to compact graphs.
The assumption that $S$ is a graph is a very mild assumption as the Stein and Mizohata--Takeuchi inequalities (and their variants) behave well under partitioning a manifold $S$ into boundedly many pieces. This allows us to extend our results a posteriori to closed manifolds such as the sphere, for example. With this in mind we make the additional (technical) assumption that 
\begin{equation}\label{cone of normals} 
N(u)\cdot N(u')\geq \frac{1}{2}\;\;\mbox{ for all }u,u'\in S,
\end{equation}
meaning that the normals to $S$ lie in a cone of some fixed aperture. 

As indicated in Remark
\ref{Remark:no curvature}, it is not anticipated that the discussed Stein and Mizohata--Takeuchi inequalities have a quantitative dependence on any lower bound on the curvature of $S$, and our results in this paper reflect this. Identifying this feature is one of the reasons why we have insisted on making our analysis as geometric (or parametrisation-free) as possible. Curiously, while our bounds do not depend on the curvature of $S$ in absolute terms, as we shall see, certain dilation-invariant curvature functionals naturally emerge. For example, for curves in the plane our Stein-type inequality may be controlled by the quantity
\begin{equation}\label{rightthing}
\Lambda(S):=\sup_{u,u'\in S}\left(\frac{|u'-u''|K(u)}{|N(u')\wedge N(u'')|}\right)^{1/2},
\end{equation}
where $K(u)$ denotes the Gaussian curvature of $S$ at the point $u$, and $u''$ is a certain point on $S$ constructed geometrically from points $u,u'\in S$ (we 
refer to Section \ref{Sect:general submanifolds} for details).
However, in this paper we shall formulate our main results in terms of a relatively simple curvature functional related to the quasi-conformality of the shape operator of $S$. This has the advantage of being effective in both the Stein and Mizohata--Takeuchi settings, and in all dimensions. To describe this it is helpful to again begin with the case $n=2$, where we shall say that a strictly convex planar curve $S$ has \textit{bounded curvature quotient} if there exists a finite constant $c$ such that
\begin{equation}\label{fcq}
K(u)\leq cK(u')
\end{equation}
for all $u, u'\in S$. Let us denote by $Q(S)$ the least such $c$. We extend this to higher dimensions by defining $Q(S)$ to be the maximum ratio of the principal curvatures of $S$, namely the smallest constant $c$ such that
\begin{equation}\label{fcqn}
\lambda_j(u)\leq c\lambda_{k}(u')
\end{equation}
for all $u, u'\in S$ and $1\leq j,k\leq n-1$, where $\lambda_j(u)$ denotes the $j$th principal curvature of $S$ at the point $u$. Evidently $Q(kS)=Q(S)$ for all isotropic dilates $kS$ of $S$ -- a natural property in this setting as we have indicated in Remark \ref{Remark:no curvature}.
\begin{remark}[Relation to shape quasi-conformality]
    The finiteness of $Q(S)$ may be interpreted as a certain rather strong \textit{quasi-conformality} condition on the shape operator $\mathrm{d}N$ of $S$. Indeed it quickly implies that the shape operator is $Q(S)$-quasi-conformal, that is $$\|\mathrm{d}N_u\|^{n-1}\leq Q(S)^{n-2}K(u)\;\;\mbox{ for all }u\in S;$$ see for example \cite{Alf} for a treatment of quasi-conformal maps.  This simply follows from the fact that the principal curvatures of $S$ are the eigenvalues of the shape operator. Arguing very similarly we see that the finiteness of $Q(S)$ also implies the ``long range" quasi-conformality condition 
    \begin{equation}\label{longrange}
    \|\mathrm{d}N_u\|^{n-1}\leq Q(S)^{n-1}K(u')\;\;\mbox{ for all }u, u'\in S,
    \end{equation}
    which has the advantage of having content also when $n=2$, where it reduces to \eqref{fcq}. This latter condition is actually \textit{equivalent} to $S$ having bounded curvature quotient even in higher dimensions, since \eqref{longrange} $\implies$ \eqref{fcqn} with $c=Q(S)^{n-1}$. 
\end{remark}

Our main theorems are the following Sobolev variants of the Stein and Mizohata--Takeuchi inequalities (stated somewhat informally for the sake of exposition -- see the forthcoming Theorems \ref{Theorem:SobStein} and \ref{theorem:SMTgen} for clarification):
\begin{theorem}[Sobolev--Stein inequality]\label{Theorem:shortSobStein} Suppose that $S$ is a smooth strictly convex surface with curvature quotient $Q(S)$, and $s<\frac{n-1}{2}$.  Then there is a dimensional constant $c$ such that
    \begin{equation}\label{SobSt}
    \int_{\mathbb{R}^n}|\widehat{g\mathrm{d}\sigma}(x)|^2w(x)\mathrm{d}x\leq cQ(S)^{\frac{5n-8}{4}}\int_S I_{S,2s}(|g|^2,|g|^2)(u)^{1/2}\|X_Sw(u,\cdot)\|_{\dot{H}^s(T_u S)}\mathrm{d}\sigma(u),
    \end{equation}
    where 
   $I_{S,s}$  is a certain bilinear fractional integral on $S$ of order $s$, and $\dot{H}^s(T_u S)$ denotes the usual homogeneous $L^2$ Sobolev space on the tangent space $T_u S$.
    \end{theorem}
\begin{theorem}[Sobolev--Mizohata--Takeuchi inequality]\label{theorem:shortSMTgen}
Suppose that $S$ is a smooth strictly convex surface with curvature quotient $Q(S)$, and 
$s<\frac{n-1}{2}$. Then there is a constant $c$, depending on at most $n$, $s$, and the diameter of $S$, such that 
\begin{equation}\label{SMTgenprelim}
\int_{\mathbb{R}^n}|\widehat{g\mathrm{d}\sigma}(x)|^2w(x)\mathrm{d}x\leq cQ(S)^{\frac{9n-12}{4}}\|g\|_{L^2(S)}^2\sup_{u\in S}\|X_Sw(u,\cdot)\|_{\dot{H}^s(T_u S)}.
\end{equation}
\end{theorem}
\begin{remark}
   While the constant in \eqref{SobSt} does not depend on the Sobolev index $s$, the restriction $s<\frac{n-1}{2}$ is imposed in order to ensure that the kernel of the fractional integral operator $I_{S,2s}$ is a locally integrable function; see the statement of Theorem \ref{Theorem:SobStein} for clarification. We refer to Remark \ref{stren} for the optimality of the threshold $\frac{n-1}{2}$.
\end{remark}
\begin{remark}[Improved constants]
It is not expected that the particular powers of $Q(S)$ featuring in Theorems \ref{Theorem:shortSobStein} and \ref{theorem:shortSMTgen} are best-possible, at least in dimensions $n>2$. Moreover, and as we have already indicated,
the curvature quotient $Q(S)$ does not capture all of the relevant geometry of the surface $S$. For example, in the relatively simple two-dimensional setting our arguments reveal that the power of $Q(S)$ in Theorem \ref{Theorem:shortSobStein} may be replaced by the smaller quantity $\Lambda(S)$ in \eqref{rightthing}.
It is straightforward to see that $\Lambda(S)$ may be finite when $S$ has a point of vanishing curvature, such as in the case of the quartic curve $S=\{(t,t^4):|t|\leq 1\}$.
We refer to Section \ref{what is needed really} for more on this.
\end{remark}
\begin{remark}[Permissibility of signed weights]\label{Rem:sig}
    Our proofs of Theorems \ref{Theorem:shortSobStein} and \ref{theorem:shortSMTgen}  reveal that they continue to hold for \textit{signed weights} $w$. This marks an essential difference between these theorems and the original Stein and Mizohata--Takeuchi problems.
\end{remark}
\begin{remark}[The strength of Theorem \ref{Theorem:shortSobStein}]\label{stren}
As we clarify in Section \ref{Sect:general submanifolds}, Theorems \ref{Theorem:shortSobStein} and \ref{theorem:shortSMTgen} (when specialised to non-negative weights $w$) are easily seen to be formally weaker than the global Stein and Mizohata--Takeuchi inequalities \eqref{Steinvgen} and \eqref{MTvgen} respectively (as we have commented in Remark \ref{Rem:Cairo}, the latter were recently shown to fail as stated in \cite{Cairo}). This follows via a standard Sobolev embedding, and as may be expected, the range $s<\frac{n-1}{2}$ is best-possible in this respect. Despite its weakness relative to the Stein inequality, the Sobolev--Stein inequality \eqref{SobSt} continues to be effective in transferring estimates for the X-ray transform to Fourier extension estimates, particularly in two dimensions. 
To see this let $\theta\in\mathbb{R}$ and write
$$
\|(-\Delta)^{\frac{\theta}{2}}|Eg|^2\|_2^2=\int_{\mathbb{R}^n}|Eg|^2w,
$$
where $w=(-\Delta)^\theta|Eg|^2$. By Theorem \ref{Theorem:shortSobStein} (noting Remark \ref{Rem:sig}) and the Cauchy--Schwarz inequality,
$$
\|(-\Delta)^{\frac{\theta}{2}}|Eg|^2\|_2^2\lesssim\|I_{S,2s}(|g|^2,|g|^2)\|_{L^1(S)}^{\frac{1}{2}}\|(-\Delta)^{\frac{s}{2}}X_S((-\Delta)^{\theta}|Eg|^2)\|_{L^2(TS)}
$$
whenever $s<\frac{n-1}{2}$.
By our forthcoming bounds on $I_{S,s}$ (see Section \ref{Sect:KS}, and in particular \eqref{quickL1}),
\begin{equation}\label{L1bound}
\|I_{S,2s}(|g|^2,|g|^2)\|_{L^1(S)}^{\frac{1}{2}}\lesssim \|g\|_4^2.
\end{equation}
Next, since $S$ is strictly convex its Gauss map is injective, and hence by a change of variables followed by the isometric property of the X-ray transform,
$$\|K(u)^{\frac{1}{2}}(-\Delta_v)^{\frac{1}{4}}X_Sw\|_{L^2(TS)}\leq\|(-\Delta_v)^{\frac{1}{4}}Xw\|_2=c_n\|w\|_{L^2(\mathbb{R}^n)}.$$ Therefore, provided $S$ has everywhere nonvanishing Gaussian curvature it follows that
\begin{eqnarray*}
\begin{aligned}
\|(-\Delta)^{\frac{s}{2}}X_S((-\Delta)^{\theta}|Eg|^2)\|_{L^2(TS)}&=\|(-\Delta)^{\frac{1}{4}}X_S((-\Delta)^{\theta+\frac{s}{2}-\frac{1}{4}}|Eg|^2)\|_{L^2(TS)}\\&\lesssim \|K(u)^{1/2}(-\Delta)^{1/4}X_S((-\Delta)^{\theta+\frac{s}{2}-\frac{1}{4}}|Eg|^2)\|_{L^2(TS)}\\&\lesssim \|(-\Delta)^{\theta+\frac{s}{2}-\frac{1}{4}}|Eg|^2\|_2.
\end{aligned}
\end{eqnarray*}
Hence
$$
\|(-\Delta)^{\frac{\theta}{2}}|Eg|^2\|_2^2\lesssim\|g\|_4^2\|(-\Delta)^{\theta+\frac{s}{2}-\frac{1}{4}}|Eg|^2\|_2
$$
whenever $s<\frac{n-1}{2}$.
Setting $\frac{\theta}{2}=\theta+\frac{s}{2}-\frac{1}{4}$, or equivalently $\theta=\frac{1}{2}-s$,
it follows that
$$
\|(-\Delta)^{\frac{\theta}{2}}|Eg|^2\|_2\lesssim \|g\|_4^2
$$
whenever $\theta>1-\frac{n}{2}$. This Sobolev-extension estimate is reminiscent of the well-known Strichartz inequalities of Ozawa and Tsutsumi \cite{OT}; see \cite{BBJP} for some further contextual discussion. In particular, when $n=2$ this implies the classical restriction theorem for smooth compact planar curves of nonvanishing curvature, since the missing case $\theta=0$ is the missing (endpoint) $L^4$ estimate in that setting.
We note that curvature only plays a role in the X-ray estimate, which is structurally consistent with Stein's inequality \eqref{Steinvgen}. This implication via \eqref{SobSt} should be compared with the passage from the Kakeya maximal conjecture to the restriction conjecture implied by Stein's inequality \eqref{Steinvgen} outlined in Remark \ref{Stein philosophy}. Some related arguments in the setting of the paraboloid may be found in \cite{PV, VegaEsc, BV}.
\end{remark}
\begin{remark}[The strength of Theorem \ref{theorem:shortSMTgen}]\label{strenMT}
    The proximity of \eqref{SMTgenprelim} to \eqref{MTvgen} varies depending on the nature of the weight $w$, and evidently this relates to the tightness of the Sobolev embedding referred to in Remark \ref{stren}. For example, in the case of the sphere (or suitable portions of it -- see Theorem \ref{theorem:SMTsphere}) and for weights of the form $w(x)=\varphi(x/R)$ where $\varphi$ is a smooth bump function and $R\gg 1$, Theorem \ref{theorem:shortSMTgen} quickly leads to the inequality
    $$\frac{1}{R}\int_{B(0,R)}|\widehat{g\mathrm{d}\sigma}(x)|^2\mathrm{d}x\leq C_\varepsilon R^{\varepsilon}\|g\|_2^2$$
    for all $\varepsilon>0$ and $R\gg 1$. Up to the $\varepsilon$-loss this coincides with \eqref{MTvgen}, and is the well-known Agmon--H\"ormander inequality \cite{AH}. For weights $w$ that lack regularity one should expect the Sobolev embedding referred to in Remark \ref{stren} to be weak, and thus \eqref{SMTgenprelim} to be considerably weaker than \eqref{MTvgen}. Examples of such weights seem likely to include those that are known to be ``critical" for \eqref{MTvgen} in the sense that they have large mass globally, but small mass on any line,  such as the combinatorially-constructed weights of Cairo \cite{Cairo}, or the random weights of Carbery \cite{tube} and Mulherkar \cite{Sidd}; see also \cite{Guthtalk}.
\end{remark}
\begin{remark}
    While the curvature quotient $Q(S)$ is invariant under isotropic dilations of $S$, our Sobolev--Mizohata--Takeuchi theorem (Theorem \ref{theorem:shortSMTgen}) is not. This stems from the fact that necessarily $s$ is strictly less than $\tfrac{n-1}{2}$ for the implicit constant to be finite, and manifests itself in the dependence on the diameter of $S$ in the statement of Theorem \ref{theorem:shortSMTgen}. That said, it does provide a bound that is independent of any lower bound on the curvature of $S$.
\end{remark}

\begin{remark}[Relation to the wavepacket approach]
The representation \eqref{PSrep} may be viewed as a certain ``scale-free" (and ``quadratic") version of the \textit{wavepacket decomposition} that has proved so effective in Fourier restriction theory. There an extension operator is expressed as a superposition of wavepackets adapted to tubes in $\mathbb{R}^n$, with the tubes corresponding to a discrete set of points in the tangent bundle of $S$. The distinction arises from a use of a conventional windowed Fourier transform (a linear operator) in the wavepacket decomposition, rather than a Wigner distribution -- the latter being a form of windowed Fourier transform where the window is the input function $g$ itself (a quadratic operator). We refer to \cite{CIW} and the references there for progress on the Stein and Mizohata--Takeuchi problems based on wavepacket analysis.
\end{remark}
\noindent
\textbf{Structure of the paper.} In Section \ref{Sect:para} we consider the case when $S$ is the paraboloid, motivating our perspective and results in classical quantum mechanical terms that date back to Wigner's original work. 
In Section \ref{Sect:Helm} we prove Theorems \ref{Theorem:shortSobStein} and \ref{theorem:shortSMTgen} when $S$ is the sphere, interpreting our perspective from the point of view of optical field theory. 
In Section \ref{Sect:general submanifolds} we turn to the much more involved geometric analysis in the setting of general submanifolds, proving Theorems \ref{Theorem:shortSobStein} and \ref{theorem:shortSMTgen}, although deferring the necessary analysis of Jacobians, distances and bilinear fractional integrals to Sections \ref{Jacobians}, \ref{distanceestimates} and \ref{Sect:KS} respectively. In Section \ref{app} we establish the characteristic marginal properties of the geometric Wigner transforms via an analysis of the appropriate geometric maximal operators. In Section \ref{Sect:tomographic} we observe that the phase-space perspective presented here coincides with a certain tomographic perspective introduced in \cite{BN} when $n=2$, highlighting a tomographic method for constructing geometric Wigner distributions. In Section \ref{Sect:Flandrin} we illustrate the effectiveness of our basic methods by establishing a form of Flandrin's conjecture in the plane with an $\varepsilon$ loss. Finally, in Section \ref{Sect:tombounds} we pose some questions. 

\vspace{3mm}
\noindent
\textbf{Notation.} Throughout this paper, for nonnegative quantities $A,B$ we write $A\lesssim B$ if there exists a constant $c$ that is independent of $S$ such that $A\leq cB$. The independence of the implicit constant $c$ of various other parameters will be clear from the context. In particular, such constants will never depend on the input function $g$, nor the weight function $w$.

\vspace{3mm}
\noindent
\textbf{Acknowledgments.}
The first and fourth authors are supported by EPSRC Grant EP/W032880/1. The third author is supported by JSPS Overseas Research Fellowship and  JSPS Kakenhi grant numbers 19K03546, 19H01796 and 21K13806. The fourth author thanks the American Institute of Mathematics for supporting the \textit{AIM Fourier restriction community}, in which he was introduced to the Mizohata-Takeuchi conjecture. He also thanks the Basque Center for Applied Mathematics (BCAM) for the invitation to deliver a series of lectures on the subject of this paper and for their kind hospitality. We thank Cristina Benea, Jos\'e Ca$\tilde{\mbox{n}}$izo, Tony Carbery, Mark Dennis, Michele Ferrante, Veronique Fischer, Kerr Maxwell, S\o ren Mikkelsen, Detlef M\"uller, Camil Muscalu, Mateus Sousa, Amy Tierney and Gennady Uraltsev for a number of helpful discussions. In particular, we thank Marco Vitturi for drawing our attention to the role of the classical Wigner transform in \cite{DMV}, which served as an important source of inspiration for this work. Finally, we thank the reviewers for their many helpful comments and suggestions.

\section{The paraboloid: a quantum mechanical viewpoint}\label{Sect:para}
In the particular case when $S$ is the paraboloid, the phase space representation \eqref{PSrep} has a well-known quantum mechanical derivation going back to the original work of Wigner \cite{Wigner}. As may be expected, this involves the classical Wigner transform, and as we shall see in this section, leads to some additional insights and simplifications in our arguments. Moreover, parametrised formulations of the Stein and Mizohata--Takeuchi inequalities \eqref{Steinvgen} and \eqref{MTvgen} will emerge rather naturally from these classical considerations, permitting them some physical (or probabilistic) interpretations.

The \textit{Wigner transform} is defined (see for example \cite{folland}) for $g_1, g_2\in L^2(\mathbb{R}^d)$ by
\begin{equation}\label{Wig}
W(g_1,g_2)(x,v)=\int_{\mathbb{R}^d}g_1\left(x+\frac{y}{2}\right)\overline{g_2\left(x-\frac{y}{2}\right)}e^{-2\pi iv\cdot y}\mathrm{d}y.
\end{equation}
For a solution $u:\mathbb{R}^d\times\mathbb{R}\rightarrow\mathbb{C}$ of the Schr\"odinger equation 
$$2\pi i\frac{\partial u}{\partial t}=\Delta_{x} u$$ with initial data $u_0\in L^2(\mathbb{R}^d)$, 
it is a classical observation dating back to Wigner
\cite{Wigner} that
\begin{equation}\label{f-300625}
    f(x,v,t):=W(u(\cdot, t), u(\cdot, t))(x,v)
\end{equation}
satisfies the kinetic transport equation
$$
\frac{\partial f}{\partial t}=2v\cdot\nabla_x f
$$
from classical mechanics. Consequently $$f(x,v,t)=f_0(x+2tv,v),$$ where $f_0=W(u_0,u_0):\mathbb{R}^d\times\mathbb{R}^d\rightarrow\mathbb{R}$ is the \textit{Wigner distribution} of the initial data $u_0$. We note that the function $f$ may be reconciled with the corresponding $f$ in the forthcoming Sections \ref{Sect:Helm} and \ref{Sect:general submanifolds} using the Fourier invariance property (see 1.94 of \cite{folland})
\begin{equation}\label{ftid}
W(g_1,g_2)(x,v)=W(\widehat{g}_1,\widehat{g}_2)(-v,x).
\end{equation}
By the classical marginal property
\begin{equation}\label{classical marginal}
\int_{\mathbb{R}^d}W(g,g)(x,v)\mathrm{d}v=|g(x)|^2
\end{equation}
of the Wigner distribution we obtain the \textit{phase-space representation}
\begin{equation}\label{parakinetic}
|u(x,t)|^2=\int_{\mathbb{R}^d}f(x,v,t)\mathrm{d}v=\int_{\mathbb{R}^d}f_0(x+2tv,v)\mathrm{d}v=:\rho(f_0)(x,t).
\end{equation}
The operator $\rho$, which is referred to as a \textit{velocity averaging operator} in kinetic theory, is easily seen to be a certain (parametrised) adjoint space-time X-ray transform, indeed
$$
\rho^*(g)(x,v)=\int_{\mathbb{R}}g(x-2tv,t)\mathrm{d}t,
$$
which is of course an integral of the space-time function $g$ along the line through the point $(x,0)$ with direction $(-2v,1)$.
We caution that the parameter $v$, being a velocity, describes the direction of this line. This differs from elsewhere in this paper where $v$ is used as a translation (or position) parameter.

As we have indicated in the introduction, the above phase-space representation is particularly natural if one is interested in \textit{weighted $L^2$ norms} of $u$, since by duality
\begin{equation}\label{dmv identity}
\int_{\mathbb{R}^d\times\mathbb{R}}|u(x,t)|^2w(x,t)\mathrm{d}x\mathrm{d}t=\int_{\mathbb{R}^d\times\mathbb{R}^d}W(u_0,u_0)(x,v)\rho^*w(x,v)\mathrm{d}x\mathrm{d}v.
\end{equation}
We refer to \cite{DMV} where this identity was recently derived directly.
If the initial data $u_0$ is a \textit{Gaussian} then $W(u_0,u_0)$ is also a (real) Gaussian, and being nonnegative it follows that
\begin{eqnarray*}
\begin{aligned}
\int_{\mathbb{R}^d\times\mathbb{R}}|u(x,t)|^2w(x,t)\mathrm{d}x\mathrm{d}t&\leq\int_{\mathbb{R}^d}\left(\int_{\mathbb{R}^d}W(u_0,u_0)(x,v)\mathrm{d}x\right)\sup_x\rho^*w(x,v)\mathrm{d}v\\&=
\int_{\mathbb{R}^d}|\widehat{u}_0(v)|^2\sup_x\rho^*w(x,v)\mathrm{d}v,
\end{aligned}
\end{eqnarray*}
which in turn implies that
$$
\int_{\mathbb{R}^d\times\mathbb{R}}|u(x,t)|^2w(x,t)\mathrm{d}x\mathrm{d}t\leq\sup_{\substack{x\in\mathbb{R}^d\\ v\in\supp(\widehat{u}_0)}}\rho^* w(x,v)\;\|u_0\|_2^2.
$$
Here we have used the further marginal property
\begin{equation}\label{classical marginal2}
\int_{\mathbb{R}^d}W(g,g)(x,v)\mathrm{d}x=|\widehat{g}(v)|^2
\end{equation}
of the Wigner distribution, followed by Plancherel's theorem.
It is therefore reasonably natural to ask whether
\begin{equation}\label{Steinpara}
\int_{\mathbb{R}^d\times\mathbb{R}}|u(x,t)|^2w(x,t)\mathrm{d}x\mathrm{d}t\lesssim
\int_{\mathbb{R}^d}|\widehat{u}_0(v)|^2\sup_x\rho^*w(x,v)\mathrm{d}v,
\end{equation}
and thus
\begin{equation}\label{MTpara}
\int_{\mathbb{R}^d\times\mathbb{R}}|u(x,t)|^2w(x,t)\mathrm{d}x\mathrm{d}t\lesssim\sup_{\substack{x\in\mathbb{R}^d\\ v\in\supp(\widehat{u}_0)}}\rho^* w(x,v)\;\|u_0\|_2^2
\end{equation}
might hold for general $u_0$. 
As we clarify shortly in Remark \ref{changeofmeasure}, the inequalities \eqref{Steinpara} and \eqref{MTpara} are parabolic forms of the Stein \eqref{Steinvgen} and Mizohata--Takeuchi \eqref{MTvgen} inequalities and as such also fail (see Remark \ref{Rem:Cairo}). 
\begin{remark}\label{Remark:failureStrichartz} As indicated in Remark \ref{Rem:Cairo}, the recent counterexamples in \cite{Cairo} leave open the possibility that for $\widehat{u}_0$ supported in the unit ball (say),
\begin{equation}\label{Steinparaeps}
    \int_{|(x,t)|\leq R}|u(x,t)|^2w(x,t)\mathrm{d}x\mathrm{d}t\leq C_\varepsilon R^\varepsilon \int_{\mathbb{R}^d}|\widehat{u}_0(v)|^2\sup_x\rho^*w(x,v)\mathrm{d}v
\end{equation}
and
\begin{equation}\label{MTparaeps}
    \int_{|(x,t)|\leq R}|u(x,t)|^2w(x,t)\mathrm{d}x\mathrm{d}t\leq C_\varepsilon R^\varepsilon \sup_{\substack{x\in\mathbb{R}^d\\ v\in\supp(\widehat{u}_0)}}\rho^* w(x,v)\;\|u_0\|_2^2
\end{equation}
might hold for each $\varepsilon>0$ and all $R\gg 1$; in other words, \eqref{Steinpara} and \eqref{MTpara} under the assumption that $w$ is supported in the space-time ball $B(0,R)$, accepting a factor of $R^\varepsilon$ in the implicit constant on the right-hand sides for each $\varepsilon>0$. The requirement that $\widehat{u}_0$ is supported in some fixed compact set (the unit ball here) prevents scale-invariance considerations reducing \eqref{Steinparaeps} and \eqref{MTparaeps} to \eqref{Steinpara} and \eqref{MTpara} respectively. We remark that these inequalities are naturally referred to as \textit{Strichartz estimates}, being bounds on space-time norms. 
\end{remark}
\begin{remark}[A quasi-probabilistic interpretation]\label{Remark:quasi}
In the phase-space formulation of quantum mechanics the Wigner distribution $W(u_0,u_0)$ is interpreted as a (quasi-) probability distribution on position-momentum space for a quantum particle, and so the inequalities \eqref{Steinpara} and \eqref{MTpara} for any given weight $w$ are the assertions that \begin{equation}\label{quasiStein}\mathbb{E}_{x,v}(\rho^*w)\lesssim\mathbb{E}_{x,v}(\|\rho^*w\|_{L^\infty_x})
\end{equation} and
\begin{equation}\label{quasiMT}\mathbb{E}_{x,v}(\rho^*w)\lesssim\|\rho^*w\|_\infty
\end{equation}
respectively; we recall that these inequalities are known to fail for general $w$ unless we make some additional localisations (see Remark \ref{Remark:failureStrichartz}). Here the expectation is taken with respect to the quasi-probability density $W(u_0,u_0)$, where of course $\|u_0\|_2=1$. Note that $\mathbb{E}_{x,v}(\|\rho^*w\|_{L^\infty_x})=\mathbb{E}_{v}(\|\rho^*w\|_{L^\infty_x})$ by the marginal property \eqref{classical marginal2}, where $\mathbb{E}_{v}$ is taken with respect to the probability density $|\widehat{u}_{0}(v)|^{2}$. 
The forthcoming Theorems \ref{prop:SST}--\ref{prop:SMT compact} may be interpreted similarly. Evidently the subtleties in \eqref{quasiStein}, \eqref{quasiMT} and all of these inequalities arise from the fact that the Wigner distribution typically takes both positive and negative values. 
\end{remark}
\begin{remark}\label{changeofmeasure}
Although \eqref{Steinpara} is false in general (see Remark \ref{Remark:failureStrichartz}), for any given weight $w$ it may be seen as an instance of \eqref{Steinvgen} where $d=n-1$ and 
\begin{equation}\label{the paraboloid}
S=\mathbb{P}^d:=\{u=(u',u_{d+1})\in\mathbb{R}^d\times\mathbb{R}:u_{d+1}=|u'|^2\}
\end{equation}
is the paraboloid.
This is a consequence of a certain change-of-measure invariance property enjoyed by the general inequality \eqref{Steinvgen}: specifically, if $\mathrm{d}\tilde{\sigma}(u)=a(u)\mathrm{d}\sigma(u)$ for some density $a$ on $S$, then \eqref{Steinvgenloc} quickly implies that
\begin{equation}\label{Steininv}
\int_{\mathbb{R}^n}|\widehat{g\mathrm{d}\tilde{\sigma}}(x)|^2w(x)\mathrm{d}x\leq C\int_S|g(u)|^2\sup_{v\in T_uS}a(u)Xw(N(u),v)\mathrm{d}\tilde{\sigma}(u).
\end{equation}
Next we define the (affine surface) measure $\mathrm{d}\tilde{\sigma}$ on
$\mathbb{P}^d$
by
\begin{equation}\label{meas on para}
\int_S\Phi \mathrm{d}\tilde{\sigma}=\int_{\mathbb{R}^d}\Phi(u',|u'|^2)\mathrm{d}u',
\end{equation}
so that $a(u)=(1+4|u|^2)^{-1/2}$. With these choices, a scalar change of variables reveals that
$$
\sup_{v\in T_u S}a(u)Xw(N(u),v)=\sup_x\rho^*w(x,u').
$$
Finally, defining $g: S\rightarrow\mathbb{C}$ by $g(\cdot,|\cdot|^2)=\widehat{u}_0$, we have that $u(x,t)=\widehat{g\mathrm{d}\tilde{\sigma}}(x,t)$, from which \eqref{Steinpara} follows. 
The change-of-measure invariance property \eqref{Steininv} enjoyed by \eqref{Steinvgen} is not inherited by the corresponding Mizohata--Takeuchi inequality \eqref{MTvgen}, meaning that there is in principle a different Mizohata--Takeuchi inequality for each density $a$ -- namely
\begin{equation}\label{MTa}
\int_{\mathbb{R}^n}|\widehat{g\mathrm{d}\tilde{\sigma}}(x)|^2w(x)\mathrm{d}x\leq C\sup_{(u,v)\in TS}a(u)Xw(N(u),v)\|g\|^2_{L^2(\mathrm{d}\tilde{\sigma})},
\end{equation}
where again, the supremum is restricted to $u\in\supp(g)$. It is straightforward to verify that \eqref{MTpara} coincides with \eqref{MTa} with the above choice of density $a$ on the paraboloid. Similar change-of-measure arguments relate the paraboloid-carried Wigner distribution referred to in \eqref{PSrep} to the classical Wigner distribution \eqref{Wig}, reconciling \eqref{parakinetic} with \eqref{PSrep}. We clarify this in Remark \ref{examples} in Section \ref{Sect:general submanifolds}.
\end{remark}
Perhaps the most obvious difficulty in going beyond Gaussian initial data is that $W(u_0,u_0)$ is everywhere nonnegative if and only if $u_0$ is a Gaussian (this is known as Hudson's theorem, see \cite{folland} for a treatment of this and other fundamental properties of the Wigner transform), and the inequality $\|W(u_0,u_0)\|_1\lesssim\|u_0\|_2^2$ fails for general $u_0$ (see \cite{Lerner}). Of course the $L^p$ estimates that \textit{do hold} for the Wigner distribution (see \cite{LiebWigner}) yield variants of \eqref{MTpara} via H\"older's inequality, such as 
\begin{equation}\label{2MTpara}
\int_{\mathbb{R}^d\times\mathbb{R}}|u(x,t)|^2w(x,t)\mathrm{d}x\mathrm{d}t\lesssim\|\rho^* w\|_{L^2(\mathbb{R}^d\times [-1,1]^d)}\|u_0\|_2^2,
\end{equation}
as was observed in \cite{DMV} whenever $\widehat{u}_0$ is supported in the cube $[-1,1]^d$.
Here we observe that further variants arise from certain \textit{Sobolev estimates} on the Wigner transform. For example, we have the following:
\begin{theorem}\label{prop:SST}
For $s>d/2$,
\begin{equation}\label{SSpara}
\int_{\mathbb{R}^d\times\mathbb{R}}|u(x,t)|^2w(x,t)\mathrm{d}x\mathrm{d}t\leq\int_{\mathbb{R}^d}\widetilde{I}_{2s}(|\widehat{u}_0|^2,|\widehat{u}_0|^2)(v)^{1/2}\|\rho^* w(\cdot,v)\|_{H_x^s}\mathrm{d}v,
\end{equation}
where 
$$
\widetilde{I}_s(g_1,g_2)(v):=\int_{\mathbb{R}^d}\frac{g_1\left(v+\frac{\xi}{2}\right)g_2\left(v-\frac{\xi}{2}\right)}{(1+|\xi|^2)^{s/2}}\mathrm{d}\xi
$$
and $H_x^s$ denotes the usual inhomogeneous $L^2$ Sobolev space in the variable $x$.
\end{theorem}
\begin{theorem}\label{prop:SMT}
For $s>d/2$,
\begin{equation}\label{SMTpara}
\int_{\mathbb{R}^d\times\mathbb{R}}|u(x,t)|^2w(x,t)\mathrm{d}x\mathrm{d}t\lesssim\sup_{v\in\frac{1}{2}(\supp(\widehat{u}_0)+\supp(\widehat{u}_0))}\|\rho^* w(\cdot,v)\|_{H_x^s}\|u_0\|_2^2,
\end{equation}
where the implicit constant depends on at most $d$ and $s$.
\end{theorem}
\begin{remark}\label{rmk1-dec1423}
    As our arguments quickly reveal, Theorems \ref{prop:SST} and \ref{prop:SMT} require no positivity hypothesis on the weight $w$.
This point aside, Theorems \ref{prop:SST} and \ref{prop:SMT} may be viewed as substitutes for \eqref{Steinpara} and \eqref{MTpara}, which are false for general weights; see Remark \ref{Remark:failureStrichartz}. This is a consequence of the elementary Sobolev embedding $H^s(\mathbb{R}^d)\subset L^\infty(\mathbb{R}^d)$, which holds whenever $s>d/2$.  
It is natural to ask whether the stronger
\begin{equation}\label{SMTparasup}
\int_{\mathbb{R}^d\times\mathbb{R}}|u(x,t)|^2w(x,t)\mathrm{d}x\mathrm{d}t\lesssim\sup_{v\in\supp(\widehat{u}_0)}\|\rho^* w(\cdot,v)\|_{H_x^s}\|u_0\|_2^2
\end{equation}
holds, as suggested by \eqref{MTpara} for arbitrary positive weights. 
\end{remark}
\begin{proof}[Proof of Theorem \ref{prop:SST}]
By \eqref{dmv identity} and an application of the duality of $H^s$ and $H^{-s}$ we have
$$
\int_{\mathbb{R}^d\times\mathbb{R}}|u(x,t)|^2w(x,t)\mathrm{d}x\mathrm{d}t\leq\int_{\mathbb{R}^d}\|W(u_0,u_0)(\cdot, v)\|_{H_x^{-s}}\|\rho^* w(\cdot,v)\|_{H_x^s}\mathrm{d}v,
$$
and so it remains to show that
\begin{equation}\label{paranuff}
\|W(u_0,u_0)(\cdot, v)\|_{H_x^{-s}}^2=\widetilde{I}_{2s}(|\widehat{u}_0|^2,|\widehat{u}_0|^2)(v).
\end{equation}
This follows from the classical Fourier invariance property \eqref{ftid}, which implies
\begin{equation}\label{point}
\mathcal{F}_x^{-1}W(g_1,g_2)(\xi, v)=\widehat{g}_1\left(-v+\frac{\xi}{2}\right)\overline{\widehat{g}_2\left(-v-\frac{\xi}{2}\right)},
\end{equation}
where $\mathcal{F}_x$ denotes the Fourier transform in $x$. The identity \eqref{paranuff} now follows by Plancherel's theorem and the definition of the inhomogeneous Sobolev norm.
\end{proof}
\begin{proof}[Proof of Theorem \ref{prop:SMT}]
Observe first that
$\widetilde{I}_{2s}(|\widehat{u}_0|^2,|\widehat{u}_0|^2)(v)=0$ whenever $$v\not\in\frac{1}{2}(\supp(\widehat{u}_0)+\supp(\widehat{u}_0)).$$
Hence, by Theorem \ref{prop:SST},
it suffices to show that
\begin{equation}\label{110.5}
\|\widetilde{I}_s(g_1,g_2)\|_{L^{1/2}(\mathbb{R}^d)}\lesssim\|g_1\|_1\|g_2\|_1
\end{equation}
whenever $s>d$.
The operator $\widetilde{I}_s$ is a variant (with singularity only at infinity) of the bilinear fractional integral operator 
\begin{equation}\label{KSdefhom}
\displaystyle I_s(g_1,g_2)(v):=\int_{\mathbb{R}^d}\frac{g_1\left(v+\frac{\xi}{2}\right)g_2\left(v-\frac{\xi}{2}\right)}{|\xi|^{s}}\mathrm{d}\xi
\end{equation}
treated by Kenig and Stein in \cite{KS} and Grafakos and Kalton in \cite{GK} (see also \cite{Graf} for estimates above $L^1$), and the bound \eqref{110.5} follows a brief inspection of their arguments. For similar arguments, see also Section \ref{Sect:KS}.
\end{proof}
Theorems \ref{prop:SST} and \ref{prop:SMT} cease to be natural if the initial datum $u_0$ has compact Fourier support, as they involve inhomogeneous Sobolev spaces, which respond to high frequencies of $u_0$ only. The appropriate substitutes are the following, which align with our main Theorems \ref{Theorem:shortSobStein} and \ref{theorem:shortSMTgen}:
\begin{theorem}[Parabolic Sobolev--Stein]\label{prop:SST compact}
For $s<d/2$,
\begin{equation}\label{SSparadot}
\int_{\mathbb{R}^d\times\mathbb{R}}|u(x,t)|^2w(x,t)\mathrm{d}x\mathrm{d}t\leq\int_{\mathbb{R}^d}I_{2s}(|\widehat{u}_0|^2,|\widehat{u}_0|^2)(v)^{1/2}\|\rho^* w(\cdot,v)\|_{\dot{H}_x^s}\mathrm{d}v,
\end{equation}
where 
$I_s(g_1,g_2)$  is given by \eqref{KSdefhom}
and $\dot{H}_x^s$ denotes the usual homogeneous $L^2$ Sobolev space in the variable $x$.
\end{theorem}
\begin{theorem}[Parabolic Sobolev--Mizohata--Takeuchi]\label{prop:SMT compact}
For $s<d/2$,
\begin{equation}\label{SMTlocpara}
\int_{\mathbb{R}^d\times\mathbb{R}}|u(x,t)|^2w(x,t)\mathrm{d}x\mathrm{d}t\lesssim\sup_{v\in\frac{1}{2}(\supp(\widehat{u}_0)+\supp(\widehat{u}_0))}\|\rho^* w(\cdot,v)\|_{\dot{H}_x^s}\|u_0\|_2^2
\end{equation}
whenever $\supp(\widehat{u}_0)\subseteq B(0;1)$. The implicit constant depends on at most $d$ and $s$.
\end{theorem}
\begin{remark}
Theorems \ref{prop:SST compact} and \ref{prop:SMT compact} also permit signed weights.
Restricting to positive weights, Theorems \ref{prop:SST compact} and \ref{prop:SMT compact} are also easily seen to be respectively weaker than \eqref{Steinpara} and \eqref{MTpara} via a Sobolev embedding. Specifically, by the support hypothesis on $\widehat{u}_0$ we may find a spatial bump function $\Phi$ such that
$$
\int_{\mathbb{R}^d\times\mathbb{R}}|u(x,t)|^2w(x,t)\mathrm{d}x\mathrm{d}t\leq \int_{\mathbb{R}^d\times\mathbb{R}}|u(x,t)|^2\Phi*w(x,t)\mathrm{d}x\mathrm{d}t,
$$
and so it suffices to observe that for any $v\in\mathbb{R}^d$,
\begin{eqnarray*}
\begin{aligned}
\|\rho^*(\Phi*w)(\cdot,v)\|_{\infty}\lesssim \|\rho^* w(\cdot,v)\|_{\dot{H}_x^s}
\end{aligned}
\end{eqnarray*}
whenever $s<d/2$. This follows by Plancherel's identity and the Cauchy--Schwarz inequality.
\end{remark}
The proofs of Theorems \ref{prop:SST compact} and \ref{prop:SMT compact} are very similar to those of Theorems \ref{prop:SST} and \ref{prop:SMT} above, the essential difference being the use of homogeneous rather than inhomogeneous Sobolev norms, and matters are reduced to
an $L^1\times L^1\rightarrow L^{1/2}$ bound on the bilinear operator
$$
T(g_1,g_2)(v):=\int_{B(0;1)}\frac{g_1\left(v+\frac{\xi}{2}\right)g_2\left(v-\frac{\xi}{2}\right)}{|\xi|^{s}}\mathrm{d}\xi.
$$
This is a local form of the bilinear fractional integral operator $I_s$ defined in \eqref{KSdefhom}, and again the required bound follows a brief inspection of the arguments in \cite{KS}.
\section{The sphere: an optical viewpoint}\label{Sect:Helm}
The extension operator for the sphere
$$\widehat{g\mathrm{d}\sigma}(x):=\int_{\mathbb{S}^{n-1}}e^{-2\pi ix\cdot\omega}g(\omega)\mathrm{d}\sigma(\omega)$$
is of central importance in optics, providing a description of a unit-wavelength (or monochromatic) optical wave field as a superposition of plane waves; note that $\widehat{g\mathrm{d}\sigma}$ solves the Helmholtz equation
$
\Delta u +u=0$ on $\mathbb{R}^n$.
Of particular physical significance is $|\widehat{g\mathrm{d}\sigma}|^2$, sometimes referred to as the \textit{local intensity} of the field; see for example \cite{Al}. 
The Stein and Mizohata--Takeuchi inequalities \eqref{Steinvgen} and \eqref{MTvgen}, when specialised to the sphere $S=\mathbb{S}^{n-1}$, become statements about this intensity, namely
\begin{equation}\label{Steinsphere}
\int_{\mathbb{R}^n}|\widehat{g\mathrm{d}\sigma}(x)|^2w(x)\mathrm{d}x\leq C\int_{\mathbb{S}^{n-1}}|g(\omega)|^2\sup_{v\in\langle\omega\rangle^\perp}Xw(\omega,v)\mathrm{d}\sigma(\omega),
\end{equation}
and
\begin{equation}\label{MTsphere}
\int_{\mathbb{R}^n}|\widehat{g\mathrm{d}\sigma}(x)|^2w(x)\mathrm{d}x\leq C\sup_{\omega\in\supp(g)}\|Xw(\omega,\cdot)\|_{L^\infty(\langle\omega\rangle^\perp)}\|g\|_{L^2(\mathbb{S}^{n-1})}^2
\end{equation}
respectively. These conjectural inequalities are well-known for radial weights, as discussed in the introduction, although we recall that for general weights they should carry a further localisation hypothesis following the counterexamples of Cairo \cite{Cairo}; see Remark \ref{Rem:Cairo} for clarification. Both \eqref{Steinsphere} and \eqref{MTsphere} capture the expectation that the intensity $|\widehat{g\mathrm{d}\sigma}|^2$ concentrates on rays (lines), and as such connect physical optics to geometric optics. 
A good illustration of this is found in the high-frequency limiting identity
\begin{equation}\label{ah}
\limsup_{R\rightarrow\infty}R^{n-1}\int_{\mathbb{R}^{n}}|\widehat{g\mathrm{d}\sigma}(Rx)|^2w(x)\mathrm{d}x=\frac{1}{(2\pi)^{n+1}}\int_{\mathbb{S}^{n-1}}|g(\xi)|^2\left(\int_{\mathbb{R}}w(t\xi)\mathrm{d}t\right)\mathrm{d}\sigma(\xi),
\end{equation}
established (for compactly supported $w$) by Agmon and H\"ormander in \cite{AH}; see \cite{BRV}.
Accordingly \eqref{Steinsphere} and \eqref{MTsphere} call for an optical (or spherical) analogue of the quantum-mechanical (or parabolic) phase-space perspective from Section \ref{Sect:para}. Fortunately such a perspective is well-known in modern optics (see \cite{Al}) and involves the \textit{spherical Wigner transform} that we define next.
For $g_1,g_2\in L^2(\mathbb{S}^{n-1})$ let
\begin{equation}\label{SpWig}
W_{\mathbb{S}^{n-1}}(g_1,g_2)(\omega, v)=\int_{\mathbb{S}^{n-1}}g_1(\omega')\overline{g_2(R_\omega \omega')}e^{-2\pi iv\cdot(\omega'-R_\omega\omega')}J(\omega,\omega')\mathrm{d}\sigma(\omega').
\end{equation}
Here $\omega\in\mathbb{S}^{n-1}$, $v\in\langle\omega\rangle^\perp$, and for a point $\omega'\in\mathbb{S}^{n-1}$, the point $R_\omega\omega'$ is defined to be the unique $\omega''\in\mathbb{S}^{n-1}$ for which $\omega$ is the geodesic midpoint of $\omega'$ and $\omega''$; that is 
\begin{equation}\label{R-omega-sphere}
R_\omega \omega'=2(\omega\cdot\omega')\omega-\omega'.
\end{equation}
The function $J(\omega,\omega'):=2^{n-2}|\omega\cdot\omega'|^{n-2}$ (see the forthcoming Remark \ref{examples}) is chosen so that 
$$
\int_{\mathbb{S}^{n-1}}\Phi(R_\omega\omega')J(\omega,\omega')\mathrm{d}\sigma(\omega)=\int_{\mathbb{S}^{n-1}}\Phi \mathrm{d}\sigma
$$
for each $\omega'$. This expression for $J$ may be obtained by direct computation, noting that the map $\omega\mapsto \omega'':=R_\omega\omega'$ is not a bijection, it mapping each component of $\mathbb{S}^{n-1}\backslash\langle\omega'\rangle^\perp$ bijectively to $\mathbb{S}^{n-1}\backslash\{-\omega'\}$ with 
\begin{equation}\label{SpWigJ}
\mathrm{d}\sigma(\omega'')=2^{n-1}|\omega\cdot\omega'|^{n-2}\mathrm{d}\sigma(\omega).
\end{equation}
The essential features of this construction are those described in \cite{Al}; see also \cite{KL}.

Motivated by the role of the transport equation in Section \ref{Sect:para}, for $g\in L^2(\mathbb{S}^{n-1})$ we define the auxiliary function $f:\mathbb{S}
^{n-1}\times\mathbb{R}^n\rightarrow\mathbb{R}$ (not to be confused with \eqref{f-300625}) by
$$
f(\omega,x)=\int_{\mathbb{S}^{n-1}}g(\omega')\overline{g(R_\omega \omega')}e^{-2\pi ix\cdot(\omega'-R_\omega\omega')}J(\omega,\omega')\mathrm{d}\sigma(\omega'),
$$
so that $W_{\mathbb{S}^{n-1}}(g,g)$ is the restriction of $f$ to the tangent bundle $T\mathbb{S}^{n-1}:=\{(\omega,v):\omega\in\mathbb{S}^{n-1}, v\in\langle\omega\rangle^\perp\}$.
That $f$ is real-valued follows from the fact that $R_\omega\circ R_\omega=I$ for each $\omega$. Evidently $f$ satisfies the transport equation
\begin{equation}\label{helmtrans}
\omega\cdot\nabla_xf=0,
\end{equation}
meaning that $f(\omega,x)=f(\omega,x_{\langle\omega\rangle^\perp})=W_{\mathbb{S}^{n-1}}(g,g)(\omega,x_{\langle\omega\rangle^\perp})$, where $x_{\langle\omega\rangle^\perp}$ is the orthogonal projection of $x$ onto $\langle\omega\rangle^\perp$.
The functions $f$ and $W_{\mathbb{S}^{n-1}}$ have some nice features, for example we have the marginal identity
\begin{equation}\label{helmmarg1}
\int_{\mathbb{S}^{n-1}}f(\omega,x)\mathrm{d}\sigma(\omega)=|\widehat{g\mathrm{d}\sigma}(x)|^2,
\end{equation}
by Fubini's theorem and the definition of $J$.
We note in passing that we have the additional marginal property
$$
\int_{\langle\omega\rangle^\perp}W_{\mathbb{S}^{n-1}}(g,g)(\omega,v)\mathrm{d}v=\frac{1}{2}\left(|g(\omega)|^2+|g(-\omega)|^2\right),
$$
very much as in the setting of the classical Wigner distribution. This may be obtained by fixing $\omega$ and considering the contributions to $W_{\mathbb{S}^{n-1}}(g,g)$ coming from the integrals over the hemispheres $\mathbb{S}^{n-1}_{\pm}:=\{\omega':\pm\omega\cdot\omega'>0\}$, and using the fact that the mapping $\omega'\mapsto \omega'-R_\omega\omega'$ is a bijection from each of $\mathbb{S}^{n-1}_{\pm}$ to the unit ball of $\langle\omega\rangle^\perp$; see the forthcoming proof of Theorem \ref{theorem:SSTsphere} for a similar argument.

These observations lead to the desired spherical analogue of \eqref{parakinetic}:
\begin{proposition}[Spherical phase-space representation]\label{Prop:BGNOWigner}
$$|\widehat{g\mathrm{d}\sigma}|^2=X^*W_{\mathbb{S}^{n-1}}(g,g).$$
\end{proposition}
\begin{proof}
By \eqref{helmmarg1}, \eqref{helmtrans} and Fubini's theorem,
\begin{eqnarray*}
\begin{aligned}
\int_{\mathbb{R}^n}|\widehat{g\mathrm{d}\sigma}(x)|^2w(x)\mathrm{d}x&=\int_{\mathbb{R}^n}\int_{\mathbb{S}^{n-1}}f(\omega,x)\mathrm{d}\sigma(\omega)w(x)\mathrm{d}x\\
&=\int_{\mathbb{S}^{n-1}}\int_{\langle\omega\rangle^\perp}f(\omega,x_{\langle\omega\rangle^\perp})\left(\int_{\langle\omega\rangle}w(x_{\langle\omega\rangle}+x_{\langle\omega\rangle^\perp})\mathrm{d}x_{\langle\omega\rangle}\right)\mathrm{d}x_{\langle\omega\rangle^\perp}\mathrm{d}\sigma(\omega)\\
&=\int_{\mathbb{S}^{n-1}}\int_{\langle\omega\rangle^\perp}W_{\mathbb{S}^{n-1}}(g,g)(\omega,v)Xw(\omega,v)\mathrm{d}v\mathrm{d}\sigma(\omega)\\
&=\int_{\mathbb{R}^n}X^*W_{\mathbb{S}^{n-1}}(g,g)(x)w(x)\mathrm{d}x
\end{aligned}
\end{eqnarray*}
for all test functions $w$.
\end{proof}
As we have already indicated, Proposition \ref{Prop:BGNOWigner} is well-known in some form in optics (at least in low dimensions) where it provides a representation of the local intensity of an optical field as a linear superposition of light rays -- a useful and explicit connection between physical and geometric optics; see Alonso \cite{Al}. 
Proposition \ref{Prop:BGNOWigner} may be used to prove the following spherical versions of Theorems \ref{Theorem:shortSobStein} and \ref{theorem:shortSMTgen}:
\begin{theorem}[Spherical Sobolev--Stein]\label{theorem:SSTsphere}
For $s<\frac{n-1}{2}$, there exists a dimensional constant $c$ such that
\begin{equation}\label{SSTsphere}
\int_{\mathbb{R}^n}|\widehat{g\mathrm{d}\sigma}(x)|^2w(x)\mathrm{d}x\leq c\int_{\mathbb{S}^{n-1}}I_{\mathbb{S}^{n-1},2s}(|g|^2,|g|^2)(\omega)^{1/2}\|Xw(\omega,\cdot)\|_{\dot{H}^s(\langle\omega\rangle^\perp)}\mathrm{d}\sigma(\omega),
\end{equation}
where 
$$
I_{\mathbb{S}^{n-1},s}(g_1,g_2)(\omega):=\int_{\mathbb{S}^{n-1}}\frac{g_1(\omega')g_2(R_\omega\omega')}{|\omega'-R_\omega\omega'|^{s}}|\omega\cdot\omega'|^{n-2}\mathrm{d}\sigma(\omega').
$$
\end{theorem}
\begin{remark}
    The hypothesis $s<\frac{n-1}{2}$ in the statement of Theorem \ref{theorem:SSTsphere} serves only to ensure that the kernel of the fractional integral operator $I_{\mathbb{S}^{n-1},s}$ is locally integrable, giving meaning to $I_{\mathbb{S}^{n-1},s}$. The corresponding Sobolev-Mizohata--Takeuchi theorem that follows rests on the availability of suitable bounds on this fractional integral, and so involves a constant that also depends on $s$.
\end{remark}
\begin{theorem}[Spherical Sobolev--Mizohata--Takeuchi]\label{theorem:SMTsphere}
For $s<\frac{n-1}{2}$, 
\begin{equation}\label{SMTsphere}
\int_{\mathbb{R}^n}|\widehat{g\mathrm{d}\sigma}(x)|^2w(x)\mathrm{d}x\lesssim\sup_{\omega\in\supp^*(g)}\|Xw(\omega,\cdot)\|_{\dot{H}^s(\langle\omega\rangle^\perp)}\|g\|_{L^2(\mathbb{S}^{n-1})}^2,
\end{equation}
where $\supp^*(g)$ is the set of all geodesic midpoints of pairs of points from $\supp(g)$. The implicit constant depends on at most $n$ and $s$.
\end{theorem}

\begin{remark}
Theorems \ref{theorem:SSTsphere} and \ref{theorem:SMTsphere} may be seen to follow from Theorems \ref{Theorem:shortSobStein} and \ref{theorem:shortSMTgen} respectively. This involves partitioning the sphere into suitable geodesically convex patches as alluded to in the introduction, and indeed this is how our proof below begins. This elementary step appears to require the weight $w$ to be non-negative, despite non-negativity not being a requirement of either Theorem \ref{Theorem:shortSobStein} or \ref{theorem:shortSMTgen}.
\end{remark}
\begin{proof}[Proof of Theorem \ref{theorem:SSTsphere}] By partitioning $\mathbb{S}^{n-1}$ into boundedly many (depending only on $n$) geodesically convex subsets (caps), it suffices to show \eqref{SSTsphere} under the assumption that $g$ is supported in a cap $S$ satisfying $\omega\cdot\omega'\geq\tfrac{1}{2}$ for all points $\omega,\omega'\in S$ (in line with \eqref{cone of normals}). 
By Proposition \ref{Prop:BGNOWigner} and the Cauchy--Schwarz inequality it suffices to show that
\begin{equation}\label{W_Sest}
\|W_{\mathbb{S}^{n-1}}(g,g)(\omega,\cdot)\|_{\dot{H}^{-s}(\langle\omega\rangle^\perp)}^2\lesssim I_{\mathbb{S}^{n-1},2s}(|g|^2,|g|^2)(\omega),
\end{equation}
for some implicit constant depending only on $n$. Next, for fixed $\omega\in S$ we make the change of variables $\xi=\omega'-R_\omega\omega'$, which maps $S$ bijectively to a subset $U$ of $\langle\omega\rangle^\perp$. Defining $\omega':U\rightarrow S$ by $\xi=\omega'(\xi)-R_\omega\omega'(\xi)$ we have
\begin{eqnarray*}
\begin{aligned}
W_{\mathbb{S}^{n-1}}(g,g)(\omega,v)&=\int_{S}g(\omega')\overline{g(R_\omega \omega')}e^{iv\cdot(\omega'-R_\omega\omega')}J(\omega,\omega')\mathrm{d}\sigma(\omega')\\
&=\int_{U}g(\omega'(\xi))\overline{g(R_\omega \omega'(\xi))}e^{iv\cdot\xi}\frac{J(\omega,\omega'(\xi))}{\widetilde{J}(\omega,\omega'(\xi))}\mathrm{d}\xi,
\end{aligned}
\end{eqnarray*}
where $\widetilde{J}(\omega,\omega')=2^{n-1}\omega\cdot\omega'\sim 1$ is the Jacobian of the change of variables. Hence
\begin{equation}\label{Fouriersphere}
\mathcal{F}_vW_{\mathbb{S}^{n-1}}(g,g)(\omega,\xi)=g(\omega'(\xi))\overline{g(R_\omega \omega'(\xi))}\frac{J(\omega,\omega'(\xi))}{\widetilde{J}(\omega,\omega'(\xi))}\mathbbm{1}_{U}(\xi),
\end{equation}
and so by Plancherel's theorem on $\langle\omega\rangle^\perp$,
\begin{eqnarray}\label{frost}
\begin{aligned}
\|W_{\mathbb{S}^{n-1}}(g,g)(\omega,\cdot)\|_{\dot{H}^{-s}(\langle\omega\rangle^\perp)}^2&=\int_{U}\left||\xi|^{-s}g(\omega'(\xi))\overline{g(R_\omega \omega'(\xi))}\frac{J(\omega,\omega'(\xi))}{\widetilde{J}(\omega,\omega'(\xi))}\right|^2\mathrm{d}\xi\\
&=\int_{S}|\omega'-R_\omega\omega'|^{-2s}|g(\omega')|^2|g(R_\omega \omega')|^2\frac{J(\omega,\omega')^2}{\widetilde{J}(\omega,\omega')}\mathrm{d}\sigma(\omega')\\
&\lesssim\int_{S}|\omega'-R_\omega\omega'|^{-2s}|g(\omega')|^2|g(R_\omega \omega')|^2|\omega\cdot\omega'|^{n-2}\mathrm{d}\sigma(\omega')\\
&=I_{\mathbb{S}^{n-1},2s}(|g|^2,|g|^2)(\omega),
\end{aligned}
\end{eqnarray}
The inequality \eqref{W_Sest} follows.
\end{proof}
\begin{remark}
The reader may be puzzled by the retention of the specific factor $|\omega\cdot\omega'|^{n-2}$ in the third line of \eqref{frost}, and its inclusion in the definition of $I_{\mathbb{S}^{n-1},s}$. This is significant as it is (up to a constant factor) the Jacobian $J(\omega,\omega')$, which is natural as it ensures that $I_{\mathbb{S}^{n-1},s}$ is symmetric, and enjoys the appropriate Lebesgue space bounds. This feature will become clearer in Section \ref{Sect:general submanifolds} in the context of more general submanifolds $S$.
\end{remark}
\begin{proof}[Proof of Theorem \ref{theorem:SMTsphere}]
Arguing as in the proof of Theorem \ref{theorem:SSTsphere}, it suffices to establish \eqref{SMTsphere} for $g$ supported in a single cap $S$. 
Since $I_{\mathbb{S}^{n-1},2s}(|g|^2,|g|^2)(\omega)=0$ if $\omega\not\in \supp^*(g)$, 
$$
\int_{\mathbb{R}^n}|\widehat{g\mathrm{d}\sigma}(x)|^2w(x)\mathrm{d}x\lesssim\sup_{\omega\in\supp^*(g)}\|Xw(\omega,\cdot)\|_{\dot{H}^s(\langle\omega\rangle^\perp)}\|I_{\mathbb{S}^{n-1},2s}(|g|^2,|g|^2)\|_{L^{1/2}(S)}^{1/2},
$$
by Theorem \ref{theorem:SSTsphere}. 
It therefore suffices to show that
\begin{equation*}
I_{S,s}(g_1,g_2)(\omega):=I_{\mathbb{S}^{n-1},s}(g_1\mathbbm{1}_S,g_2\mathbbm{1}_S)(\omega)=\int_{S}\frac{g_1(\omega')g_2(R_\omega\omega')}{|\omega'-R_\omega\omega'|^{s}}|\omega\cdot\omega'|^{n-2}\mathrm{d}\sigma(\omega')
\end{equation*}
is bounded from $L^1\times L^1$ into $L^{1/2}$ whenever $s<n-1$. This will be established in Section \ref{Sect:KS}, where more general surface-carried bilinear fractional integral operators are estimated.
\end{proof}

\section{General submanifolds: a geometric viewpoint}\label{Sect:general submanifolds}
As we shall see, identifying a phase-space representation of $|\widehat{g\mathrm{d}\sigma}|^2$ that is explicit enough to establish Theorems \ref{Theorem:shortSobStein} and \ref{theorem:shortSMTgen} requires some careful geometric analysis, beginning with the identification of a suitable generalised Wigner distribution (or transform). We present this for general smooth submanifolds of $\mathbb{R}^n$ that are strictly convex in the sense that their shape operators $\mathrm{d}N_u$ are positive definite at all points $u\in S$.
\subsection{Surface-carried Wigner transforms}
The general procedure for constructing a suitable Wigner transform on a submanifold of Euclidean space is again well-known in optics \cite{Al}, \cite{PA}; see for example \cite{GFH} for related intrinsic constructions in quantum physics. As is pointed out in \cite{Al}, for $n\geq 3$ matters are considerably more involved as there is some choice to be exercised.

For compactly supported function $g_1,g_2\in L^2(S)$ let
\begin{equation}\label{genWig}
W_S(g_1,g_2)(u, v)=\int_{S}g_1(u')\overline{g_2(R_u u')}e^{-2\pi iv\cdot(u'-R_uu')}J(u,u')\mathrm{d}\sigma(u').
\end{equation}
Here $u\in S$, $v\in T_uS$, and we define, for $u'\neq u$, $R_u u'$ to be the unique point $u''\in S$ with $u''\not=u'$ such that
\begin{equation}\label{collision condition 1}
(u'-u'')\cdot N(u)=0
\end{equation}
and 
\begin{equation}\label{collision condition 2}
N(u)\wedge N(u')\wedge N(u'')=0.
\end{equation}
Define $R_{u}u:=u$ for all $u\in S$. Condition \eqref{collision condition 1} stipulates that $u'-u''\in T_u S$, which as we shall see, is necessary for the phase-space representation \eqref{PSrep}; see Figure \ref{Rmap}. Condition \eqref{collision condition 2}, which stipulates that $N(u), N(u'), N(u'')$ lie on a great circle, is where we have exercised some choice. This appears to be physically significant, and is at least implicitly referred to in the optics literature; see for example \cite{Al} (Page 346) in the context of the sphere. Moreover, the appropriateness of \eqref{collision condition 2} is particularly apparent when $S$ is the paraboloid, as we clarify in the forthcoming Remark \ref{examples}.
In \eqref{genWig} the function $J(u,u')$ is the reciprocal of the Jacobian of the mapping $u\mapsto R_uu'$, so that
\begin{equation}\label{defofJ-240625}
\int_S\Phi(R_uu')J(u,u')\mathrm{d}\sigma(u)=\int_{S}\Phi \mathrm{d}\sigma
\end{equation}
for each $u'\in S$. The required bijectivity here follows from the assumed geodesic convexity of $N(S)$ referred to in Section \ref{Sect:intro}. We refer to $W_S(g_1,g_2)$ as the \textit{Wigner transform on $S$}, and $W_S(g,g)$ as the \textit{Wigner distribution on $S$}. 
As we shall see shortly, the Jacobian $J$ is a bounded function on compact subsets of $S\times S$, allowing $W_S(g_1,g_2)$ to be defined as a Lebesgue integral.
\begin{figure}[ht]
\centering
\includegraphics[width=1\linewidth]{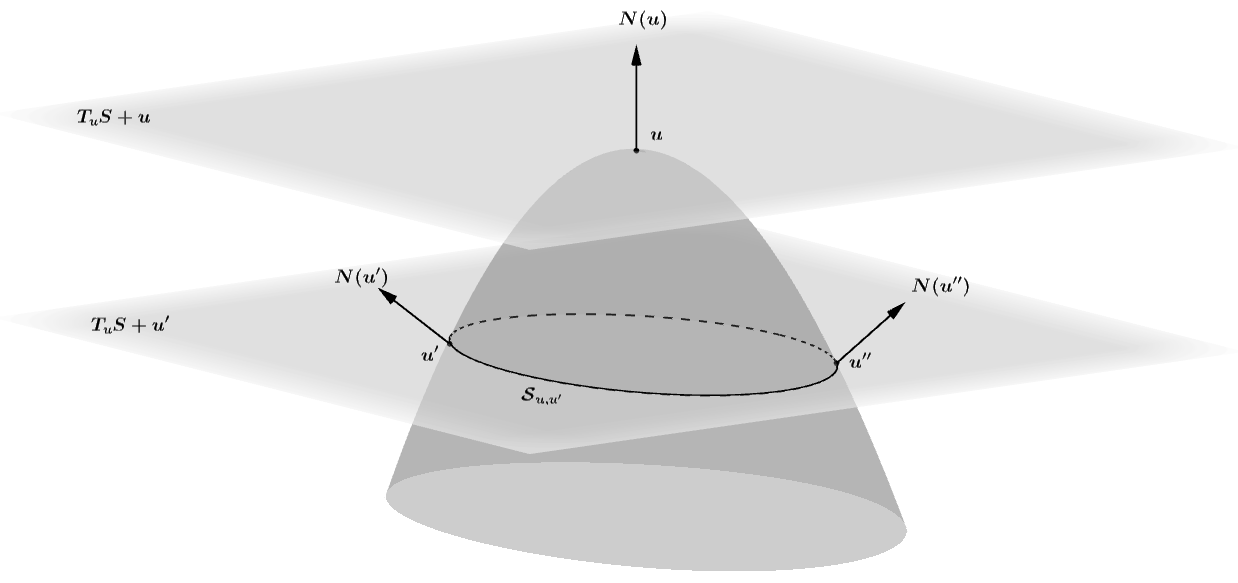}
\caption{\footnotesize A depiction of the choice of $u''$ via the conditions \eqref{collision condition 1} and \eqref{collision condition 2}.} \label{Rmap}
\end{figure}

The point $u''$ may seem rather difficult to identify at first sight, although it has a simple alternative description that is \textit{constructive}. This is shown in Figure \ref{Rmap2}, and will play an important role in our analysis.
\begin{figure}[ht]
\centering
\includegraphics[width=.65\linewidth]{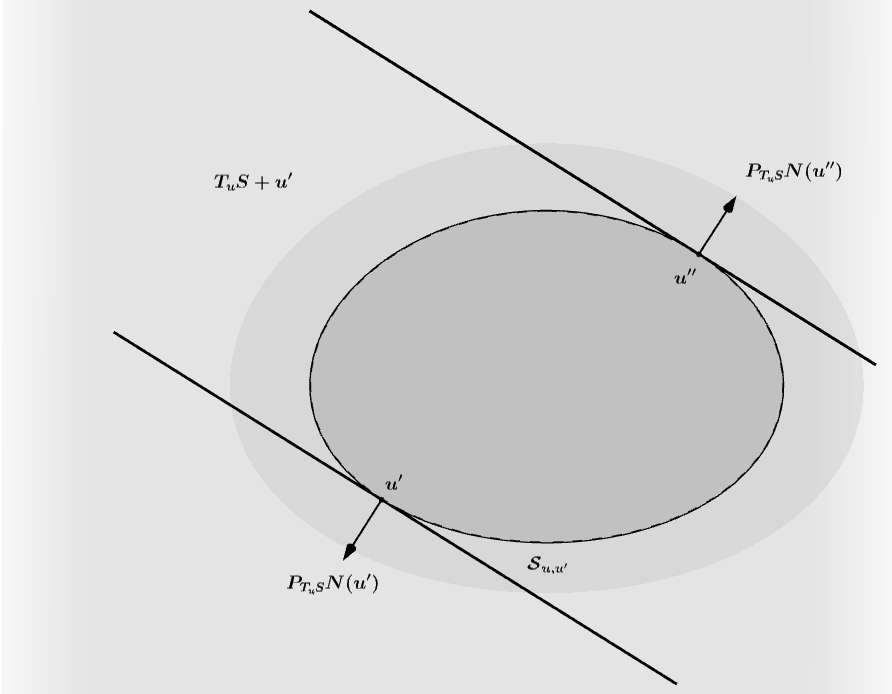}
\caption{\footnotesize The construction of $u''$ via parallel supporting hyperplanes in $T_uS+\{u'\}$.} \label{Rmap2}
\end{figure}
\begin{remark}[Existence of $u''$]\label{Rmk1-130825}
    There is a technical point that we have glossed over in the above definition of $W_S$ and Figures \ref{Rmap} and \ref{Rmap2}. For given $u,u'\in S$ our hypotheses do not guarantee the existence of such a point $u'':=R_u u'$, unless $S$ is closed (the boundary of a convex body in $\mathbb{R}^n$). One way to remedy this might be to continue $S$ to a closed submanifold, upon which $R_u u'$ may always be defined, and observe that the resulting function $W_S(g_1,g_2)$ is independent of the choice of extension since $g_2$ is supported on $S$. In any event, the integral in \eqref{genWig} should be interpreted as taken over $$\{u'\in S: (u'-u'')\cdot N(u)=0 \mbox{ and } N(u)\wedge N(u')\wedge N(u'')=0\mbox{ for some }u''\not=u'\}.$$ Naturally such domain restrictions will be apparent in our analysis of the Jacobian $J$ in Section \ref{Jacobians}.
\end{remark}
\begin{remark}[Differentiability of $u''$]\label{Remark:differentiability}
We expect that the maps $u\mapsto R_u u'$ and $u'\mapsto R_u u'$ are differentiable away from $u=u'$, and that this should follow from \eqref{collision condition 1} and \eqref{collision condition 2} by a suitable application of the implicit function theorem; see Figure \ref{Rmap2}. This smoothness is of course clear when $S$ is the sphere thanks to the explicit formula \eqref{R-omega-sphere}, and is assumed to be true of the submanifolds $S$ considered here. 
\end{remark}
\begin{remark}[Rationale for the choice of third point $u''$]\label{Remark:nondeg}
    As is pointed out in \cite{Al} and \cite{PA}, for $n\geq 3$ there are many possible ways of defining the third point $u''$ in terms of $u'$ and $u$, although for the purposes of proving Theorems \ref{Theorem:shortSobStein} and \ref{theorem:shortSMTgen} there are a number of natural requirements that significantly constrain this choice. First of all, the choice should be ``nondegenerate" in the sense that the distances $|u'-u|$ and $|u'-u''|$ should be comparable (suitably uniformly in terms of the geometry of $S$), it should be symmetric so that the resulting Wigner distribution is real-valued (and the Wigner transform is conjugate symmetric), and it should be geometrically/physically natural, so that the Jacobian $J$ may be expressed in terms of the Gauss map $N$ and its derivative $\mathrm{d}N$ (the shape operator). The forthcoming Propositions \ref{metricestimate} and \ref{jacprop} show that our choice of $u''$ has these features. As we shall see, the coplanarity condition \eqref{collision condition 2} is natural as it allows the mapping $u\mapsto R_u u'$ to be transformed to a relatively simple ``outward vector field" on the tangent space $T_{u'}S$. This involves parametrising $S$ using the Gauss map followed by stereographic projection (a composition that may also be found in the theory of minimal surfaces).
\end{remark}
It will be important for us to understand how the distances between the three points $u, u', u''$ relate to each other. This is provided by the following proposition, whose proof is deferred to Section \ref{distanceestimates}. In particular it tells us that the function $\rho(u,u'):=|u'-R_uu'|$ on $S\times S$ is a \textit{quasi-distance}, as we clarify in Section \ref{app}. 
\begin{proposition}[Distance estimates]\label{metricestimate}
     For all $u,u', u''\in S$ with $u''=R_uu'$, 
    \begin{equation}\label{normal2plane}
|u'-u''|\lesssim Q(S)^{1/2}|u-u'|
    \end{equation}
    and
    \begin{equation}\label{easier}
    |u'-u''|\gtrsim \frac{1}{Q(S)}|u-u'|.
    \end{equation}
\end{proposition}
We now turn from the metric properties to the measure-theoretic properties of the map $R_u$, and a host of explicit identities satisfied by the Wigner transform $W_S$.  

To see that $W_S$ is conjugate-symmetric, which in particular implies that the Wigner distribution $W_S(g,g)$ is real-valued, already appears to require some work. For fixed $u\in S$ observe first that if $u''=R_u u'$ then $u'=R_u u''$, and so by a change of variables,
\begin{eqnarray*}
    \begin{aligned}
        W_S(g_1,g_2)(u,v)=\int_S g_1(R_u u'')\overline{g_2(u'')}e^{-2\pi iv\cdot(R_u u''-u'')}J(u,R_u u'')\Delta(u,u'')\mathrm{d}\sigma(u''),
    \end{aligned}
\end{eqnarray*}
where $\Delta(u,u'')$ is the Jacobian of the change of variables $u'=R_u u''$. It therefore remains to show that \begin{equation*}
J(u,u')\Delta(u,u'')=J(u,u''),
\end{equation*}
recalling that $J$ was defined in \eqref{defofJ-240625}. Fortunately we have explicit formulae for the Jacobians $J$ and $\Delta$ from which this quickly follows. In the following proposition we denote by $K(u)$ the Gaussian curvature of $S$ at the point $u$, recalling that $K(u)$ is the determinant of the shape operator $\mathrm{d}N_u$. Further, we denote by  $P_{W}v$ the orthogonal projection of a vector $v\in\mathbb{R}^n$ onto a subspace $W$ of $\mathbb{R}^n$.
\begin{proposition}[Jacobian identities]\label{jacprop}
For all $u,u', u''\in S$ with $u''=R_uu'$, 
\begin{equation}\label{Jformula-prop-19apr24}
    J(u,u')=\left(\frac{|N(u')\wedge N(u'')|}{|N(u)\wedge N(u')|}\right)^{n-2}\left|\frac{\langle u''-u',N(u'')\rangle}{\langle P_{T_{u''}S}N(u),(\mathrm{d}N_{u''})^{-1}(P_{T_{u''}S}N(u))\rangle}\right|\frac{K(u)}{K(u'')},
\end{equation}
  \begin{equation}\label{deltaformula-prop-19apr24}
      \Delta(u,u')=\left(\frac{|N(u)\wedge N(u'')|}{|N(u)\wedge N(u')|}\right)^{n-1}\frac{|\langle P_{T_{u'}S}N(u),(\mathrm{d}N_{u'})^{-1}(P_{T_{u'}S}N(u))\rangle|}{|\langle P_{T_{u''}S}N(u),(\mathrm{d}N_{u''})^{-1}(P_{T_{u''}S}N(u))\rangle|}\frac{K(u')}{K(u'')}
  \end{equation}
  and
 \begin{equation}\label{switch}
     J(u,u')\Delta(u,u'')=J(u,u'').
 \end{equation}
\end{proposition}
We defer the proof of Proposition \ref{jacprop} to Section \ref{Jacobians}.
\begin{remark}[Interpreting $J$]\label{interpreting J}
The expression for $J$ in Proposition \eqref{jacprop}, while seemingly rather complicated, may be understood in somewhat simple geometric terms. In particular:
\begin{enumerate}
\item[(i)] Matters are much simpler when $n=2$, where we may write
    \begin{equation*}
    \eqalign{
    \displaystyle J(u,u')&\displaystyle =\left|\frac{\langle u''-u',P_{T_uS}N(u'')\rangle}{\langle P_{T_{u''}S}N(u),(\mathrm{d}N_{u''})^{-1}(P_{T_{u''}S}N(u))\rangle}\right|\frac{K(u)}{K(u'')} \cr
    &\displaystyle = \frac{|u''-u'|\cdot |N(u)\wedge N(u'')|}{|P_{T_{u''}S}N(u)|^{2}}K(u) \cr
    &\displaystyle = \frac{|u''-u'|}{|N(u)\wedge N(u'')|}K(u). \cr
    }
\end{equation*}
Here we have used the (two-dimensional) formula
$$\langle P_{T_{u''}S}N(u),(\mathrm{d}N_{u''})^{-1}(P_{T_{u''}S}N(u))\rangle=\frac{1}{K(u'')}|P_{T_{u''}S}N(u)|^{2},$$
along with the elementary identities $|P_{T_{u''}S}N(u)|=|P_{T_{u}S}N(u'')|=|N(u)\wedge N(u'')|$.
\item[(ii)] 
The factor
    \begin{equation}\label{harmdir}\langle P_{T_{u''}S}N(u),(\mathrm{d}N_{u''})^{-1}(P_{T_{u''}S}N(u))\rangle^{-1}
    \end{equation}
    is bounded above by 
    $\langle P_{T_{u''}S}N(u),\mathrm{d}N_{u''}(P_{T_{u''}S}N(u))\rangle$
    by the harmonic-arithmetic mean inequality.
   This bound is (up to a suitable normalisation factor) the directional curvature of $S$ at the point $u''$ in the direction $P_{T_{u''}S}N(u)$. One might therefore interpret the factor  \eqref{harmdir} as a certain ``harmonic directional curvature".
   \item[(iii)] The factor $$\frac{|N(u')\wedge N(u'')|}{|N(u)\wedge N(u')|}$$ quantifies (in relative terms) the transversality of the tangent spaces to $S$ at the points $u,u', u''$, and is therefore also a manifestation of the curvature profile of $S$; see Figure \ref{Rmap}.
   \item[(iv)] The factor $\langle u''-u',N(u'')\rangle$ is different in nature as it explicitly relates to the \textit{positions} of the points $u', u''$. It is instructive to use the fact that $u'-u''\in T_uS$ to write this as
   $$\langle u''-u',P_{T_uS}N(u'')\rangle=|N(u)\wedge N(u'')||u''-u'|\left\langle \tfrac{\displaystyle u''-u'}{\displaystyle |u''-u'|},\tfrac{\displaystyle P_{T_uS}N(u'')}{\displaystyle |P_{T_uS}N(u'')|}\right\rangle.$$
   We observe that the inner product in the final expression above
   quantifies the extent to which $u''$ is displaced from the line through $u'$ in the direction $P_{T_uS}N(u'')$; see Figure \ref{Rmap2}.
   \item[(v)] The Jacobian $J$ is scale-invariant in the sense that an isotropic scaling of $S$ leaves $J$ unchanged. This is apparent from the definition of $J$, but is also manifest in the formula \eqref{Jformula-prop-19apr24}.
   \end{enumerate}
   \end{remark}
   \begin{remark}[Examples]\label{examples}
   Proposition \ref{jacprop} is easily applied to examples.
       \begin{enumerate}
       \item[(i)] If $S=\mathbb{P}^{n-1}$, the paraboloid \eqref{the paraboloid}, then a careful calculation using Proposition \ref{jacprop} reveals that
$$
J(u,u')=2^{n-1}\left(\frac{1+4|x''|^2}{1+4|x|^2}\right)^{1/2},
$$
where we are writing $u=(x,|x|^2)$, $u'=(x',|x'|^2)$, $u'':=R_uu'=(x'',|x''|^2)$.
As should be expected from our analysis in Section \ref{Sect:para}, the parabolic Wigner distribution $W_{\mathbb{P}^{n-1}}$ may be pulled back to the classical Wigner distribution via a suitable map $\Phi:\mathbb{R}^d\times\mathbb{R}^d\rightarrow T\mathbb{P}^d$; in this case $\Phi(x,v)=((x,|x|^2),P_{T_{(x,|x|^2)}\mathbb{P}^d}(v,0))$. This uses the simple geometric fact that the coplanarity condition \eqref{collision condition 2} transforms to a \textit{colinearity} condition in parameter space.  More specifically, if for a function $g:\mathbb{R}^d\rightarrow\mathbb{C}$ we let $Lg(x,|x|^2)=(1+4|x|^2)^{-\frac{1}{2}}g(x)$, and for a function $h:T\mathbb{P}^d\rightarrow\mathbb{C}$ we let $Uh(x,v)=(1+4|x|^2)^{1/2}h(\Phi(x,v))$, then
$$UW_{\mathbb{P}^d}(Lg,Lg)=W(g,g).$$ Moreover, 
$X_{\mathbb{P}^d}^*h=\rho(Uh)$, allowing one to deduce the quantum-mechanical phase-space representation \eqref{parakinetic} from the forthcoming Proposition \ref{Prop:genBGNOWigner}.
We refer to \cite{Al} (Page 353) for a similar remark.
\item[(ii)] If $S=\mathbb{S}^{n-1}$, evidently $K\equiv 1$ and $N(\omega)=\omega$, and to be consistent with Section \ref{Sect:Helm} we use $\omega$ rather than $u$ to represent a point. We may use the explicit formula \eqref{R-omega-sphere} to write
\begin{equation*}\label{part1Jsphere}
    \frac{|N(\omega')\wedge N(\omega'')|}{|N(\omega)\wedge N(\omega')|}=\frac{(1-(\omega'\cdot\omega'')^{2})^{\frac{1}{2}}}{(1-(\omega\cdot\omega')^{2})^{\frac{1}{2}}}=\frac{|P_{\langle\omega'\rangle^{\perp}}\omega''|}{|P_{\langle\omega\rangle^{\perp}}\omega'|}=2|\omega\cdot\omega'|.
\end{equation*}
On the other hand, since $\langle \omega''-\omega',N(\omega'')\rangle=\langle \omega''-\omega',P_{\langle\omega\rangle^{\perp}}\omega''\rangle$, projecting both sides of \eqref{R-omega-sphere} to $\langle\omega\rangle^{\perp}$ yields
\begin{equation*}\label{part2Jsphere}
    \left|\frac{\langle \omega''-\omega',N(\omega'')\rangle}{\langle P_{T_{\omega''}S}N(\omega),(\mathrm{d}N_{\omega''})^{-1}(P_{T_{\omega''}S}N(\omega))\rangle}\right|=\frac{|\langle \omega''-\omega',\omega'\rangle|}{|1-(\omega\cdot\omega'')^{2}|}=\frac{|1-(\omega'\cdot\omega'')|}{|1-(\omega\cdot\omega')^{2}|}=2,
\end{equation*}
since $\omega\cdot\omega'' = \omega\cdot\omega'$ and $\omega'\cdot\omega''=2(\omega\cdot\omega')^{2}-1$.
Altogether we conclude that
$$J(\omega,\omega')=2^{n-1}|\omega\cdot\omega'|^{n-2},$$
as appears in \eqref{SpWigJ}.
    \end{enumerate}
\end{remark}

We now come to the phase space representation of $|\widehat{g\mathrm{d}\sigma}|^2$, and we begin by defining an auxiliary function $f:S\times\mathbb{R}^n\rightarrow\mathbb{R}$ by
$$
f(u,x)=\int_{S}g(u')\overline{g(R_uu')}e^{-2\pi ix\cdot(u'-R_uu')}J(u,u')\mathrm{d}\sigma(u'),
$$
so that $W_S(g,g)$ is the restriction of $f$ to the tangent bundle $TS:=\{(u,v):u\in S, v\in T_uS\}$.
As in the spherical case, we continue to have the marginal identity
\begin{equation}\label{genmarg1}
\int_{S}f(u,x)\mathrm{d}\sigma(u)=|\widehat{g\mathrm{d}\sigma}(x)|^2
\end{equation}
by Fubini's theorem and the definition of $J$.
While we shall not need to use it, it is pertinent to also note the second marginal property
\begin{equation}\label{margin2}
\int_{T_uS}W_S(g,g)(u,v)\mathrm{d}v=|g(u)|^2
\end{equation}
here (possibly subject to an additional regularity assumption on $S$) referred to in the introduction; we refer to Section \ref{app} for clarification of this, along with the sense in which it holds as a pointwise identity. 
Another key property is that $f$ satisfies the transport equation
\begin{equation}\label{gentrans}
N(u)\cdot\nabla_xf=0,
\end{equation}
meaning that $f(u,x)=W_S(g,g)(u,P_{T_uS}x)$, where $P_{T_uS}:\mathbb{R}^n\rightarrow T_uS$ is the orthogonal projection onto $T_uS$.

\begin{proposition}[General phase-space representation]\label{Prop:genBGNOWigner}
\begin{equation}\label{BGNOidWig}
|\widehat{g\mathrm{d}\sigma}|^2=X_S^*W_S(g,g)
\end{equation}
where $X_Sw(u,v):=Xw(N(u),v)$, the pullback of $Xw$ under the Gauss map
$$
TS\ni (u,v)\mapsto (N(u),v)\in T\mathbb{S}^{d-1}.
$$
\end{proposition}
We note that for a phase-space function $h:TS\rightarrow\mathbb{C}$ we have the explicit expression $$X_S^*h(x)=\int_{S}h(u,P_{T_uS}x)\mathrm{d}\sigma(u).$$
\begin{proof}[Proof of Proposition \ref{Prop:genBGNOWigner}]
By \eqref{genmarg1}, \eqref{gentrans} and Fubini's theorem,
\begin{eqnarray*}
\begin{aligned}
\int_{\mathbb{R}^n}|\widehat{g\mathrm{d}\sigma}(x)|^2w(x)\mathrm{d}x&=\int_{\mathbb{R}^n}\int_{S}f(u,x)\mathrm{d}\sigma(u)w(x)\mathrm{d}x\\
&=\int_{S}\int_{T_uS}f(u,v)\left(\int_{(T_uS)^\perp}w(v+z)\mathrm{d}z\right)\mathrm{d}v\mathrm{d}\sigma(u)\\
&=\int_S\int_{T_uS}W_S(g,g)(u,v)Xw(N(u),v)\mathrm{d}v\mathrm{d}\sigma(u)\\
&=\int_{\mathbb{R}^n}X_S^*W(g,g)(x)w(x)\mathrm{d}x
\end{aligned}
\end{eqnarray*}
for all test functions $w$.
\end{proof}
\begin{remark}[A polarised form]
The polarised form 
    $$
    \widehat{g_1\mathrm{d}\sigma}\:\overline{\widehat{g_2\mathrm{d}\sigma}}=X_S^*W_S(g_1,g_2)
    $$
    of \eqref{BGNOidWig} may be established similarly, and indeed may be deduced directly from \eqref{BGNOidWig}.
\end{remark}

\begin{remark}
There is a point of contact here with \cite{BNS}, where among other things it is shown that the classical Radon transform fails to distinguish $|\widehat{g\mathrm{d}\sigma}|^2$ from $X_S^*\nu$ for a large class of distributions $\nu$ on $TS$, provided a suitable transversality condition is satisfied. Perhaps unsurprisingly, $W_S(g,g)$ is easily seen to be an example of such a distribution.
\end{remark}
We are now ready to state or main theorems (Theorems \ref{Theorem:shortSobStein} and \ref{theorem:shortSMTgen}) in full.
\begin{theorem}[$L^2$ Sobolev--Stein inequality]\label{Theorem:SobStein}
Suppose that $S$ is a smooth strictly convex surface with curvature quotient $Q(S)$, and 
$s<\frac{n-1}{2}$. Then there is a dimensional constant $c$ such that
    \begin{equation}\label{SobStfull}
    \int_{\mathbb{R}^n}|\widehat{g\mathrm{d}\sigma}(x)|^2w(x)\mathrm{d}x\leq c Q(S)^{\frac{5n-8}{4}}\int_S I_{S,2s}(|g|^2,|g|^2)(u)^{1/2}\|X_Sw(u,\cdot)\|_{\dot{H}^s(T_u S)}\mathrm{d}\sigma(u),
    \end{equation}
    where 
    \begin{equation}\label{Sfractional}
I_{S,s}(g_1,g_2)(u):=\int_S\frac{g_1(u')g_2(R_u u')}{|u'-R_u u'|^s}J(u,u')\mathrm{d}\sigma(u').
    \end{equation}
    \end{theorem}
    \begin{remark}
        The $S$-carried fractional integral $I_{S,s}$ is natural for a number of reasons relating to the presence of the Jacobian factor $J$. In particular, it is symmetric thanks to \eqref{switch} (a property that is analogous to the conjugate symmetry of the Wigner transform $W_S$), and as we shall see in Section \ref{Sect:KS}, its Lebesgue space bounds do not depend on any lower bound on the curvature of $S$. The restriction
$s<\frac{n-1}{2}$ ensures that the kernel of $I_{S,s}$ is locally integrable.
    \end{remark}
\begin{theorem}[$L^2$ Sobolev--Mizohata--Takeuchi inequality]\label{theorem:SMTgen}
Suppose that $S$ is a smooth strictly convex surface with curvature quotient $Q(S)$, and 
$s<\frac{n-1}{2}$. Then there exists a constant $c$, depending on at most $n$, $s$, and the diameter of $S$, such that
\begin{equation}\label{SMTgen}
\int_{\mathbb{R}^n}|\widehat{g\mathrm{d}\sigma}(x)|^2w(x)\mathrm{d}x\leq cQ(S)^{\frac{9n-12}{4}}\sup_{u\in\supp^*(g)}\|X_Sw(u,\cdot)\|_{\dot{H}^s(T_u S)}\|g\|_{L^2(S)}^2,
\end{equation}
where
$
\supp^*(g):=\{u\in S: R_u u'\in \supp(g)\;\mbox{ for some }\; u'\in \supp(g)\}$.
\end{theorem}
\begin{remark}
We remark that $\supp(g)\subseteq\supp^*(g)$, and often this containment is strict. When $S$ is the sphere, $\supp^*(g)$ is the ``support midpoint set”,
consisting of all geodesic midpoints of pairs of points from the support of $g$. Hence $\supp^*(g)\subseteq \cvx\supp(g)$ in this case, where $\cvx$ forms the geodesic convex hull. More generally, $\supp^*(g)\subseteq N^{-1}\cvx(N(\supp(g)))$, so that 
$$
\int_{\mathbb{R}^n}|\widehat{g\mathrm{d}\sigma}(x)|^2w(x)\mathrm{d}x\leq cQ(S)^{\frac{9n-12}{4}}\sup_{\omega\in\cvx(N(\supp(g)))}\|Xw(\omega,\cdot)\|_{\dot{H}^s(T_u S)}\|g\|_{L^2(S)}^2.
$$
\end{remark}
\begin{remark}
    While we expect that the power of $Q(S)$ in the statement of Theorem \ref{Theorem:SobStein} is sharp when $n=2$, it seems unlikely that it is in higher dimensions. The power of $Q(S)$ in the statement of Theorem \ref{theorem:SMTgen} is of course larger still, incurring extra factors from the bounds on the bilinear fractional integrals $I_{S,s}$ in Section \ref{Sect:KS}.
\end{remark}
\subsection{Proof of the Sobolev--Stein inequality (Theorem \ref{Theorem:shortSobStein})} In this section we prove Theorem \ref{Theorem:shortSobStein}, or more specifically, Theorem \ref{Theorem:SobStein}.
We begin with an application of 
Proposition \ref{Prop:genBGNOWigner} and the Cauchy--Schwarz inequality to write
\begin{equation}\label{goodcorrelation}
\int_{\mathbb{R}^n}|\widehat{g\mathrm{d}\sigma}|^2w\leq \int_S \|W_S(g,g)(u,\cdot)\|_{\dot{H}^{-s}(T_u S)}\|X_Sw(u,\cdot)\|_{\dot{H}^s(T_u S)}\mathrm{d}\sigma(u)
\end{equation}
for any $s\in\mathbb{R}$.
In order to estimate the Sobolev norm of the Wigner distribution above we fix $u\in S$ and 
make the change of variables 
\begin{equation}\label{cov}
\xi=u'-R_u u'.
\end{equation}
Since $S$ is the graph of a strictly convex function, the map $u'\mapsto\xi$ is a bijection from $S$ to a subset $U$ of $T_uS$. To see this it suffices to establish injectivity, and hence we look to show that $u'-R_uu'\not=\tilde{u}'-R_u\tilde{u}'$ for $u'\neq \tilde{u}'$. We may assume that $u'-R_uu'$ and $\tilde{u}'-R_u\tilde{u}'$ are parallel, as otherwise the desired conclusion is immediate. Observe that $u'-\tilde{u}'\not\in T_uS$, as otherwise strict convexity of the level sets of $S$ (sections of $S$ by translates of $T_uS$) would force $u'=\tilde{u}'$ or $R_uu'=\tilde{u}'$; see Figure \ref{Rmap2}. Since $u'-\tilde{u}'\not\in T_uS$ and $S$ is the graph of a function, the level sets of $S$ through $u'$ and $\tilde{u}'$ respectively, when projected onto $T_uS$, are both enclosed by the supporting hyperplanes depicted in Figure \ref{Rmap2}; this may require interchanging the roles of $u'$ and $\tilde{u}'$, as we may. Since $u'-R_uu'$ and $\tilde{u}'-R_u\tilde{u}'$ are parallel it follows that $|\tilde{u}'-R_u\tilde{u}'|<|u'-R_uu'|$, and thus $u'-R_uu'\not=\tilde{u}'-R_u\tilde{u}'$. As a result of this bijectivity,
$$
\|W_{S}(g,g)(u,\cdot)\|_{H^{-s}(T_u S)}^2=\int_{T_u S}\left|\int_U g(u'(\xi))\overline{g(R_u u'(\xi))}|\xi|^{-s}e^{-2\pi iv\cdot\xi}J(u,u'(\xi))\frac{\mathrm{d}\xi}{\widetilde{J}(u,u'(\xi))}\right|^2\mathrm{d}v,
$$
where $\widetilde{J}(u,u')$ is the Jacobian of the map $u'\mapsto\xi$.
Hence by Plancherel's theorem on $T_u S$,
\begin{eqnarray*}
    \begin{aligned}
        \|W_{S}(g,g)(u,\cdot)\|_{H^{-s}(T_u S)}^2&=\int_{U}\left|g(u'(\xi))\overline{g(R_u u'(\xi))}|\xi|^{-s}\frac{J(u,u'(\xi))}{\widetilde{J}(u,u'(\xi))}\right|^2\mathrm{d}\xi\\
        &=\int_S\frac{|g(u')|^2|g(R_u u')|^2}{|u'-R_u u'|^{2s}}\frac{J(u,u')^2}{\widetilde{J}(u,u')}\mathrm{d}\sigma(u').
    \end{aligned}
\end{eqnarray*}
In order to complete the proof of
Theorem \ref{Theorem:SobStein} it therefore suffices to prove that 
\begin{equation}\label{clever bit} 
\frac{J(u,u')}{\widetilde{J}(u,u')}\lesssim Q(S)^{\frac{5n-8}{2}}
\end{equation}
with implicit constant depending only on the dimension. We do this in two steps.

\subsubsection*{Step 1: Bounding $\widetilde{J}(u,u')$} The goal here is to obtain a suitable lower bound for $\widetilde{J}(u,u')$.

\begin{proposition}\label{prop-Jtilde-09apr24} We have that
\begin{equation}
    \widetilde{J}(u,u')\geq (1+\Delta(u,u')^{2})^{\frac{1}{2}}
\end{equation}
    for all $u,u'\in S$.
\end{proposition}
\begin{proof} Let $u\in S$ be fixed. The Jacobian $\widetilde{J}$ of the change of variables
\begin{equation*}
    \xi(u')=u'-R_{u}u',
\end{equation*}
may be expressed as
\begin{equation}\label{firstjac}
    \widetilde{J}(u,u')=\frac{|(\mathrm{d}\xi)_{u'}(v_{1})\wedge\cdots\wedge (\mathrm{d}\xi)_{u'}(v_{n-1})|}{|v_{1}\wedge\cdots\wedge v_{n-1}|},
\end{equation}
where $v_{1},\ldots,v_{n-1}$ is a basis for $T_{u'}S$. We remark that 
$$(\mathrm{d}\xi)_{u'}(v_{1})\wedge\cdots\wedge (\mathrm{d}\xi)_{u'}(v_{n-1})\in\Lambda^{n-1}(T_{u}S)\quad\textnormal{and}\quad v_{1}\wedge\cdots\wedge v_{n-1}\in\Lambda^{n-1}(T_{u'}S),$$
and we identify the exterior algebras $\Lambda^{n-1}(T_{u'}S)$ and $\Lambda^{n-1}(T_{u}S)$ with subspaces of $\Lambda^{n-1}(\mathbb{R}^{n})$ via the natural embedding induced by the inclusions $T_{u'}S\subset\mathbb{R}^{n}$ and $T_{u}S\subset\mathbb{R}^{n}$, respectively.

It will be convenient to fix $u'$ and express \eqref{firstjac} in terms of unit velocities of trajectories along smooth curves in $S$ emanating from $u'$. In what follows $c:I\rightarrow S$ will denote the arc-length parametrisation of such a curve, where $I$ is an open interval containing $0$ such that $c(0)=u'$. If $\mathcal{C}$ denotes the set of all such mappings $c$, then evidently
\begin{equation*}
    T_{u'}S = \langle\left\{\dot{c}(0): c\in\mathcal{C}\right\}\rangle.
\end{equation*}
By the strict convexity of $S$, the $(n-1)$-dimensional spaces $T_{u'}S$ and $T_{u}S$ intersect in an $(n-2)$-dimensional subspace $\mathcal{H}$. We then pick curves $c_{1},\ldots,c_{n-2}\in\mathcal{C}$ such that
$$\mathcal{H}=\langle \dot{c}_{1}(0),\ldots,\dot{c}_{n-2}(0)\rangle,$$
and the set $\{\dot{c}_{i}(0)\}_{1\leq i\leq n-2}$ is orthonormal. To obtain an orthonormal basis for $T_{u'}S$, we simply take any other curve $c_{n-1}\in\mathcal{C}$ such that $\dot{c}_{n-1}(0)\in \mathcal{H}^{\perp}\cap T_{u'}S$. There is one more degree of freedom in choosing $c_{n-1}$, and we assume without loss of generality that $\dot{c}_{n-1}(0)\cdot N(u)\geq 0$. This gives
$$T_{u'}S=\langle \dot{c}_{1}(0),\ldots,\dot{c}_{n-2}(0),\dot{c}_{n-1}(0)\rangle.$$
Since
\begin{equation*}
    (\mathrm{d}\xi)_{u'}(\dot{c}_{i}(0))= (\xi\circ c_{i})'(0) = \dot{c}_{i}(0)- (\mathrm{d}R_{u})_{u'}(\dot{c}_{i}(0)),\quad 1\leq i\leq n-1,
\end{equation*}
then, since $|\dot{c}_{1}(0)\wedge\cdots\wedge \dot{c}_{n-1}(0)|=1$ by orthonormality of the chosen basis of $T_{u'}S$,
\begin{equation}\label{eq1-09apr24}
    \eqalign{
    \displaystyle\widetilde{J}(u,u')&\displaystyle =|(\mathrm{d}\xi)_{u'}(\dot{c}_{1}(0))\wedge\cdots\wedge (\mathrm{d}\xi)_{u'}(\dot{c}_{n-1}(0))| \cr
    &\displaystyle= |(\dot{c}_{1}(0)- (\mathrm{d}R_{u})_{u'}(\dot{c}_{1}(0)))\wedge\cdots\wedge (\dot{c}_{n-1}(0)- (\mathrm{d}R_{u})_{u'}(\dot{c}_{n-1}(0)))| \cr
    &\displaystyle= |W_{1}-W_{2}|, \cr
    }
\end{equation}
where
\begin{equation*}
    \eqalign{
    \displaystyle W_{1}&\displaystyle := (\dot{c}_{1}(0)- (\mathrm{d}R_{u})_{u'}(\dot{c}_{1}(0)))\wedge\cdots\wedge (\dot{c}_{n-2}(0)- (\mathrm{d}R_{u})_{u'}(\dot{c}_{n-2}(0)))\wedge \dot{c}_{n-1}(0),\cr
    \displaystyle W_{2}&\displaystyle := (\dot{c}_{1}(0)- (\mathrm{d}R_{u})_{u'}(\dot{c}_{1}(0)))\wedge\cdots\wedge (\dot{c}_{n-2}(0)- (\mathrm{d}R_{u})_{u'}(\dot{c}_{n-2}(0)))\wedge (\mathrm{d}R_{u})_{u'}(\dot{c}_{n-1}(0)).\cr
    }
\end{equation*}

The next claim collects a few useful facts about the action of $(\mathrm{d}R_{u})_{u'}$ on $\mathcal{H}$.

\begin{claim}\label{claim1-09apr24} The following hold:
\begin{enumerate}
    \item The subspace $\mathcal{H}=T_{u'}S\cap T_{u}S$ generated by the set of vectors $\{\dot{c}_{1}(0),\ldots,\dot{c}_{n-2}(0)\}$ is invariant under the map $(\mathrm{d}R_{u})_{u'}$. Moreover, $\left.(\mathrm{d}R_{u})_{u'}\right|_{\mathcal{H}}:\mathcal{H}\longrightarrow\mathcal{H}$ is an isomorphism. Equivalently,
    \begin{equation}\label{eq1-17apr24}
    \mathcal{H}=\langle \dot{c}_{1}(0),\ldots,\dot{c}_{n-2}(0)\rangle=\langle (\mathrm{d}R_{u})_{u'}(\dot{c}_{1}(0)),\ldots,(\mathrm{d}R_{u})_{u'}(\dot{c}_{n-2}(0))\rangle.
    \end{equation}

\item Let $M_{u,u'}:=\left.(\mathrm{d}R_{u})_{u'}\right|_{\mathcal{H}}: \mathcal{H}\longrightarrow \mathcal{H}$ denote the restriction of $(\mathrm{d}R_{u})_{u'}$ to the invariant subspace $\mathcal{H}$. Then $I-M_{u,u'}:\mathcal{H}\rightarrow\mathcal{H}$ satisfies

\begin{equation}\label{eq3-17apr24}
    \det{(I-M_{u,u'})} \geq 1.
\end{equation}
\end{enumerate}
\end{claim}
\begin{proof} Let $\omega:= N(u)$. Notice that the coplanarity condition \eqref{collision condition 2}
implies that
\begin{equation}\label{forIFT}
    v_{1}:=\frac{P_{\langle\omega\rangle^{\perp}}N(u')}{|P_{\langle\omega\rangle^{\perp}}N(u')|}=-\frac{P_{\langle\omega\rangle^{\perp}}N(u'')}{|P_{\langle\omega\rangle^{\perp}}N(u'')|}=:-v_{2}.
\end{equation}
On the other hand, $v_{1}$ and $v_{2}$ are the outward normal vectors (in $T_{u}S$) of the convex submanifold 
\begin{equation}\label{sectionsdef}
\mathcal{S}_{u,u'}:=S\cap (T_{u}S + u')
\end{equation}
at $u'$ and $u''$ respectively, hence
\begin{equation*}
    T_{u'}\mathcal{S}_{u,u'} = T_{u''}\mathcal{S}_{u,u'},
\end{equation*}
from which \eqref{eq1-17apr24} follows; see Figure \ref{Rmap2}. Observe also that on $\mathcal{S}_{u,u'}$ we have
\begin{equation}\label{footmeasuringdevice}
    R_{u}u' = \widetilde{N}^{-1}(-\widetilde{N}(u')),
\end{equation}
where $\widetilde{N}:\mathcal{S}_{u,u'}\rightarrow\mathbb{S}^{n-2}$ is the Gauss map of $\mathcal{S}_{u,u'}\subset u'+T_{u}S$. Computing derivatives, $\left.(\mathrm{d}R_{u})_{u'}\right|_{\mathcal{H}}: \mathcal{H}\longrightarrow \mathcal{H}$ satisfies
\begin{equation*}
    (\mathrm{d}R_{u})_{u'} = -\mathrm{d}\widetilde{N}^{-1}_{-\widetilde{N}(u')}\circ\mathrm{d}\widetilde{N}_{u'}.
\end{equation*}
Finally, since $\mathrm{d}\widetilde{N}^{-1}_{-\widetilde{N}(u')}$ and $\mathrm{d}\widetilde{N}_{u'}$ are positive definite (recall that our assumptions on $S$ imply positive definiteness of $\mathrm{d}N_{u}$ for all $u\in S$, hence the same holds for $\mathrm{d}\widetilde{N}_{u}$) the product 
$\mathrm{d}\widetilde{N}^{-1}_{-\widetilde{N}(u')}\circ\mathrm{d}\widetilde{N}_{u'}$ has positive eigenvalues, therefore
\begin{equation*}
    \det{(I-M_{u,u'})}=\det{(I+\mathrm{d}\widetilde{N}^{-1}_{-\widetilde{N}(u')}\circ\mathrm{d}\widetilde{N}_{u'})}\geq 1.
\end{equation*}
\end{proof}

The next claim contains three key identities involving $W_{1}$ and $W_{2}$.

\begin{claim}\label{claim1-26may24} The following identities hold:
\begin{equation*}
    \eqalign{
    \displaystyle \langle W_{1},W_{1}\rangle &\displaystyle = \det{(I-M_{u,u'})}^{2},\cr
    \displaystyle\langle W_{1},W_{2}\rangle &=\displaystyle \det{(I-M_{u,u'})}^{2}\langle (\mathrm{d}R_{u})_{u'}(\dot{c}_{n-1}(0)),\dot{c}_{n-1}(0)\rangle \cr
    \displaystyle \langle W_{2},W_{2}\rangle&\displaystyle = \Delta(u,u')^{2}\det{(M_{u,u'}^{-1}-I)}^{2}.\cr
    }
\end{equation*} 
\end{claim}

\begin{proof} Let $\textbf{0}_{1\times (n-2)}$ be the $1\times(n-2)$ zero row and let $\textbf{X}_{u,u'}$ be the $(n-2)\times (n-2)$ matrix whose $(i,j)$ entry is given by
\begin{equation*}
    (\textbf{X}_{u,u'})_{i,j}:= \langle \dot{c}_{i}(0)- (\mathrm{d}R_{u})_{u'}(\dot{c}_{i}(0)),\dot{c}_{j}(0)- (\mathrm{d}R_{u})_{u'}(\dot{c}_{j}(0))\rangle.
\end{equation*}
Observe that
\begin{equation*}
\eqalign{
    \displaystyle\langle W_{1},W_{1}\rangle&\displaystyle=\det{\begin{pmatrix}
    \textbf{X}_{u,u'} & \textbf{0}_{1\times (n-2)}^{\top} \\
    \textbf{0}_{1\times (n-2)} & 1
  \end{pmatrix}} =\displaystyle\det{(\textbf{X}_{u,u'})}=\displaystyle \det{(I-M_{u,u'})}^{2}, \cr
  }
\end{equation*}
where we used the facts that $\mathcal{H}$ is invariant under $(\mathrm{d}R_{u})_{u'}$ (as verified in Claim \ref{claim1-09apr24}) and that $\dot{c}_{n-1}(0)$ is orthogonal to $\mathcal{H}$. Now let $\textbf{Y}_{u,u'}$ be the $(n-2)\times (n-2)$ matrix whose $(i,j)$ entry is given by
\begin{equation*}
    (\textbf{Y}_{u,u'})_{i,j}:= \langle (\mathrm{d}R_{u})_{u'}^{-1}(\dot{c}_{i}(0))-\dot{c}_{i}(0),(\mathrm{d}R_{u})_{u'}^{-1}(\dot{c}_{j}(0))-\dot{c}_{j}(0)\rangle.
\end{equation*}
Analogously,
\begin{equation*}
    \eqalign{
    \displaystyle\langle W_{2},W_{2}\rangle &\displaystyle = |(\dot{c}_{1}(0)- (\mathrm{d}R_{u})_{u'}(\dot{c}_{1}(0)))\wedge\cdots\wedge (\dot{c}_{n-2}(0)- (\mathrm{d}R_{u})_{u'}(\dot{c}_{n-2}(0)))\wedge (\mathrm{d}R_{u})_{u'}(\dot{c}_{n-1}(0))|^{2} \cr
    &\displaystyle = \Delta(u,u')^{2}\left|\left(\bigwedge_{j=1}^{n-2}(\mathrm{d}R_{u})_{u'}^{-1}[\dot{c}_{j}(0)- (\mathrm{d}R_{u})_{u'}(\dot{c}_{j}(0))]\right)\wedge \dot{c}_{n-1}(0)\right|^{2}\cr
    &\displaystyle = \Delta(u,u')^{2}\left|\left(\bigwedge_{j=1}^{n-2}[(\mathrm{d}R_{u})_{u'}^{-1}-I](\dot{c}_{j}(0))\right)\wedge \dot{c}_{n-1}(0)\right|^{2}\cr
&=\Delta(u,u')^{2}\det(\textbf{Y}_{u,u'}) \cr
  &= \Delta(u,u')^{2}\det{(M_{u,u'}^{-1}-I)}^{2}.
    }
\end{equation*}
Finally,
\begin{equation*}
    \eqalign{
    \displaystyle\langle W_{1},W_{2}\rangle &\displaystyle= \det{\begin{pmatrix}
    \textbf{X}_{u,u'} & A_{(n-2)\times 1} \\
    \textbf{0}_{1\times (n-2)} & \langle (\mathrm{d}R_{u})_{u'}(\dot{c}_{n-1}(0)),\dot{c}_{n-1}(0)\rangle
  \end{pmatrix}} \cr
  &=\displaystyle\det{(I-M_{u,u'})}^{2}\langle (\mathrm{d}R_{u})_{u'}(\dot{c}_{n-1}(0)),\dot{c}_{n-1}(0)\rangle, \cr
    }
\end{equation*}
where $A_{(n-2)\times 1}$ is a $(n-2)\times 1$ column that does not feature in the final expression.
\end{proof}

Expanding $|W_{1}-W_{2}|^{2}$ using the standard scalar product on the exterior algebra $\Lambda^{n-1}(\mathbb{R}^{n})$,
\begin{eqnarray}\label{ineq1-22may24}
\begin{aligned}
|W_{1}-W_{2}|^{2} &= \langle W_{1}, W_{1}\rangle - 2\langle W_{1}, W_{2}\rangle + \langle W_{2}, W_{2}\rangle \\
    &= \det{(I-M_{u,u'})}^{2} - 2\det{(I-M_{u,u'})}^{2}\langle (\mathrm{d}R_{u})_{u'}(\dot{c}_{n-1}(0)),\dot{c}_{n-1}(0)\rangle\\& + \Delta(u,u')^{2}\det{(M_{u,u'}^{-1}-I)}^{2},
    \end{aligned}
\end{eqnarray}
thanks to Claim \ref{claim1-26may24}. We continue with the following key observation:

\begin{claim}\label{claim1-22may24} Under \eqref{cone of normals}, it holds that
\begin{equation*}
    \langle (\mathrm{d}R_{u})_{u'}(\dot{c}_{n-1}(0)),\dot{c}_{n-1}(0)\rangle < 0.
\end{equation*}
\end{claim}

\begin{proof} Recall that $\dot{c}_{n-1}(0)\cdot N(u) > 0$ by assumption, therefore differentiating at $t=0$ the identity
$$\langle R_{u}(c_{n-1}(t)) - c_{n-1}(t), N(u)\rangle =0$$
gives $\langle (\mathrm{d}R_{u})_{u'}(\dot{c}_{n-1}(0)), N(u)\rangle >0$. Next, observe that $N(u'), N(u), N(u'')$ and $\dot{c}_{n-1}(0)$ are in $\mathcal{H}^{\perp}$, the (two-dimensional) orthogonal complement of $\mathcal{H}$ in $\mathbb{R}^{n}$. Since $N(p)\cdot N(q)\geq\frac{1}{2}$ for all $p,q\in S$ by assumption, the angles $\alpha_{1}$ (between $N(u')$ and $N(u)$) and $\alpha_{2}$ (between $N(u)$ and $N(u'')$) are such that $0<\alpha_{1}+\alpha_{2}<\frac{\pi}{2}$.
Since $N(u)\in\mathcal{H}^{\perp}$, we have by the self-adjointness of the projection operator $P_{\mathcal{H}^{\perp}}$,
$$0<\langle (\mathrm{d}R_{u})_{u'}(\dot{c}_{n-1}(0)), N(u) \rangle = \langle (\mathrm{d}R_{u})_{u'}(\dot{c}_{n-1}(0)), P_{\mathcal{H}^{\perp}}N(u)\rangle = \langle P_{\mathcal{H}^{\perp}}[(\mathrm{d}R_{u})_{u'}(\dot{c}_{n-1}(0))], N(u)\rangle ,$$
which implies that $P_{\mathcal{H}^{\perp}}[(\mathrm{d}R_{u})_{u'}(\dot{c}_{n-1}(0))]$ is in the upper-half space of $\mathcal{H}^{\perp}$ (here we are assuming without loss of generality that $N(u)=e_{2}$, the second canonical vector of $\mathcal{H}^{\perp}\cong\mathbb{R}^{2}$). On the other hand, $(\mathrm{d}R_{u})_{u'}(\dot{c}_{n-1}(0))\in T_{u''}S$, hence $\langle P_{\mathcal{H}^{\perp}}[(\mathrm{d}R_{u})_{u'}(\dot{c}_{n-1}(0))], N(u'') \rangle =0$, i.e. the angle between $N(u'')$ and $P_{\mathcal{H}^{\perp}}[(\mathrm{d}R_{u})_{u'}(\dot{c}_{n-1}(0))]$ is $\frac{\pi}{2}$. Since $\theta:=\frac{\pi}{2}-(\alpha_{1}+\alpha_{2})$ is strictly positive, the angle $\gamma:=\frac{\pi}{2}+\theta$ between $P_{\mathcal{H}^{\perp}}[(\mathrm{d}R_{u})_{u'}(\dot{c}_{n-1}(0))]$ and $\dot{c}_{n-1}(0)$ is strictly larger than $\frac{\pi}{2}$ (see Figure \ref{Hperppic}), which implies that
$$\langle P_{\mathcal{H}^{\perp}}[(\mathrm{d}R_{u})_{u'}(\dot{c}_{n-1}(0))], \dot{c}_{n-1}(0)\rangle < 0.$$
Finally, again by the self-adjointness of $P_{\mathcal{H}^{\perp}}$,
\begin{equation*}
    \eqalign{
\displaystyle \langle (\mathrm{d}R_{u})_{u'}(\dot{c}_{n-1}(0)), \dot{c}_{n-1}(0)\rangle &\displaystyle= \langle (\mathrm{d}R_{u})_{u'}(\dot{c}_{n-1}(0)), P_{\mathcal{H}^{\perp}}[\dot{c}_{n-1}(0)]\rangle \cr
&\displaystyle= \langle P_{\mathcal{H}^{\perp}}[(\mathrm{d}R_{u})_{u'}(\dot{c}_{n-1}(0))], \dot{c}_{n-1}(0)\rangle \cr
&\displaystyle < 0,\cr
    }
\end{equation*}
which concludes the proof of the claim.
\begin{figure}[ht]
\centering
\includegraphics[width=.9\linewidth]{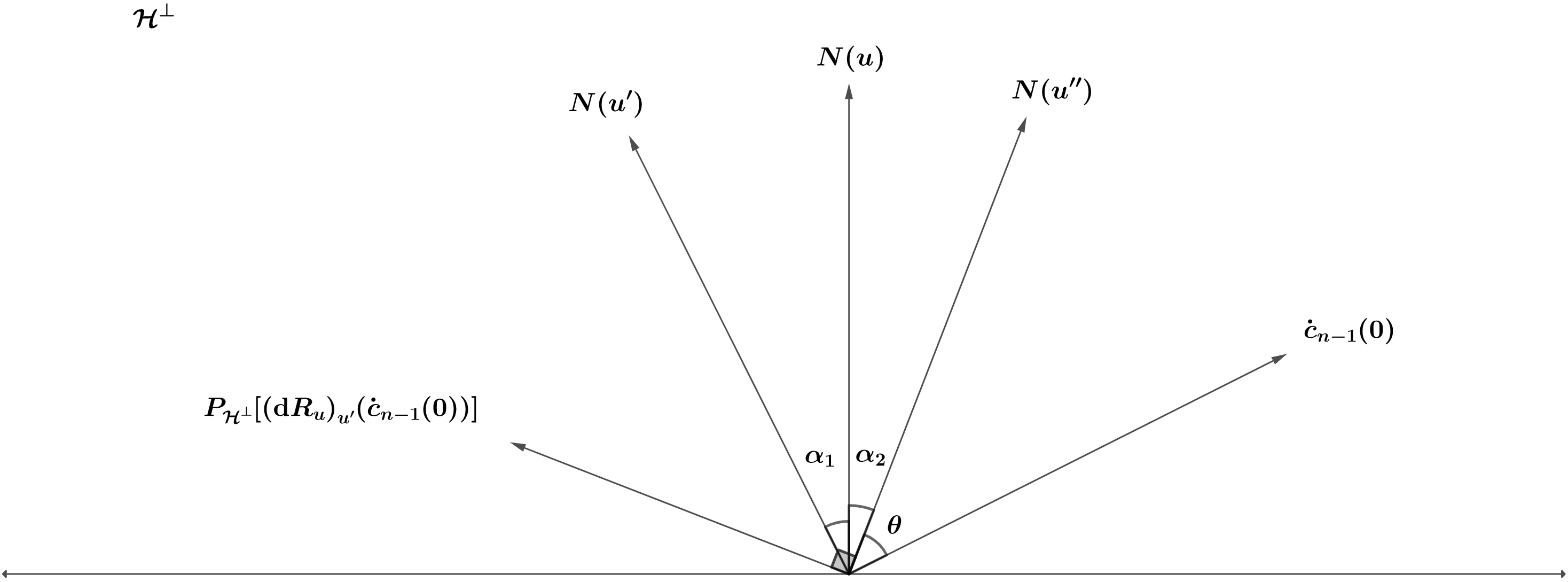}
\caption{\footnotesize A graphical representation of the proof of Claim \ref{claim1-22may24}.} \label{Hperppic}
\end{figure}
\end{proof}
Returning to \eqref{ineq1-22may24},
\begin{eqnarray}\label{eq1-28may24}
\begin{aligned}
|W_{1}-W_{2}|^{2}
    &\displaystyle= \det{(I-M_{u,u'})}^{2} - 2\det{(I-M_{u,u'})}^{2}\langle (\mathrm{d}R_{u})_{u'}(\dot{c}_{n-1}(0)),\dot{c}_{n-1}(0)\rangle\\& + \Delta(u,u')^{2}\det{(M_{u,u'}^{-1}-I)}^{2}\\
    &\geq 1+\Delta(u,u')^{2},
\end{aligned}
\end{eqnarray}
by \eqref{eq3-17apr24} and Claim \ref{claim1-22may24}, which concludes the proof of Proposition \ref{prop-Jtilde-09apr24}.
\end{proof}

\subsubsection*{Step 2: Bounding $J/\widetilde{J}$}
Let $\lambda_{1}(p)\leq\lambda_{2}(p)\leq\cdots\leq\lambda_{n-1}(p)$ be the eigenvalues of the shape operator $\mathrm{d}N$ at $p$. 
Since $u''-u'\in \langle N(u)\rangle^{\perp}$,
\begin{equation*}
\eqalign{
    \displaystyle\left|\frac{\langle u''-u',N(u'')\rangle}{\langle P_{T_{u''}S}N(u),(\mathrm{d}N_{u''})^{-1}(P_{T_{u''}S}N(u))\rangle}\right|&\displaystyle = \left|\frac{\langle u''-u',N(u'')-N(u)\rangle}{\langle P_{T_{u''}S}N(u),(\mathrm{d}N_{u''})^{-1}(P_{T_{u''}S}N(u))\rangle}\right| \cr
    &\leq\displaystyle \lambda_{n-1}(u'') \frac{|\langle u''-u',N(u'')-N(u)\rangle|}{|P_{T_{u''}S}N(u)|^{2}}.\cr
    }
\end{equation*}
Using the fact that $|P_{T_{u''}S}N(u)|=|N(u'')\wedge N(u)|\approx |N(u'')-N(u)|$, which follows from \eqref{cone of normals}, we have
\begin{equation*}\label{ineq1-24may24}
\eqalign{
\displaystyle J(u,u')&\displaystyle =\left(\frac{|N(u')\wedge N(u'')|}{|N(u)\wedge N(u')|}\right)^{n-2}\left|\frac{\langle u''-u',N(u'')\rangle}{\langle P_{T_{u''}S}N(u),(\mathrm{d}N_{u''})^{-1}(P_{T_{u''}S}N(u))\rangle}\right|\frac{K(u)}{K(u'')} \cr
&\displaystyle \lesssim\left(\frac{|N(u')-N(u'')|}{|N(u)- N(u')|}\right)^{n-2}\frac{|\langle u''-u',N(u'')-N(u)\rangle|}{|P_{T_{u''}S}N(u)|^{2}}\frac{\prod_{j=1}^{n-1}\lambda_{j}(u)}{\prod_{j=1}^{n-1}\lambda_{j}(u'')} \lambda_{n-1}(u'') \cr
&\lesssim\displaystyle\left(\frac{|u'-u''|}{|u-u'|}\right)^{n-2}\left(\frac{\sup_{p}\lambda_{n-1}(p)}{\inf_{p}\lambda_{1}(p)}\right)^{n-2} \frac{|u''-u'|\cdot|N(u'')-N(u)|}{|N(u'')-N(u)|^{2}}\frac{\prod_{j=1}^{n-2}\lambda_{j}(u)}{\prod_{j=1}^{n-2}\lambda_{j}(u'')}\lambda_{n-1}(u) \cr
&\lesssim\displaystyle\left(\frac{|u'-u''|}{|u-u'|}\right)^{n-2}\frac{|u''-u'|}{|N(u'')-N(u)|} Q(S)^{2(n-2)} \sup_{p}\lambda_{n-1}(p). \cr
}
\end{equation*}
Hence by 
\eqref{normal2plane} and the fact that $\widetilde{J}(u,u')\geq 1$ (see Proposition \ref{prop-Jtilde-09apr24}),
\begin{equation}\label{take}
        \displaystyle\frac{J(u,u')}{\widetilde{J}(u,u')}\lesssim Q(S)^{\frac{5(n-2)}{2}} \frac{|u''-u'|}{|N(u'')-N(u)|}  \sup_{p}\lambda_{n-1}(p).
    \end{equation}
On the other hand, using the fact that $\widetilde{J}(u,u')\geq\Delta(u,u')$, which also follows from Proposition \ref{prop-Jtilde-09apr24},
    \begin{equation*}
    \eqalign{
    \displaystyle \frac{J(u,u')}{\widetilde{J}(u,u')}\leq\displaystyle \frac{J(u,u')}{\Delta(u,u')}=J(u,u'')&\lesssim\displaystyle\left(\frac{|u'-u''|}{|u-u''|}\right)^{n-2}\frac{|u''-u'|}{|N(u')-N(u)|} Q(S)^{2(n-2)} \sup_{p}\lambda_{n-1}(p)\cr
    &\lesssim\displaystyle Q(S)^{\frac{5(n-2)}{2}}\frac{|u''-u'|}{|N(u')-N(u)|} \sup_{p}\lambda_{n-1}(p),\cr
    }
\end{equation*}
by the distance estimate \eqref{normal2plane} and by \eqref{switch}. Consequently,
\begin{equation}\label{ineq1-18apr24}
    \eqalign{
    \displaystyle\frac{J(u,u')}{\widetilde{J}(u,u')}&\displaystyle\lesssim |u''-u'| Q(S)^{\frac{5(n-2)}{2}}\frac{1}{\max\{|N(u'')-N(u)|,|N(u')-N(u)|\}}\sup_{p}\lambda_{n-1}(p)\cr
    &\displaystyle\lesssim |u''-u'| Q(S)^{\frac{5(n-2)}{2}}\frac{1}{|N(u'')-N(u)|+|N(u')-N(u)|} \sup_{p}\lambda_{n-1}(p) \cr
    &\displaystyle\lesssim \frac{|u''-u'|}{|N(u'')-N(u')|} Q(S)^{\frac{5(n-2)}{2}} \sup_{p}\lambda_{n-1}(p) \cr
     &\displaystyle\lesssim \frac{1}{\inf_{p}\lambda_{1}(p)} Q(S)^{\frac{5(n-2)}{2}} \sup_{p}\lambda_{n-1}(p)  \cr
        &\displaystyle\lesssim Q(S)^{\frac{5n-8}{2}}, \cr
    }
\end{equation}
by the mean-value inequality applied to the Gauss map $N$. This implies \eqref{clever bit}, completing the proof of Theorem \ref{Theorem:shortSobStein} (Theorem \ref{Theorem:SobStein}).

\subsection{Proof of the Sobolev--Mizohata--Takeuchi inequality (Theorem \ref{theorem:shortSMTgen})} 
In this section we prove Theorem \ref{theorem:shortSMTgen}, or more specifically, Theorem \ref{theorem:SMTgen}.
We begin by observing that
if $u\not\in\supp^*(g)$ and $u'\in S$ then either $u'\not\in\supp(g)$ or $R_u u'\not\in\supp(g)$, meaning that $I_{S,s}(|g|^2, |g|^2)(u)=0$. Consequently, by Theorem \ref{Theorem:SobStein},
$$
\int_{\mathbb{R}^n}|\widehat{g\mathrm{d}\sigma}|^2w\leq c Q(S)^{\frac{5n-8}{4}} \sup_{u\in\supp^*(g)}\|X_Sw(u,\cdot)\|_{\dot{H}^s(T_u S)}\int_S I_{S,s}(|g|^2,|g|^2)(u)^{1/2}\mathrm{d}\sigma(u),
$$
and so we are reduced to proving a suitable $L^1(S)\times L^1(S)\rightarrow L^{1/2}(S)$ estimate on the bilinear operator 
\begin{equation}\label{nuffgen}
I_{S,s}(g_1,g_2)(u):=\int_{S}\frac{g_1(u')g_2(R_u u')}{|u'-R_u u'|^{s}}J(u,u')\mathrm{d}\sigma(u')
\end{equation}
whenever $s<n-1$. This follows by a direct application of the forthcoming Theorem \ref{l:EndpointKS}.
\subsection{Improved Sobolev--Stein constants in the plane}\label{what is needed really}
Our proof of Theorem \ref{Theorem:shortSobStein} identifies $\|J/\widetilde{J}\|_\infty^{1/2}$ as the naturally occurring dilation-invariant functional on the surface $S$, rather than the power of the curvature quotient $Q(S)$ that we use to bound it. In two dimensions our expression for $J$, being relatively simple, permits the bound $\|J/\widetilde{J}\|_\infty^{1/2}\lesssim\Lambda(S)$, where $\Lambda(S)$ is defined in \eqref{rightthing}.
To see this we argue as in \eqref{ineq1-18apr24}, using Propositions \ref{jacprop} and \ref{prop-Jtilde-09apr24} to write
\begin{eqnarray*}
    \begin{aligned}
\frac{J(u,u')}{\widetilde{J}(u,u')}\leq \min\{J(u,u'),J(u,u'')\}=|u'-u''|K(u)\min\left\{\frac{1}{|N(u)\wedge N(u'')|}, \frac{1}{|N(u)\wedge N(u')|}\right\}
\lesssim \Lambda(S).
\end{aligned}
\end{eqnarray*}
The two-dimensional case of Theorem \ref{Theorem:shortSobStein} may then be strengthened to the following:
\begin{theorem}[Improved Sobolev--Stein in the plane]\label{better2} Suppose that $s<\frac{1}{2}$. There is an absolute constant $c$ such that
$$
\int_{\mathbb{R}^2}|\widehat{g\mathrm{d}\sigma}(x)|^2w(x)\mathrm{d}x\leq c\Lambda(S)\int_S I_{S,2s}(|g|^2,|g|^2)(u)^{1/2}\|X_Sw(u,\cdot)\|_{\dot{H}^s(T_u S)}\mathrm{d}\sigma(u).
$$
\end{theorem}
A similar, although potentially rather more complicated statement is possible in higher dimensions, and is left to the interested reader.

\section{Estimating distances: the proof of Proposition \ref{metricestimate}}\label{distanceestimates}
We begin with \eqref{normal2plane}, and the elementary observation that if $\pi$ is 2-plane that is normal to $S$ at a point $u$, then by \eqref{cone of normals}, it must be close to normal at all points of intersection with $S$. More specifically, for $\widetilde{u}\in S$ we have $$|P_\pi N(\widetilde{u})|\geq|P_{(T_uS)^\perp}N(\widetilde{u})|=N(u)\cdot N(\widetilde{u})\geq 1/2.$$ It follows by Meusnier's theorem that for such a $\pi$, the curvature of the curve $S\cap\pi$ at a point is comparable to a normal curvature of $S$ at that same point. This allows us to transfer the curvature quotient of $S$ to such curves, and we shall appeal to this momentarily.

Now let $\pi'$ and $\pi''$ be the normal 2-planes at the point $u$ that pass through the points $u'$ and $u''$ respectively. Let $x$ be the orthogonal projection of $u$ onto the plane $T_uS+\{u'\}$, and note that $\{u,u',x\}$ and $\{u,u'',x\}$ are the vertices of right angled triangles in the 2-planes $\pi'$ and $\pi''$ respectively. Next observe that by the triangle inequality and Pythagoras' theorem, it is enough to show that 
\begin{equation}\label{please}
|x-u''|\lesssim Q(S)^{1/2}|x-u'|.
\end{equation}
To see this we write $S$ as a graph over $T_uS+\{u'\}$ as follows: let $\phi_u:T_uS+\{u'\}\rightarrow\mathbb{R}$ be such that  $x'\mapsto x'+\phi_u(x')N(u)$ is a bijective map from a subset $U\subset T_uS$ into $S$; see Figure \ref{Rmap}. That this is possible, and indeed that $\phi_u$ is uniquely defined, follows from \eqref{cone of normals} (a point that is elaborated in \cite{BNS}). Notice that 
\begin{equation}\label{preparing for Taylor}
\phi_u(u')=0, \phi_u(x)=|x-u|\;\;\mbox{ and }\;\;\nabla\phi_u(x)=0,
\end{equation}
by construction.
Assuming that $N(u)=e_n$, as we may, the graph condition \eqref{cone of normals} implies that the normal vector $(\nabla\phi_u, -1)$ lies in some fixed (proper) vertical cone, and so in particular we also have 
\begin{equation}\label{slow}
|\nabla\phi_u|\lesssim 1.
\end{equation}
We now apply Taylor's theorem on the line segment $[x,u']$, along with \eqref{preparing for Taylor}, to obtain
$$
|x-u|=\phi_u(x)-\phi_u(u')=\frac{1}{2}k'(u,u')|x-u'|^2,
$$
where $k'(u,u')$ is a quantity comparable to some normal curvature of $S$ at some point. Here we have used \eqref{slow} along with our initial observation via Meusnier's theorem. By symmetry a similar statement may be made with $u''$ in place of $u'$, from which we deduce that $$k'(u,u')|x-u'|^2=k''(u,u'')|x-u''|^2.$$ The inequality \eqref{please} now follows from the definition of $Q(S)$ and taking square roots.

Turning to \eqref{easier}, we fix $u$ and exploit the properties of the map $H:=H_{\omega}= N^{-1}\circ\Phi_{\omega}$ from Section \ref{Jacobians}.
By the mean value theorem and Claim \ref{claim1-26feb24},
\begin{equation*}
    \eqalign{
    \displaystyle |u-u'| = |H(0)-H(x')| &\displaystyle\leq \sup_{\theta}\|\mathrm{d}H_{\theta}\|\cdot |x'|
    \displaystyle\leq \sup_{\theta}\|\mathrm{d}H_{\theta}\|\cdot \frac{|(1-\widetilde{\eta}(x'))x'|}{|1-\widetilde{\eta}(x')|}
    \displaystyle=\sup_{\theta}\|\mathrm{d}H_{\theta}\|\cdot \frac{|x'-x''|}{|1-\widetilde{\eta}(x')|}, \cr
    }
\end{equation*}
where $x''$ is such that $H(x'')=u''$. Consequently,
\begin{equation*}
    \eqalign{
    \displaystyle |u-u'|  &\displaystyle\leq\sup_{\theta}\|\mathrm{d}H_{\theta}\|\cdot  \frac{|H^{-1}(H(x'))-H^{-1}(H(x''))|}{|1-\widetilde{\eta}(x')|} \cr
    &\displaystyle\leq\sup_{\theta}\|\mathrm{d}H_{\theta}\|\cdot \sup_{\widetilde{\theta}}\|\mathrm{d}H^{-1}_{\widetilde{\theta}}\|\cdot  \frac{|H(x')-H(x'')|}{|1-\widetilde{\eta}(x')|} \cr
    &\displaystyle=\sup_{\theta}\|\mathrm{d}H_{\theta}\|\cdot \sup_{\widetilde{\theta}}\|\mathrm{d}H^{-1}_{\widetilde{\theta}}\|\cdot  \frac{|u'-u''|}{|1-\widetilde{\eta}(x')|}, \cr
    }
\end{equation*}
and
therefore
\begin{equation*}
    |u'-u''|\geq \frac{|1-\widetilde{\eta}(x')|}{\sup_{\theta}\|\mathrm{d}H_{\theta}\|\cdot \sup_{\widetilde{\theta}}\|\mathrm{d}H^{-1}_{\widetilde{\theta}}\|}\cdot |u-u'|.
\end{equation*}
We also have, for a fixed $\theta$,
\begin{equation*}
    \|\mathrm{d}H_{\theta}\|\leq \|\mathrm{d}N^{-1}_{\Phi(\theta)}\|\cdot \|\mathrm{d}\Phi_{\theta}\|\leq\frac{1}{\inf_{p\in S}\lambda_1(p)}\cdot \|\mathrm{d}\Phi_{\theta}\|_{L^{\infty}_{\theta}},
\end{equation*}
where $\inf_{p\in S}\lambda_1(p)$ is the infimum over $p\in S$ of the smallest eigenvalue $\lambda_1(p)$ of the shape operator $\mathrm{d}N_{p}$. Similarly, 
\begin{equation*}
     \|\mathrm{d}H^{-1}_{\widetilde{\theta}}\|\leq  \|\mathrm{d}\Phi^{-1}_{\widetilde{\theta}}\|\cdot \|\mathrm{d}N_{\Phi(\widetilde{\theta})}\|\leq\|\mathrm{d}\Phi^{-1}_{\widetilde{\theta}}\|_{L^{\infty}_{\widetilde{\theta}}}\cdot\sup_{p}\lambda_{n-1}(p),
\end{equation*}
where $\sup_{p\in S}\lambda_{n-1}(p)$ is the supremum over $p\in S$ of the largest eigenvalue $\lambda_{n-1}(p)$.
Consequently,
\begin{equation}\label{comparability-3apr24}
    |u'-u''|\gtrsim |1-\widetilde{\eta}(x')|\cdot \frac{\inf_{p\in S}\lambda_1(p)}{\sup_{p\in S}\lambda_{n-1}(p)}\cdot |u-u'|\gtrsim \frac{1}{Q(S)}\cdot |u-u'|,
\end{equation}
since $\widetilde{\eta}<0$ by the strict convexity of $S$.

\section{Computing Jacobians: the proof of Proposition \ref{jacprop}}\label{Jacobians}
In this section we provide detailed proofs of \eqref{Jformula-prop-19apr24}, \eqref{deltaformula-prop-19apr24} and \eqref{switch}. 
The key idea is that the maps $u\mapsto R_u u'$ and $u'\mapsto R_uu'$ may be transformed into \textit{outward vector fields} on Euclidean spaces (specifically $T_{u'}S$ and $T_uS$ respectively) by conjugating them with a composition of the Gauss map and a suitable stereographic projection. The derivatives of such vector fields have only two eigenspaces, allowing the computation of their Jacobians to be reduced to the identification of just two eigenvalues, one of which has multiplicity $n-2$ (see the forthcoming Lemma \ref{detphilemma}). This is manifested in the factor raised to the power $n-2$ in the formula \eqref{Jformula-prop-19apr24} for $J$.
We begin by recalling and introducing the notation and geometric objects that will feature in our computations of $J$ and $\Delta$.
\begin{itemize}
    \item $N:S\rightarrow\mathbb{S}^{n-1}$ is the Gauss map, $\mathrm{d}N_{u}:T_{u}S\rightarrow T_{N(u)}\mathbb{S}^{n-1}$ is the shape operator (recall that $T_{u}S= T_{N(u)}\mathbb{S}^{n-1}$), and $K(u)=\det(\mathrm{d}N_{u})$ is the Gaussian curvature at $u\in S$. 
    \item The formulas of this section will be written in terms of the parameters $u$, $u'$ and $u''=R_{u}u'$, which are points on $S$. We will denote their images via the Gauss map by $\omega$, $\omega^{\prime}$ and $\omega^{\prime\prime}$, respectively.
    \item For a fixed $\omega^{\prime}\in\mathbb{S}^{n-1}$, $\Phi_{\omega^{\prime}}:\langle\omega^{\prime}\rangle^{\perp}\rightarrow\mathbb{S}^{n-1}$ denotes the \textit{inverse} of the stereographic projection map with respect to $-\omega^{\prime}$. Explicitly
\begin{equation}\label{coord1-20feb24}
    \Phi_{\omega^{\prime}}(x)=\left(\frac{2x}{1+|x|^{2}},\frac{1-|x|^{2}}{1+|x|^{2}}\right)
\end{equation}
via the identification $\mathbb{R}^n=\langle\omega'\rangle^\perp\times\langle\omega'\rangle$.
If $\omega=\Phi_{\omega^{\prime}}(x)$, it follows that 
\begin{equation}\label{xandomega-23feb24}
    x=\frac{\omega-\langle\omega,\omega^{\prime}\rangle\omega^{\prime}}{1+\langle\omega,\omega^{\prime}\rangle}.
\end{equation}
The differential $(\mathrm{d}\Phi_{\omega^{\prime}})_{x}:\langle\omega^{\prime}\rangle^{\perp}\rightarrow\langle\omega\rangle^{\perp}$ satisfies
\begin{equation}\label{diffPhi-23feb24}
    (\mathrm{d}\Phi_{\omega^{\prime}})_{x}(x)=\langle\omega,\omega^{\prime}\rangle\omega-\omega^{\prime}.
\end{equation}
The determinants of $(\mathrm{d}\Phi_{\omega^{\prime}})_{x}$ and its inverse are, respectively,
\begin{equation}\label{jacstereo1}
    \det((\mathrm{d}\Phi_{\omega^{\prime}})_{x})=\left(\frac{2}{1+|x|^{2}}\right)^{n-1}=(1+\langle\omega,\omega^{\prime}\rangle)^{n-1}
\end{equation}
and
\begin{equation}\label{jacstereo2}
    \det((\mathrm{d}\Phi_{\omega^{\prime}}^{-1})_{\omega})=\left(\frac{1}{1+\langle\omega,\omega^{\prime}\rangle}\right)^{n-1}.
\end{equation}
We refer the reader to Chapter 4 of \cite{LL} for further discussion on the properties of these maps.
\item For $\omega$ fixed, set 
\begin{equation*}
    H_{\omega}=N^{-1}\circ\Phi_{\omega}.
\end{equation*}
$H_{\omega}$ will play a crucial role in this section. As we shall see, it allows us to reduce the computations of $J$ and $\Delta$ to certain Euclidean analogues with simple spectral structure (outward vector fields, as discussed above and alluded to in Remark \ref{Remark:nondeg}).
\end{itemize}

We are now ready to prove \eqref{Jformula-prop-19apr24}, \eqref{deltaformula-prop-19apr24} and \eqref{switch}.
\subsection{Computing $J$}\label{subsectionJ} For fixed $\omega'$ we define the map $\Psi_{\omega^{\prime}}:N(S)\rightarrow\mathbb{S}^{n-1}$ by
\begin{equation}\label{changevar-20feb24}
    \Psi_{\omega^{\prime}}(\omega)=N(R_{N^{-1}(\omega)}N^{-1}(\omega^{\prime})).
\end{equation}
Strictly speaking the domain of  $\Psi_{\omega^{\prime}}$ depends on $\omega'$, as we allude to in Remark \ref{Rmk1-130825}. The parameter $\omega\in \mathbb{S}^{n-1}$ will be a variable and we will use $x\in\langle\omega^{\prime}\rangle^{\perp}$ to represent its preimage by the map $\Phi_{\omega^{\prime}}$. Explicitly,
\begin{equation*}
    x\xmapsto{\Phi_{\omega^{\prime}}} \omega\xmapsto{N^{-1}} u.
\end{equation*}
By \eqref{changevar-20feb24} and the definition of $J(u,u')$, along with the fact that the Gaussian curvature $K(u)$ is the determinant of the shape operator $\mathrm{d}N_u$, we have
\begin{equation}\label{jac19feb24}
    J(u,u')=\left|\det{\left(\mathrm{d}\Psi_{N(u')}(N(u))\right)}\right|\frac{K(u)}{K(u'')}.
\end{equation}
The next step is to reduce the computation of the Jacobian determinant $\det{\left(\mathrm{d}\Psi_{N(u')}(N(u))\right)}$ to one of a much simpler outward vector field $\varphi$ on the tangent space at $u'$ (see Lemma \ref{detphilemma} below). This will be achieved by combining properties of the inverse stereographic projection map $\Phi_{\omega^{\prime}}$ with the geometric condition \eqref{collision condition 2}. To this end we define the map $\varphi:\langle\omega^{\prime}\rangle^{\perp}\rightarrow\langle\omega^{\prime}\rangle^{\perp}$ by
$$\varphi(x):=\Phi_{\omega^{\prime}}^{-1}\circ\Psi_{\omega^{\prime}}\circ\Phi_{\omega^{\prime}}(x).$$
\begin{claim}\label{claim1-23feb24} The vector field $\varphi:\langle\omega'\rangle^\perp\rightarrow\langle\omega'\rangle^\perp$ is given by
\begin{equation}
    \varphi(x) = \eta(x)x,
\end{equation}
where
\begin{equation}\label{etadef}
     \eta(x) = \frac{\langle x,H_{\omega^{\prime}}^{-1}(R_{H_{\omega^{\prime}}(x)}H_{\omega^{\prime}}(0))\rangle}{|x|^{2}}=\frac{\langle x,\Phi_{\omega^{\prime}}^{-1}(\omega^{\prime\prime})\rangle}{|x|^{2}}.
\end{equation}
\end{claim}
\begin{proof}[Proof of Claim \ref{claim1-23feb24}] By definition of the map $R_{(\cdot)}u'$, the normals $\omega$, $\omega^{\prime}$ and $\omega^{\prime\prime}$ are coplanar, therefore they lie on a great circle. This implies that
$$\varphi(x)=\mu(x)x$$
for some $\mu(x)$, which we conclude to be equal to $\eta(x)$ by taking scalar products with $x$ on both sides of the equation above.
\end{proof}
By the chain rule,
$$\det(\mathrm{d}\varphi(x))=\det((\mathrm{d}\Phi_{\omega^{\prime}}^{-1})(\omega^{\prime\prime}))\det((\mathrm{d}\Psi_{\omega^{\prime}})(\omega))\det((\mathrm{d}\Phi_{\omega^{\prime}})(x)),$$
hence
\begin{equation*}
    \det((\mathrm{d}\Psi_{\omega^{\prime}})(\omega))=\frac{\det(\mathrm{d}\varphi(x))}{\det((\mathrm{d}\Phi_{\omega^{\prime}}^{-1})(\omega^{\prime\prime}))\det((\mathrm{d}\Phi_{\omega^{\prime}})(x))}.
\end{equation*}
This implies, by \eqref{jac19feb24},
\begin{equation}\label{almostfinaljac-20feb24}
    J(u,u')=\frac{|\det(\mathrm{d}\varphi(x))|}{|\det((\mathrm{d}\Phi_{\omega^{\prime}}^{-1})(\omega^{\prime\prime}))||\det((\mathrm{d}\Phi_{\omega^{\prime}})(x))|}\frac{K(u)}{K(u'')}.
\end{equation}

We are now in a position to invoke the following elementary lemma, whose proof is left to the reader:
\begin{lemma}[Differential structure of an outward vector field]\label{detphilemma} Let $\eta:\mathbb{R}^{d}\rightarrow\mathbb{R}$ be a $C^{1}$ function and let $\varphi:\mathbb{R}^{d}\rightarrow\mathbb{R}^{d}$ given by
\begin{equation}
    \varphi(x)=\eta(x)x.
\end{equation}
The linear map 
$$\mathrm{d}\varphi(x)=x(\nabla\eta(x))^{\top}+\eta(x) I_{d}$$
has eigenvalues $\lambda_{1}(x)=\eta(x)$ and $\lambda_{2}(x)=\langle\nabla\eta(x),x\rangle + \eta(x)$ of multiplicity $(d-1)$ and $1$, respectively. The eigenspaces associated to these eigenvalues are
\begin{equation*}
    \eqalign{
    \displaystyle E_{\lambda_{1}(x)}&\displaystyle :=\langle\nabla\eta(x)\rangle^{\perp}, \cr
    \displaystyle E_{\lambda_{2}(x)}&\displaystyle :=\langle x\rangle. \cr
    }
\end{equation*}
In particular,
    \begin{equation}\label{det-eigenvalues}
        \det(\mathrm{d}\varphi(x))=[\eta(x)]^{d-1}(\langle\nabla\eta(x),x\rangle + \eta(x)).
    \end{equation}
\end{lemma}

The parameter $u'\in S$ is fixed in this subsection, therefore $\omega^{\prime}$ will also be fixed, and we write $H_{\omega^{\prime}}=H$ to simplify notation. Let us use \eqref{det-eigenvalues} to compute $\det(\mathrm{d}\varphi(x))$. The eigenvalue $\lambda_{1}(x)$ of $\mathrm{d}\varphi(x)$ is
\begin{equation*}
    \lambda_{1}(x)=\eta(x)=\frac{\langle x,H^{-1}(R_{H(x)}H(0))\rangle}{|x|^{2}},
\end{equation*}
hence, by \eqref{det-eigenvalues}, all there is left to do is to compute the eigenvalue $\lambda_{2}(x)$ of $\mathrm{d}\varphi(x)$. By definition of the map $R_{(\cdot)}u'$, the vector $R_{u}u'-u'$ is in the tangent space of $S$ at $u$. In short,
    \begin{equation*}
        \langle R_{u}u'-u', N(u)\rangle = 0.
    \end{equation*}
Equivalently,
  \begin{equation}\label{cond2-19feb24}
        \langle H(\eta(x)x)-H(0), N(H(x))\rangle = 0.
    \end{equation}
Differentiating both sides of \eqref{cond2-19feb24} with respect to $x$,
\begin{equation*}
    \displaystyle 0  = \mathrm{d}(N\circ H)_{x}^{\top}\left(H(\eta(x)x)-H(0)\right) + \left(x\cdot\nabla\eta(x)^{\top}+\eta(x)I_{n-1}\right)^{\top}\mathrm{d}H_{\eta(x) x}^{\top}\left(N\circ H(x)\right).
\end{equation*}
Taking scalar products on both sides with $x$ and using that $N\circ H=\Phi_{\omega'}$, we have
\begin{equation*}
    \displaystyle 0  = \langle H(\eta(x)x)-H(0),(\mathrm{d}\Phi_{\omega^{\prime}})_{x}(x)\rangle + \langle\mathrm{d}H_{\eta(x) x}^{\top}\left(\Phi_{\omega^{\prime}}(x)\right),\left(x\cdot\nabla\eta(x)^{\top}+\eta(x)I_{n-1}\right)(x)\rangle. 
\end{equation*}
By Lemma \ref{detphilemma},
$$\left(x\cdot\nabla\eta(x)^{\top}+\eta(x)I_{n-1}\right)(x)=(\langle\nabla\eta(x),x\rangle+\eta(x)) x,$$
and hence
\begin{equation*}
    \lambda_2(x)=\langle\nabla\eta(x),x\rangle+\eta(x) = -\frac{\langle H(\eta(x)x)-H(0),(\mathrm{d}\Phi_{\omega^{\prime}})_{x}(x)\rangle}{\langle\mathrm{d}H_{\eta(x) x}^{\top}\left(\Phi_{\omega^{\prime}}(x)\right),x\rangle}.
\end{equation*}
By Lemma \ref{detphilemma} again,
\begin{equation*}
    \det(\mathrm{d}\varphi(x))=-[\eta(x)]^{n-2}\frac{\langle H(\eta(x)x)-H(0),(\mathrm{d}\Phi_{\omega^{\prime}})_{x}(x)\rangle}{\langle\mathrm{d}H_{\eta(x) x}^{\top}\left(\Phi_{\omega^{\prime}}(x)\right),x\rangle}.
\end{equation*}
By \eqref{almostfinaljac-20feb24},
\begin{equation}
    J(u,u')=|\eta(x)|^{n-2}\frac{1}{|\langle\mathrm{d}H_{\eta(x) x}^{\top}\left(\Phi_{\omega^{\prime}}(x)\right),x\rangle|}\frac{|\langle H(\eta(x)x)-H(0),(\mathrm{d}\Phi_{\omega^{\prime}})_{x}(x)\rangle|}{|\det((\mathrm{d}\Phi_{\omega^{\prime}}^{-1})(\omega^{\prime\prime}))||\det((\mathrm{d}\Phi_{\omega^{\prime}})(x))|}\frac{K(u)}{K(u'')}.
\end{equation}
To proceed, we need to understand each factor in the formula above, which is the content of the next claim.
\begin{claim}\label{claim2-23feb24} The following identities hold:
    \begin{equation}\label{etaformula}
        |\eta(x)|=\left|\frac{\langle\omega,\omega^{\prime\prime}\rangle-\langle\omega,\omega^{\prime}\rangle\langle\omega^{\prime},\omega^{\prime\prime}\rangle}{(1+\langle\omega^{\prime\prime},\omega^{\prime}\rangle)(1-\langle\omega,\omega^{\prime}\rangle)}\right|;
    \end{equation}
    \begin{equation}\label{secondformula-23feb24}
        \langle H(\eta(x)x)-H(0),(\mathrm{d}\Phi_{\omega^{\prime}})_{x}(x)\rangle = -\langle u''-u',\omega'\rangle;
    \end{equation}
    \begin{equation}
        \langle\mathrm{d}H_{\eta(x) x}^{\top}\left(\Phi_{\omega^{\prime}}(x)\right),x\rangle = \frac{1}{\eta(x)}\langle\omega,\mathrm{d}N^{-1}_{\omega^{\prime\prime}}(\langle\omega^{\prime\prime},\omega^{\prime}\rangle\omega^{\prime\prime}-\omega^{\prime})\rangle.
    \end{equation}
\end{claim}
Let us assume Claim \ref{claim2-23feb24} for the moment and complete the proof of the proposition. By the claim, \eqref{jacstereo1} and \eqref{jacstereo2},
\begin{equation}\label{finalJ-19apr24}
    \eqalign{
     \displaystyle J(u,u')&=\displaystyle\left|\frac{\langle\omega,\omega^{\prime\prime}\rangle-\langle\omega,\omega^{\prime}\rangle\langle\omega^{\prime},\omega^{\prime\prime}\rangle}{(1+\langle\omega^{\prime\prime},\omega^{\prime}\rangle)(1-\langle\omega,\omega^{\prime}\rangle)}\right|^{n-1}\frac{|\langle u''-u',\omega'\rangle|}{|\langle\omega,\mathrm{d}N^{-1}_{\omega^{\prime\prime}}(\langle\omega^{\prime\prime},\omega^{\prime}\rangle\omega^{\prime\prime}-\omega^{\prime})\rangle|}\frac{|1+\langle\omega^{\prime\prime},\omega^{\prime}\rangle|^{n-1}}{|1+\langle\omega,\omega^{\prime}\rangle|^{n-1}}\frac{K(u)}{K(u'')} \cr
     &=\displaystyle\left|\frac{\langle\omega,\omega^{\prime\prime}\rangle-\langle\omega,\omega^{\prime}\rangle\langle\omega^{\prime},\omega^{\prime\prime}\rangle}{1-\langle\omega,\omega^{\prime}\rangle^{2}}\right|^{n-1}\frac{|\langle u''-u',\omega'\rangle|}{|\langle\omega,\mathrm{d}N^{-1}_{\omega^{\prime\prime}}(\langle\omega^{\prime\prime},\omega^{\prime}\rangle\omega^{\prime\prime}-\omega^{\prime})\rangle|}\frac{K(u)}{K(u'')}. \cr
&=\displaystyle\left(\frac{(1-\langle\omega',\omega''\rangle^{2})^{\frac{1}{2}}}{(1-\langle\omega,\omega^{\prime}\rangle^{2})^{\frac{1}{2}}}\right)^{n-1}\frac{|\langle u''-u',\omega'\rangle|}{|\langle P_{T_{u''}S}N(u),(\mathrm{d}N_{u''})^{-1}(\langle\omega^{\prime\prime},\omega^{\prime}\rangle\omega^{\prime\prime}-\omega^{\prime})\rangle|}\frac{K(u)}{K(u'')}, \cr
    }
\end{equation}
where we used the facts that $\langle \omega,v\rangle=\langle P_{T_{u''}S}N(u),v\rangle$ for every $v\in T_{u''}S$, and that three coplanar vectors $\omega,\omega'$ and $\omega''$ on the sphere satisfy
\begin{equation*}
    \langle\omega,\omega^{\prime\prime}\rangle-\langle\omega,\omega^{\prime}\rangle\langle\omega^{\prime},\omega^{\prime\prime}\rangle = (1-\langle\omega',\omega''\rangle^{2})^{\frac{1}{2}}(1-\langle\omega,\omega'\rangle^{2})^{\frac{1}{2}}.
\end{equation*}
We exploit the coplanarity of $\omega,\omega'$ and $\omega''$ twice more. First, it implies the existence of $a,b\in\mathbb{R}$ such that
\begin{equation}\label{coplanarcond-19apr24}
    \omega''=a\omega+b\omega'.
\end{equation}
Consequently,
\begin{equation*}
    \frac{|\langle u''-u',\omega^{\prime\prime}\rangle|}{|\langle u''-u',\omega^{\prime}\rangle|}=\frac{|\langle u''-u',a\omega+b\omega'\rangle|}{|\langle u''-u',\omega^{\prime}\rangle|}=|b|,
\end{equation*}
since $u''-u'$ is perpendicular to $N(u)=\omega$. On the other hand, projecting both sides of \eqref{coplanarcond-19apr24} to $\langle\omega\rangle^{\perp}$ gives
\begin{equation*}
    P_{\langle\omega\rangle^{\perp}}\omega'' = b P_{\langle\omega\rangle^{\perp}}\omega'\Longrightarrow |b|=\frac{|P_{\langle\omega\rangle^{\perp}}\omega''|}{|P_{\langle\omega\rangle^{\perp}}\omega'|},
\end{equation*}
which in turn implies
\begin{equation}\label{claimDelta3-19apr24}
    \frac{|\langle u''-u',\omega^{\prime\prime}\rangle|}{|\langle u''-u',\omega^{\prime}\rangle|}=\frac{|P_{\langle\omega\rangle^{\perp}}\omega''|}{|P_{\langle\omega\rangle^{\perp}}\omega'|} \Longrightarrow |\langle u''-u',\omega^{\prime}\rangle|=\frac{(1-\langle\omega,\omega^{\prime}\rangle^{2})^{\frac{1}{2}}}{(1-\langle\omega,\omega''\rangle^{2})^{\frac{1}{2}}}|\langle u''-u',N(u'')\rangle|.
\end{equation}
Second, the fact that $\omega,\omega'$ and $\omega''$ are coplanar also gives us that $P_{\langle\omega''\rangle^{\perp}}\omega'$ and $P_{\langle\omega''\rangle^{\perp}}\omega$ are parallel, therefore
\begin{equation}\label{claimDelta4-19apr24}
    \langle\omega^{\prime\prime},\omega^{\prime}\rangle\omega^{\prime\prime}-\omega^{\prime}=P_{\langle\omega''\rangle^{\perp}}\omega'=\frac{|P_{\langle\omega''\rangle^{\perp}}\omega'|}{|P_{\langle\omega''\rangle^{\perp}}\omega|}P_{\langle\omega''\rangle^{\perp}}\omega =\frac{(1-\langle\omega',\omega''\rangle^{2})^{\frac{1}{2}}}{(1-\langle\omega,\omega''\rangle^{2})^{\frac{1}{2}}}P_{T_{u''}S}N(u).
\end{equation}
Likewise, or by symmetry,
\begin{equation}\label{claimDelta5-19apr24}
    \langle\omega^{\prime\prime},\omega^{\prime}\rangle\omega^{\prime}-\omega^{\prime\prime}=P_{\langle\omega'\rangle^{\perp}}\omega''=\frac{|P_{\langle\omega'\rangle^{\perp}}\omega''|}{|P_{\langle\omega'\rangle^{\perp}}\omega|}P_{\langle\omega'\rangle^{\perp}}\omega=\frac{(1-\langle\omega',\omega''\rangle^{2})^{\frac{1}{2}}}{(1-\langle\omega,\omega'\rangle^{2})^{\frac{1}{2}}}P_{T_{u'}S}N(u).
\end{equation}
Using \eqref{claimDelta3-19apr24} and \eqref{claimDelta4-19apr24} in \eqref{finalJ-19apr24} gives \eqref{Jformula-prop-19apr24}. We now move to the final part of the argument.
\begin{proof}[Proof of Claim \ref{claim2-23feb24}] By \eqref{etadef} and \eqref{xandomega-23feb24},
\begin{equation*}
\eqalign{
    \displaystyle|\eta(x)|&\displaystyle=\left|\left\langle \frac{\omega-\langle\omega,\omega^{\prime}\rangle\omega^{\prime}}{1+\langle\omega,\omega^{\prime}\rangle},\frac{\omega^{\prime\prime}-\langle\omega^{\prime\prime},\omega^{\prime}\rangle\omega^{\prime}}{1+\langle\omega^{\prime\prime},\omega^{\prime}\rangle}\right\rangle\right|\frac{|1+\langle\omega,\omega^{\prime}\rangle|^{2}}{|\omega-\langle\omega,\omega^{\prime}\rangle\omega^{\prime}|^{2}} \cr
    &\displaystyle= \left|\frac{\langle\omega,\omega^{\prime\prime}\rangle-\langle\omega,\omega^{\prime}\rangle\langle\omega^{\prime},\omega^{\prime\prime}\rangle}{(1+\langle\omega^{\prime\prime},\omega^{\prime}\rangle)(1-\langle\omega,\omega^{\prime}\rangle)}\right|,\cr
    }
\end{equation*}
which verifies \eqref{etaformula}. To establish \eqref{secondformula-23feb24}, we simply observe that $H(\eta(x)x)-H(0)=u''-u'$, and this together with \eqref{diffPhi-23feb24} implies that
\begin{equation*}
     \langle H(\eta(x)x)-H(0),(\mathrm{d}\Phi_{\omega^{\prime}})_{x}(x)\rangle = \langle u''-u',\langle\omega,\omega^{\prime}\rangle\omega-\omega^{\prime}\rangle = -\langle u''-u',\omega'\rangle,
\end{equation*}
since $u''-u'$ is perpendicular to $\omega$ by definition of $u''$. Finally, notice that $\Phi_{\omega^{\prime}}(\eta(x)x)=\omega^{\prime\prime}$ and that a direct computation gives
\begin{equation}\label{diffPhieta-23feb24}
    (\mathrm{d}\Phi_{\omega^{\prime}})_{\eta(x)x}(x)=\frac{1}{\eta(x)}\left(\langle\omega^{\prime\prime},\omega^{\prime}\rangle\omega^{\prime\prime}-\omega^{\prime}\right).
\end{equation}
Therefore by definition of $H$, the chain rule and \eqref{diffPhieta-23feb24}, we have
\begin{equation*}
    \eqalign{
    \displaystyle \langle\mathrm{d}H_{\eta(x) x}^{\top}\left(\Phi_{\omega^{\prime}}(x)\right),x\rangle &\displaystyle = \langle\omega,\mathrm{d}H_{\eta(x) x}(x)\rangle \cr
    &\displaystyle= \langle\omega,\mathrm{d}N^{-1}_{\Phi_{\omega^{\prime}}(\eta(x)x)}\circ(\mathrm{d}\Phi_{\omega^{\prime}})_{\eta(x)x}(x)\rangle\cr
     &\displaystyle= \frac{1}{\eta(x)}\langle\omega,\mathrm{d}N^{-1}_{\omega^{\prime\prime}}(\langle\omega^{\prime\prime},\omega^{\prime}\rangle\omega^{\prime\prime}-\omega^{\prime})\rangle,\cr
    }
\end{equation*}
which concludes the proof of Claim \ref{claim2-23feb24}.
\end{proof}

\subsection{Computing $\Delta$} Arguing as in Section \ref{subsectionJ}, for fixed $\omega$ we define the map $\widetilde{\Psi}_{\omega}:N(S)\rightarrow\mathbb{S}^{n-1}$ by
\begin{equation}\label{changevartilde-26feb24}
    \widetilde{\Psi}_{\omega}(\omega')=N(R_{N^{-1}(\omega)}N^{-1}(\omega^{\prime})).
\end{equation}
Recalling from Section \ref{Sect:general submanifolds} that $\Delta(u,u'')$ is the Jacobian of the change of variables $u'=R_u u''$, it follows that
\begin{equation}\label{Delta26feb24}
    \Delta(u,u')=\left|\det{\left(\mathrm{d}\widetilde{\Psi}_{N(u)}(N(u'))\right)}\right|\frac{K(u')}{K(u'')}.
\end{equation}
Recall that $\omega^{\prime}\in \mathbb{S}^{n-1}$ is a variable now. We will use $x'\in\langle\omega\rangle^{\perp}$ to represent its preimage by the map $\Phi_{\omega}$:
\begin{equation*}
    x'\xmapsto{\Phi_{\omega}} \omega^{\prime}\xmapsto{N^{-1}} u'.
\end{equation*}
Once more we reduce the computation of $\det{\left(\mathrm{d}\widetilde{\Psi}_{N(u)}(N(u'))\right)}$ to an application of Lemma \ref{detphilemma}. Define $\widetilde{\varphi}:\langle\omega\rangle^{\perp}\rightarrow\langle\omega\rangle^{\perp}$ by
$$\widetilde{\varphi}(x'):=\Phi_{\omega}^{-1}\circ\widetilde{\Psi}_{\omega}\circ\Phi_{\omega}(x').$$
\begin{claim}\label{claim1-26feb24} $\widetilde{\varphi}$ is given by
\begin{equation}
    \widetilde{\varphi}(x') = \widetilde{\eta}(x')x',
\end{equation}
where
\begin{equation}\label{etatildedef}
     \widetilde{\eta}(x') = \frac{\langle x',H_{\omega}^{-1}(R_{H_{\omega}(0)}H_{\omega}(x'))\rangle}{|x'|^{2}}=\frac{\langle x',\Phi_{\omega}^{-1}(\omega^{\prime\prime})\rangle}{|x'|^{2}}.
\end{equation}
\end{claim}
The proof of Claim \ref{claim1-26feb24} is similar to the one of Claim \ref{claim1-23feb24}. By the chain rule,
$$\det(\mathrm{d}\widetilde{\varphi}(x'))=\det((\mathrm{d}\Phi_{\omega}^{-1})(\omega^{\prime\prime}))\det((\mathrm{d}\Psi_{\omega})(\omega^{\prime}))\det((\mathrm{d}\Phi_{\omega})(x')),$$
hence
\begin{equation*}
    \det((\mathrm{d}\widetilde{\Psi}_{\omega})(\omega^{\prime}))=\frac{\det(\mathrm{d}\widetilde{\varphi}(x'))}{\det((\mathrm{d}\Phi_{\omega}^{-1})(\omega^{\prime\prime}))\det((\mathrm{d}\Phi_{\omega})(x'))}.
\end{equation*}
This implies, by \eqref{Delta26feb24}, that
\begin{equation}\label{almostfinaljac-26feb24}
    \Delta(u,u')=\frac{|\det(\mathrm{d}\widetilde{\varphi}(x'))|}{|\det((\mathrm{d}\Phi_{\omega}^{-1})(\omega^{\prime\prime}))||\det((\mathrm{d}\Phi_{\omega})(x'))|}\frac{K(u')}{K(u'')}.
\end{equation}
The parameter $u\in S$ is fixed in this subsection (and therefore so is $\omega\in\mathbb{S}^{n-1}$), so we lighten notation by writing $H_{\omega}=H$. We may again compute $\det(\mathrm{d}\widetilde{\varphi}(x'))$ using Lemma \ref{detphilemma}. The eigenvalue $\widetilde{\lambda}_{1}(x')$ of $\mathrm{d}\widetilde{\varphi}(x')$ is
\begin{equation*}
    \widetilde{\lambda}_{1}(x')=\widetilde{\eta}(x')=\frac{\langle x',H^{-1}(R_{H(0)}H(x'))\rangle}{|x'|^{2}},
\end{equation*}
hence we just have to compute the eigenvalue $\widetilde{\lambda}_{2}(x')$ of $\mathrm{d}\widetilde{\varphi}(x')$ and use \eqref{det-eigenvalues}. Recall that
    \begin{equation*}
        \langle R_{u}u'-u', N(u)\rangle = 0.
    \end{equation*}
Equivalently,
  \begin{equation}\label{cond2-26feb24}
        \langle H(\widetilde{\eta}(x')x')-H(x'), N(H(0))\rangle = 0.
    \end{equation}
Differentiating both sides of \eqref{cond2-26feb24} with respect to $x'$ and taking scalar products with $x'$ as well gives
\begin{equation*}
    \langle (\mathrm{d}H_{\widetilde{\varphi}(x')}\circ\mathrm{d}\varphi_{x'})^{\top}(\Phi_{\omega}(0)),x'\rangle = \langle\omega,\mathrm{d}H_{x'}(x')\rangle,
\end{equation*}
which in turn implies that
\begin{equation*}
    \langle (\mathrm{d}H_{\widetilde{\varphi}(x')})^{\top}(\omega),(\mathrm{d}\widetilde{\varphi}_{x'})(x')\rangle =\langle(\mathrm{d}H_{x'})^{\top}\omega,(x')\rangle= \langle\omega,\mathrm{d}H_{x'}(x')\rangle.
\end{equation*}
Using the fact that $x'$ is an eigenvector of $\mathrm{d}\widetilde{\varphi}(x')$ with eigenvalue $\widetilde{\lambda}_{2}(x')$ and that $H=N^{-1}\circ\Phi_{\omega}$ yields
\begin{equation*}
\eqalign{
    \displaystyle\widetilde{\lambda}_{2}(x')&\displaystyle =\frac{\langle\omega,\mathrm{d}N^{-1}_{\omega'}\circ(\mathrm{d}\Phi_{\omega})_{x'}(x')\rangle}{\langle\omega,\mathrm{d}N^{-1}_{\omega''}\circ(\mathrm{d}\Phi_{\omega})_{\widetilde{\eta}(x')x'}(x')\rangle} 
    =\displaystyle \widetilde{\eta}(x')\frac{\langle\omega,\mathrm{d}N^{-1}_{\omega'}(\langle\omega,\omega'\rangle\omega'-\omega)\rangle}{\langle\omega,\mathrm{d}N^{-1}_{\omega''}(\langle\omega'',\omega\rangle\omega''-\omega)\rangle}. \cr
    }
\end{equation*}
By Lemma \ref{detphilemma} once more,
\begin{equation*}
    \det(\mathrm{d}\widetilde{\varphi}(x'))=[\widetilde{\eta}(x')]^{n-1}\frac{\langle\omega,\mathrm{d}N^{-1}_{\omega'}(\langle\omega,\omega'\rangle\omega'-\omega)\rangle}{\langle\omega,\mathrm{d}N^{-1}_{\omega''}(\langle\omega'',\omega\rangle\omega''-\omega)\rangle}.
\end{equation*}
By \eqref{almostfinaljac-26feb24},
\begin{equation*}
    \Delta(u,u')=\frac{1}{|\det((\mathrm{d}\Phi_{\omega}^{-1})(\omega^{\prime\prime}))||\det((\mathrm{d}\Phi_{\omega})(x'))|}|\widetilde{\eta}(x')|^{n-1}\frac{|\langle\omega,\mathrm{d}N^{-1}_{\omega'}(\langle\omega,\omega'\rangle\omega'-\omega)\rangle|}{|\langle\omega,\mathrm{d}N^{-1}_{\omega''}(\langle\omega'',\omega\rangle\omega''-\omega)\rangle|}\frac{K(u')}{K(u'')}.
\end{equation*}
By \eqref{etatildedef},
\begin{equation}\label{etatildeformula}
\eqalign{
    \displaystyle|\widetilde{\eta}(x')|&\displaystyle=\left|\left\langle \frac{\omega'-\langle\omega,\omega^{\prime}\rangle\omega}{1+\langle\omega,\omega^{\prime}\rangle},\frac{\omega^{\prime\prime}-\langle\omega^{\prime\prime},\omega\rangle\omega}{1+\langle\omega^{\prime\prime},\omega\rangle}\right\rangle\right|\frac{|1+\langle\omega,\omega^{\prime}\rangle|^{2}}{|\omega'-\langle\omega,\omega^{\prime}\rangle\omega|^{2}} \cr
    &\displaystyle= \left|\frac{\langle\omega',\omega^{\prime\prime}\rangle-\langle\omega,\omega^{\prime}\rangle\langle\omega,\omega^{\prime\prime}\rangle}{(1+\langle\omega^{\prime\prime},\omega\rangle)(1-\langle\omega,\omega^{\prime}\rangle)}\right|.\cr
    }
\end{equation}
By \eqref{jacstereo1}, \eqref{jacstereo2}, and \eqref{etatildeformula},
\begin{equation*}
    \eqalign{
\displaystyle\Delta(u,u')&=\displaystyle \left|\frac{\langle\omega',\omega^{\prime\prime}\rangle-\langle\omega,\omega^{\prime}\rangle\langle\omega,\omega^{\prime\prime}\rangle}{(1+\langle\omega^{\prime\prime},\omega\rangle)\cdot(1-\langle\omega,\omega^{\prime}\rangle)}\right|^{n-1}\left|\frac{1+\langle\omega'',\omega\rangle}{1+\langle\omega',\omega\rangle}\right|^{n-1}\frac{|\langle\omega,\mathrm{d}N^{-1}_{\omega'}(\langle\omega,\omega'\rangle\omega'-\omega)\rangle|}{|\langle\omega,\mathrm{d}N^{-1}_{\omega''}(\langle\omega'',\omega\rangle\omega''-\omega)\rangle|}\frac{K(u')}{K(u'')} \cr
&\displaystyle= \left(\frac{|N(u)\wedge N(u'')|}{|N(u)\wedge N(u')|}\right)^{n-1}\frac{|\langle P_{T_{u'}S}N(u),(\mathrm{d}N_{u'})^{-1}(P_{T_{u'}S}N(u))\rangle|}{|\langle P_{T_{u''}S}N(u),(\mathrm{d}N_{u''})^{-1}(P_{T_{u''}S}N(u))\rangle|}\frac{K(u')}{K(u'')}, \cr
    }
\end{equation*}
by \eqref{claimDelta4-19apr24}, \eqref{claimDelta5-19apr24} and by similar geometric considerations to those in Section \ref{subsectionJ}. This establishes \eqref{deltaformula-prop-19apr24}.
\subsection{Relating $J$ and $\Delta$} Here we establish \eqref{switch}, the ``switching property" of $\Delta$.
By \eqref{Jformula-prop-19apr24} we have
\begin{equation*}\label{ineqDelta1-28feb24}
    \eqalign{
\displaystyle\frac{J(u,u'')}{J(u,u')}
&\displaystyle=\left(\frac{|N(u)\wedge N(u')|}{|N(u)\wedge N(u'')|}\right)^{n-2}\left|\frac{\langle P_{T_{u''}S}N(u),(\mathrm{d}N_{u''})^{-1}(P_{T_{u''}S}N(u))\rangle}{\langle P_{T_{u'}S}N(u),(\mathrm{d}N_{u'})^{-1}(P_{T_{u'}S}N(u))\rangle}\right|\frac{|\langle u''-u',N(u')\rangle|}{|\langle u''-u',N(u'')\rangle|}\frac{K(u'')}{K(u')}.\cr
    }
\end{equation*}
Using the coplanarity condition \eqref{collision condition 2}, an elementary argument similar to that leading to \eqref{claimDelta3-19apr24} reveals that
\begin{equation*}
    \frac{|\langle u''-u',N(u')\rangle|}{|\langle u''-u',N(u'')\rangle|}=\frac{|P_{\langle\omega\rangle^{\perp}}N(u')|}{|P_{\langle\omega\rangle^{\perp}}N(u'')|} =\frac{(1-\langle\omega,\omega^{\prime}\rangle^{2})^{\frac{1}{2}}}{(1-\langle\omega,\omega''\rangle^{2})^{\frac{1}{2}}}=\frac{|N(u)\wedge N(u')|}{|N(u)\wedge N(u'')|},
\end{equation*}
from which \eqref{switch} follows.


\section{Surface-carried fractional integrals}\label{Sect:KS}
In this section we establish Lebesgue space bounds on the bilinear fractional integrals
$$
I_{S,s} (g_1,g_2)(u) := \int_S \frac{g_1(u') g_2( R_uu' )}{ |u' - R_uu'|^{s} } J(u,u') \mathrm{d}\sigma(u')
$$
arising in Section \ref{Sect:general submanifolds}.
\begin{remark}[Relation to classical fractional integral operators]
This is a surface-carried variant of the bilinear fractional integral operator 
$$
I_s(f_1,f_2)(x):=\int_{\mathbb{R}^d}\frac{f_1\left(x+\frac{y}{2}\right)f_2\left(x-\frac{y}{2}\right)}{|y|^s}dy
$$
that naturally arises when $S$ is the paraboloid (see Section \ref{Sect:para}), and has been studied by several authors; we refer to \cite{Graf} and \cite{KS}. 
\end{remark}
As indicated in Section \ref{Sect:general submanifolds}, the presence of the factor $J$ in the kernel implies that this operator is symmetric -- that is, $I_{S,s}(g_1,g_2)=I_{S,s}(g_2,g_1)$. It is also natural for geometric reasons, allowing for bounds that are independent of any lower bounds on the curvature of $S$. For example, we have
\begin{eqnarray}\label{quickL1}
    \begin{aligned}
\|I_{S,s}(f_1,f_2)\|_1&=\int_S\int_S \frac{f_1(u')f_2(R_u u')}{|u'-R_u u'|^s}J(u,u')\mathrm{d}\sigma(u)\mathrm{d}\sigma(u')\\&=\int_S\int_S\frac{f_1(u')f_2(u'')}{|u'-u''|^s}\mathrm{d}\sigma(u')\mathrm{d}\sigma(u'')\\
&\leq C_s\|f_1\|_2\|f_2\|_2,
    \end{aligned}
\end{eqnarray}
where $$C_s:=\sup_{u\in S}\int_S\frac{\mathrm{d}\sigma(u')}{|u-u'|^s}.$$ Evidently $C_s$ does not depend on any lower bound on the curvature of $S$.
More generally we have the following:
\begin{theorem}\label{l:EndpointKS}
Let $0<s<n-1$, $q\in [\frac12,1]$, and $S$ be as above. 
Then 
$$
\|I_{S,s}(g_1,g_2)\|_{L^q(S)} \lesssim Q(S)^{2(n-1)}\|g_1\|_{L^{2q}(S)} \|g_2\|_{L^{2q}(S)}, 
$$
where the implicit constant depends on $n,s,q$ and the diameter of $S$.
\end{theorem}
In order to prove Theorem \ref{l:EndpointKS} we adapt the argument of Kenig and Stein \cite{KS} from the Euclidean setting. 
\begin{proof}[Proof of Theorem \ref{l:EndpointKS}] For each dyadic scale $\lambda\lesssim\diam(S)$ we decompose $S$ into a collection $\Theta_\lambda$ of $\lambda$-caps $\theta$, noting that $|\theta| \sim \lambda^{n-1}$ for such a cap. 
    Performing a dyadic decomposition and using the embedding $\ell^q \subset \ell^1$, for $q\le 1$, we have that (recall that $u''=R_{u}u'$)
    \begin{align*}
        \int_S I_{S,s}(g_1,g_2)^q \mathrm{d}\sigma(u) 
        &\lesssim 
        \sum_{0<\lambda \lesssim\diam(S)} \lambda^{-qs} \int_S \Bigg( \int_{u'\in S: |u'-u''|\sim \lambda}g_1(u') g_2(u'' ) J(u,u') \mathrm{d}\sigma(u')  \Bigg)^{q} \mathrm{d}\sigma(u). 
    \end{align*}
    Next, we fix an arbitrary dyadic scale $\lambda$ and decompose 
    \begin{align*}
        \int_S \Bigg( \int_{u'\in S: |u'-u''|\sim \lambda}&g_1(u') g_2( u'' ) J(u,u') \mathrm{d}\sigma(u')  \Bigg)^q \mathrm{d}\sigma(u)\\
        &= 
        \sum_{\theta\in\Theta_\lambda} \int_\theta \Bigg( \int_{u'\in S: |u'-u''|\sim \lambda}g_1(u') g_2( u'' ) J(u,u') \mathrm{d}\sigma(u')  \Bigg)^q \mathrm{d}\sigma(u)\\
        &\lesssim 
        \lambda^{(n-1)(1-q)} 
        \sum_{\theta\in\Theta_\lambda}\Bigg( \int_\theta  \int_{u'\in S: |u'-u''|\sim \lambda}g_1(u') g_2( u'' ) J(u,u') \mathrm{d}\sigma(u')   \mathrm{d}\sigma(u)\Bigg)^q. 
    \end{align*}
    Here we used that $0<q\le 1$ once more. 
    Recall that $|u-u'|\lesssim Q(S) |u' - u''|$ for all $u,u'\in S$ by Proposition \ref{metricestimate}. 
    Thus if $u \in \theta\in\Theta_\lambda$ and $|u' - u''| \sim \lambda$, then $|u-u'|\lesssim Q(S)\lambda$ which means that $u' \in \theta^*$, where $\theta^*$ is an $O(Q(S))$ dilate of $\theta$. Similarly, $u'' \in \theta^*$. 
    Consequently,
    \begin{align*}
        \int_S \Bigg( \int_{u'\in S: |u'-u''|\sim \lambda}&g_1(u') g_2( u'' ) J(u,u') \mathrm{d}\sigma(u')  \Bigg)^q \mathrm{d}\sigma(u)\\
        &\lesssim 
        \lambda^{(n-1)(1-q)} 
        \sum_{\theta\in\Theta_\lambda}\Bigg( \int_S  \int_{S} g_1\mathbbm{1}_{\theta^*}(u') g_2\mathbbm{1}_{\theta^*}( u'' ) J(u,u') \mathrm{d}\sigma(u')   \mathrm{d}\sigma(u)\Bigg)^q\\
        &= 
        \lambda^{(n-1)(1-q)}
        \sum_{\theta\in\Theta_\lambda}\|g_1 \mathbbm{1}_{\theta^*}\|_{L^1(S)}^q \|g_2 \mathbbm{1}_{\theta^*}\|_{L^1(S)}^q \\
        &\lesssim 
        \lambda^{(n-1)(1-q)}
        \Bigg( \sum_{\theta\in\Theta_\lambda}\|g_1 \mathbbm{1}_{\theta^*}\|_{L^1(S)}^{2q} \Bigg)^{\frac12 }\Bigg( \sum_{\theta\in\Theta_\lambda}\|g_2 \mathbbm{1}_{\theta^*}\|_{L^1(S)}^{2q} \Bigg)^{\frac12 }\\
        &\lesssim 
        \lambda^{(n-1)(1-q)} (Q(S)\lambda)^{(n-1)\frac{2q}{p'}}
        \Bigg( \sum_{\theta\in\Theta_\lambda}  \|g_1 \mathbbm{1}_{\theta^*}\|_{L^p(S)}^{2q} \Bigg)^{\frac12 }\Bigg( \sum_{\theta\in\Theta_\lambda} \|g_2 \mathbbm{1}_{\theta^*}\|_{L^p(S)}^{2q} \Bigg)^{\frac12 },
    \end{align*}
    for $p\ge1$, where we have used that $$\int_S\int_S f(u)g(u'') J(u,u') \mathrm{d}\sigma(u)\mathrm{d}\sigma(u') = \|f\|_{L^1(S)}\|g\|_{L^1(S)}.$$ 
    Since $q\ge\frac12$ and $p=2q$, we obtain that 
    \begin{equation*}
    \eqalign{
        \displaystyle\int_S \Bigg( \int_{u'\in S: |u'-u''|\sim \lambda}&g_1(u') g_2( u'' ) J(u,u') \mathrm{d}\sigma(u')  \Bigg)^q \mathrm{d}\sigma(u) \cr
        &\displaystyle\lesssim 
        \lambda^{(n-1)(1-q)} (Q(S)\lambda)^{(n-1)\frac{2q}{p'}}
        Q(S)^{n-1}\|g_1 \|_{L^p(S)}^{q} \|g_2 \|_{L^p(S)}^{q} \cr      
   &\displaystyle=\lambda^{(n-1)(1-q)+(n-1)\frac{2q}{p'}} Q(S)^{2q(n-1)}\|g_1 \|_{L^p(S)}^{q} \|g_2 \|_{L^p(S)}^{q}, \cr
        }
    \end{equation*}
    since the set of dilated caps $\{\theta^{\ast}:\theta\in\Theta_{\lambda}\}$ covers $S$ with a $Q(S)^{n-1}$ overlap factor. The geometric series converges as long as $-qs + (n-1)( 1- q + \frac{2q}{p'} )>0$. Since $p=2q$, this is equivalent to $s<n-1$.  
\end{proof}

\section{Surface-carried maximal operators}\label{app}
Recall from Section \ref{Sect:general submanifolds} that the geometric Wigner distribution $W_S(g,g)$ possesses the marginal properties \eqref{genmarg1} and \eqref{margin2}. In the (superficially) more general polarised form these are the identities
\begin{equation}\label{genmarg1pol}
\int_{S}W_S(g_1,g_2)(u,P_{T_uS}x)\mathrm{d}\sigma(u)=\widehat{g_1\mathrm{d}\sigma}(x)\overline{\widehat{g_2\mathrm{d}\sigma}(x)}
\end{equation}
and
\begin{equation}\label{margin2pol}
\int_{T_uS}W_S(g_1,g_2)(u,v)\mathrm{d}v=g_1(u)\overline{g_2(u)}
\end{equation}
respectively.
While \eqref{genmarg1pol} is an elementary consequence of Fubini's theorem and the definition of the Jacobian $J$, the property \eqref{margin2pol} appears to be a little more delicate in general. In particular, for $g_1,g_2$ merely in $L^2$, the integral in identity \eqref{margin2pol} should be interpreted as a suitable pointwise limit -- see the forthcoming Proposition \ref{whatwemean}. As may be expected, a maximal analogue of the bilinear fractional integral operator $I_{S,s}$ of Section \ref{Sect:KS} naturally arises in our analysis.
For locally integrable functions $f_1,f_2:S\rightarrow \mathbb{R}_+$ and $0<\delta<1$ we define the ``averaging" operator
$$
A_{S,\delta}(f_1,f_2)(u)=\delta^{-(n-1)}\int_{|u'-R_uu'|<\delta}f_1(u')f_2(R_uu')J(u,u')\mathrm{d}\sigma(u'),
$$
and maximal operator
$$
M_S(f_1,f_2)(u)=\sup_{0<\delta<1}A_{S,\delta}(f_1,f_2)(u).
$$
\begin{remark}[Relation to classical maximal operators]
The operator $M_S$ is a surface-carried variant of the classical bi(sub)-linear Hardy--Littlewood maximal operator
$$
M(f_1,f_2)(x)=\sup_{\delta>0}\frac{1}{|B(0,\delta)|}\int_{B(0,\delta)}f_1\left(x+\frac{y}{2}\right)f_2\left(x-\frac{y}{2}\right)\mathrm{d}y
$$
on a Euclidean space.  
\end{remark}
We shall need the following estimate:
\begin{theorem}\label{maximal theorem}
    If $S$ is smooth, strictly convex and has finite curvature quotient $Q(S)$, then
    \begin{equation}\label{maximal inequality}
    M_{S}:L^2(S)\times L^2(S)\rightarrow L^{1,\infty}(S).
    \end{equation}
\end{theorem}
\begin{proof}
We begin by using the Cauchy--Schwarz inequality to write
\begin{eqnarray*}
    \begin{aligned}
A_{S,\delta}(f_1,f_2)(u)&\leq\left(\delta^{-(n-1)}\int_{|u'-R_uu'|<\delta}f_1(u')^2J(u,u')\mathrm{d}\sigma(u')\right)^{1/2}\\&\;\;\;\;\;\;\times\left(\delta^{-(n-1)}\int_{|u'-R_uu'|<\delta}f_2(R_uu')^2J(u,u')\mathrm{d}\sigma(u')\right)^{1/2}.
\end{aligned}
\end{eqnarray*}
Making the change of variables $R_uu'=u''$ in the second factor above, using Proposition \ref{jacprop}, and the fact that $R_uu''=u'$, we see that
\begin{eqnarray*}\begin{aligned}
\int_{|u'-R_uu'|<\delta}f_2(R_uu')^2J(u,u')\mathrm{d}\sigma(u')&=\int_{|u''-R_uu''|<\delta}f_2(u'')^2J(u,u')\Delta(u,u'')\mathrm{d}\sigma(u'')\\&=\int_{|u''-R_uu''|<\delta}f_2(u'')^2J(u,u'')\mathrm{d}\sigma(u'')\\&=\int_{|u'-R_uu'|<\delta}f_2(u')^2J(u,u')\mathrm{d}\sigma(u').
\end{aligned}
\end{eqnarray*}
Thus, $$M_S(f_1,f_2)(u)\leq M_S^1(f_1^2)(u)^{1/2}M_S^1(f_2^2)(u)^{1/2},$$
where $$M_S^1(f)(u):=\sup_{0<\delta<1}\delta^{-(n-1)}\int_{|u'-R_uu'|<\delta}f(u')J(u,u')\mathrm{d}\sigma(u').$$
Hence 
\begin{eqnarray*}
\begin{aligned}
    \lambda\left|\left\{u\in S:M_S(f_1,f_2)(u)>\lambda\right\}\right|&\leq\lambda\left|\left\{u\in S:M_S^1(f_1^2)(u)M_S^1(f_2^2)(u)>\lambda^2\right\}\right|\\&\leq \lambda\left|\left\{u\in S:M_S^1(f_1^2)(u)>\varepsilon\lambda\right\}\right|\\
    &+\lambda\left|\left\{u\in S:M_S^1(f_2^2)(u)>\varepsilon^{-1}\lambda\right\}\right|
    \end{aligned}
    \end{eqnarray*}
for all $\varepsilon>0$. We claim that the sublinear operator $M_S^1$ is of weak-type (1,1), and assuming this momentarily we have
\begin{eqnarray*}
\begin{aligned}
    \lambda\left|\left\{u\in S:M_S(f_1,f_2)(u)>\lambda\right\}\right|&\lesssim \varepsilon^{-1}\|f_1\|_2^2+\varepsilon\|f_2\|_2^2
\end{aligned}
\end{eqnarray*}
uniformly in $\varepsilon$. Optimising in $\varepsilon$ now yields the claimed weak-type bound on the bi-sublinear operator $M_S$. A similar argument in a Euclidean context may be found in \cite{Graf}.

It remains to establish that $M_S^1:L^1(S)\rightarrow L^{1,\infty}$, and we do this by applying the well-known abstract form of the classical Hardy--Littlewood maximal theorem presented in \cite{Stein}. To this end we let $B_\delta(u)=\{u'\in S: \rho(u,u')<\delta\}$, the ball in $S$ centred at $u$ with respect to the function $\rho(u,u'):=|u'-R_uu'|$. By Proposition \ref{metricestimate} it follows that $\rho$ is a quasi-distance, as defined in \cite{Stein} (Page 10). Specifically, we may quickly verify that (i) $\rho(x,y)=0\iff x=y$, (ii) $\rho(x,y)\leq c\rho(y,x)$, and (iii) $\rho(x,y)\leq c(\rho(x,z)+\rho(y,z))$, for some positive constant $c$ depending on $Q(S)$.
By the change of variables \eqref{cov} and an application of Proposition \ref{prop-Jtilde-09apr24},
\begin{equation}\label{ballmeasure}
|B_\delta(u)|=\int_{|\xi|\leq \delta}\widetilde{J}(u,u'(\xi))^{-1}\mathrm{d}\xi\leq \delta^{n-1},
\end{equation}
so that
$$M_S^1f(u)\leq\sup_{0<\delta<1}\frac{1}{|B_\delta(u)|}\int_{B_\delta(u)}f(u')J(u,u')\mathrm{d}\sigma(u').$$ 
Arguing as in the proof of \eqref{take}, we have
\begin{equation*}
        \displaystyle J(u,u')\lesssim Q(S)^{\frac{5(n-2)}{2}} \frac{|u''-u'|}{|N(u'')-N(u)|}  \sup_{p}\lambda_{n-1}(p),
    \end{equation*}
which by a further use of Proposition \ref{metricestimate} and the mean value theorem applied to the Gauss map, shows that
$J(u,u')$ is, up to a dimensional constant, bounded from above by a power of $Q(S)$. Consequently,
$$M_S^1f(u)\lesssim \sup_{0<\delta<1}\frac{1}{|B_\delta(u)|}\int_{B_\delta(u)}f(u')\mathrm{d}\sigma(u'),$$
where the implicit constant is permitted to depend on $Q(S)$.
It remains to show that the surface measure on $S$ is doubling with respect to the family of balls $B_\delta(u)$, as we may then apply the abstract Hardy--Littlewood maximal theorem of \cite{Stein} (see Page 37). By \eqref{ballmeasure} it suffices to show that 
$|B_\delta(u)|\geq cQ(S)^{n-1}\delta^{n-1}$, for some dimensional constant $c$. 
However, this follows from Proposition \ref{metricestimate} since
$$
B_\delta(u)\supseteq\{u'\in S: |u'-u|\lesssim Q(S)\delta\}.
$$
\end{proof}
\begin{remark}[$L^p$ estimates for $M_S$]
A minor modification of the arguments in the proof of Theorem \ref{maximal theorem} (a use of H\"older's inequality in place of the Cauchy--Schwarz inequality) shows that $M_S:L^{p_1}(S)\times L^{p_2}(S)\rightarrow L^q(S)$
whenever $p_1, p_2, q>1$ and $\tfrac{1}{p_1}+\tfrac{1}{p_2}=\tfrac{1}{q}$. Implicitly, and as in the statement of Theorem \ref{maximal theorem}, the bounds here depend on the dimension and $Q(S)$.
\end{remark}
Equipped with the above maximal theorem we may now clarify the marginal property \eqref{margin2pol}.
While we expect that \eqref{margin2pol} (suitably interpreted) holds for all of the submanifolds $S$ that we consider in this paper, our approach seems to require the additional assumption that 
\begin{equation}\label{sillylimit}
    \lim_{u'\rightarrow u}(\mathrm{d}R_u)_{u'}\;\;\mbox{ exists.}
\end{equation}
We note that \eqref{sillylimit} requires some interpretation since for each $u'\not=u$, the map $(\mathrm{d}R_u)_{u'}:T_{u'}S\rightarrow T_{u''}S$, and the limit should be interpreted as a linear transformation of $T_u S$. One way to do this is to parametrise $S$ by $T_uS$, upon which the map $R_u$ may be parametrised by a map $y_u$ on the fixed domain $T_uS$. We clarify this technical point in the arguments that follow. The local statement \eqref{sillylimit} appears to be an extremely mild assumption. It is straightforward to verify for parabolic $S$, and since a smooth strictly convex surface is locally parabolic (by Taylor's theorem), one might reasonably expect it to be verifiable in general. 
\begin{proposition}\label{whatwemean} Let $S$ be smooth and strictly convex.
Suppose $\chi$ is a Schwartz function on $T_uS$ with $\chi(0)=1$, and $\chi_r(v)=\chi(v/r)$ for each $r>0$. Then for compactly supported $g_1,g_2\in L^2(S)$,
$$
\int_{T_uS}W_S(g_1,g_2)(u,v)\chi_r(v)\mathrm{d}v\rightarrow g_1(u)\overline{g_2(u)}
$$
as $r\rightarrow\infty$ for almost every $u\in S$. Moreover, if $g_1,g_2$ are continuous then this convergence holds at all points $u$.
\end{proposition}
Before we turn to the proof of Proposition \ref{whatwemean}, we state a lemma whose (somewhat technical) proof we leave to the end of the section.
\begin{lemma}\label{limJ}
If the limit \eqref{sillylimit} exists then for each $u\in S$,
$$\lim_{u'\rightarrow u}\widetilde{J}(u,u')=2^{n-1}$$
and
$$\lim_{\substack{u'\rightarrow u\\u'-u''\in\langle\omega\rangle}}J(u,u')=2^{n-1}$$
for each $\omega\in T_uS\backslash\{0\}$.
\end{lemma}
\begin{proof}[Proof of Proposition \ref{whatwemean}]
We begin by writing
\begin{eqnarray}\label{looking for maximal}
    \begin{aligned}
        \int_{T_uS}W_S(g_1,g_2)(u,v)\chi_r(v)\mathrm{d}v&=\int_{T_uS}\int_S g_1(u')\overline{g_2(R_uu')}e^{-2\pi i v\cdot(u'-R_uu')}J(u,u')\mathrm{d}\sigma(u')\chi_r(v)\mathrm{d}v\\
        &=\int_S g_1(u')\overline{g_2(R_uu')}\widehat{\chi}_r(u'-R_uu')J(u,u')\mathrm{d}\sigma(u')\\
        &=:\mathcal{A}_{S,r}(g_1,g_2)(u).
    \end{aligned}
\end{eqnarray}
Since $\widehat{\chi}$ is a bump function, it follows that 
$\mathcal{M}_S(g_1,g_2)(u)\lesssim M_S(|g_1|,|g_2|)(u)$ where
$$
\mathcal{M}_S(g_1,g_2)(u):=\sup_{r>1}|\mathcal{A}_{S,r}(g_1,g_2)(u)|.
$$
Consequently 
\begin{equation}\label{nicemax}
\mathcal{M}_S:L^2(S)\times L^2(S)\rightarrow L^{1,\infty}(S),
\end{equation}
by Theorem \ref{maximal theorem}.
Proposition \ref{whatwemean} requires us to show that
\begin{equation}\label{billeb}
\mathcal{A}_{S,r}(g_1,g_2)(u)\rightarrow g_1(u)\overline{g_2(u)}\;\mbox{ for almost every }\;u\in S.
\end{equation}
The first step, which uses a minor variant of a standard argument in the setting of sublinear maximal operators (see for example \cite{SW}), is to use the maximal estimate \eqref{nicemax} to reduce to the case of continuous $g_1,g_2$. We leave this classical exercise to the reader.
Suppose now that $g_1,g_2$ are continuous functions. It suffices to show that
\begin{equation}\label{11}
\mathcal{A}_{S,r}(1,1)(u):=\int_S \widehat{\chi}_r(u'-R_uu')J(u,u')\mathrm{d}\sigma(u')\rightarrow 1.
\end{equation}
Invoking the change of variables \eqref{cov} and using polar coordinates in $T_uS$ we have
\begin{eqnarray*}
    \begin{aligned}
\mathcal{A}_{S,r}(1,1)(u)&=\int_{T_uS}\widehat{\chi}_r(\xi)\frac{J(u,u'(\xi))}{\widetilde{J}(u,u'(\xi))}\mathrm{d}\xi\\
&=\int_0^\infty \int_{\mathbb{S}^{n-2}(T_uS)}r^{n-1}\widehat{\chi}(rt\omega)\frac{J(u,u'(t\omega))}{\widetilde{J}(u,u'(t\omega))}\mathrm{d}\sigma(\omega)t^{n-2}\mathrm{d}t\\
&=\int_0^\infty \int_{\mathbb{S}^{n-2}(T_uS)}\widehat{\chi}(s\omega)\frac{J(u,u'(r^{-1}s\omega))}{\widetilde{J}(u,u'(r^{-1}s\omega))}\mathrm{d}\sigma(\omega)\mathrm{d}s,
        \end{aligned}
        \end{eqnarray*}
        where $\mathbb{S}^{n-2}(T_uS)$ denotes the unit sphere in $T_uS$.
The limit \eqref{11} now follows by Lemma \ref{limJ} since $u'(r^{-1}s\omega)\rightarrow u$ as $r\rightarrow \infty$, while $u'(r^{-1}s\omega)-R_uu'(r^{-1}s\omega)=r^{-1}s\omega\in \langle\omega\rangle$.
\end{proof}
It remains to prove Lemma \ref{limJ}.
\begin{proof}[Proof of Lemma \ref{limJ}]
We begin by clarifying the hypothesis \eqref{sillylimit}, and showing that this limit must actually equal $-I$, where $I$ denotes the identity on $T_uS$. This reflects a crucial ``limiting symmetry" of the configuration of points $u,u',u''$ as $u'\rightarrow u$.
By translation and rotation invariance we may suppose that $u=0$ and $S=\{(x',\phi(x')):x'\in X\}$, for some smooth real-valued function $\phi$ on a subset $X$ of $T_uS$ satisfying $\nabla\phi(0)=0$ and $\hess(\phi)(x')>_{pd}0$ for all $x'$. The map $R:=R_u$ then takes the form $R(x',\phi(x'))=(x'',\phi(x''))$, for some unique $x''\in T_uS$ satisfying 
\begin{equation}\label{cond1-6june24}
    \phi(x'')=\phi(x')
\end{equation}
and
\begin{equation}\label{cond2-6june24}
    \frac{\nabla\phi(x'')}{|\nabla\phi(x'')|}=-\frac{\nabla\phi(x')}{|\nabla\phi(x')|}.
\end{equation} 
Observe that \eqref{cond1-6june24} follows by \eqref{collision condition 1} and \eqref{cond2-6june24} is a consequence of \eqref{footmeasuringdevice}. Writing $x''=y(x')$ allows us to interpret \eqref{sillylimit} as the existence of the limit $\mathrm{d}y_0:=\lim_{x'\rightarrow 0}\mathrm{d}y_{x'}:T_{u}S\rightarrow T_{u}S$.
In order to show that $\mathrm{d}y_0=-I$, we fix $v\in T_{u}S$ and let $x_{k}'\rightarrow 0$ be a sequence in $T_{u}S$ satisfying
$$\frac{\nabla\phi(y(x_{k}'))}{|\nabla\phi(y(x_{k}'))|}=v$$
for all $k$. This sequence exists as the Gauss maps $\widetilde{N}$ of the sections $\mathcal{S}_{u,u'}$ (see \eqref{sectionsdef}) are bijections. Differentiating \eqref{cond1-6june24} at the points of this sequence, we have
\begin{equation*}
    \mathrm{d}y(x_{k}')^{\top}(\nabla\phi(y(x_{k}')))=\nabla\phi(x').
\end{equation*}
Using \eqref{cond2-6june24},
\begin{equation*} \mathrm{d}y(x_{k}')^{\top}\left(\frac{\nabla\phi(x_{k}')}{|\nabla\phi(x_{k}')|}\right)=-\frac{|\nabla\phi(x_{k}')|}{|\nabla\phi(y(x_{k}'))|}\frac{\nabla\phi(x_{k}')}{|\nabla\phi(x_{k}')|},
\end{equation*}
which implies
\begin{equation}\label{eq1-6june24} \mathrm{d}y(x_{k}')^{\top}(v)=-\frac{|\nabla\phi(x_{k}')|}{|\nabla\phi(y(x_{k}'))|}v
\end{equation}
for all $k\in\mathbb{N}$. By the mean-value inequality,
\begin{equation*}
    \frac{|\nabla\phi(x_{k}')|}{|\nabla\phi(y(x_{k}'))|}\leq\frac{\sup\|\hess\phi\|_{\infty}}{\inf\|\hess\phi\|_{\infty}}\frac{|x_{k}'|}{|y(x_{k}')|}\leq \frac{\sup\|\hess\phi\|_{\infty}}{\inf\|\hess\phi\|_{\infty}}\frac{1}{\|\mathrm{d}y(c_{k})\|_{\infty}}
\end{equation*}
for some $c_{k}$ with $c_{k}\rightarrow 0$. On the other hand, $y(y(x_{k}'))=x_{k}'$ (recall that $R_{u}(R_{u}u')=u'$), hence  $\mathrm{d}y(y(x_{k}'))\circ\mathrm{d}y(x_{k}')=I$, which gives $\mathrm{d}y_0^{2}=I$, therefore $\|\mathrm{d}y(c_{k})\|_{\infty}$ does not approach $0$ and the sequence
$$\frac{|\nabla\phi(x_{k}')|}{|\nabla\phi(y(x_{k}'))|}$$
is bounded. By passing to a subsequence and by taking limits, we conclude from \eqref{eq1-6june24} that
\begin{equation*}
    \mathrm{d}y_0^{\top}(v)=-Lv
\end{equation*}
for some positive real number $L$ and for all $v\in T_{u}S$. On the other hand, since $\mathrm{d}y_0^{2}=I$, the only possible eigenvalues of $\mathrm{d}y_0$ are $\pm 1$, hence $\mathrm{d}y_0=-I$. Finally, 
taking the limit as $u'\rightarrow u$ in the first identity of \eqref{eq1-28may24} gives
\begin{equation}
    \lim_{u'\rightarrow u}\widetilde{J}(u,u')=2^{n-1}.
\end{equation}

Turning to the limiting identity for $J$, we first establish some bounds relating to the limiting arrangements of the points $u, u', u''$ and their normals $N(u), N(u'), N(u'')$, beginning with 
\begin{equation}\label{limitingmidpoint0}
    u'+u''-2u=o(|u-u'|).
\end{equation}
To see this (recalling that we are supposing $u=0$) observe that $u'+u''=(x'+y(x'), 2\phi(x'))$, and since $\phi(x')=O(|x'|^2)$, it remains to show that $h(x'):=x'+y(x')=o(|x'|)$. By the mean value theorem it suffices to observe that
$\mathrm{d}h_{x'}=I+\mathrm{d}y_{x'}=o(1)$ as $x'\rightarrow 0$, since $\mathrm{d}y_{x'}\rightarrow -I$.
A similar, albeit lengthier argument reveals that 
\begin{equation}\label{limitingmidpoint}
N(u')+N(u'')-2N(u)=o(|u-u'|).
\end{equation}
Recalling the formula for $J(u,u')$, we observe first that the factor
$$
\frac{|N(u')\wedge N(u'')|}{|N(u')\wedge N(u)|}=\frac{2|N(u')\wedge N(u)|+o(|u'-u|)}{|N(u')\wedge N(u)|}\rightarrow 2
$$
as $u'\rightarrow u$. Here we are also using \eqref{cone of normals}, which tells us that $|N(u')\wedge N(u)|\sim |u'-u|$.
It remains to show that for each unit vector $\omega\in T_uS$,
\begin{equation}\label{lastthing}
\left|\frac{\langle u''-u',N(u'')\rangle}{\langle P_{T_{u''}S}N(u),(\mathrm{d}N_{u''})^{-1}(P_{T_{u''}S}N(u))\rangle}\right|\rightarrow 2
\end{equation}
as $u'\rightarrow u$ with $u'-u''\in\langle\omega\rangle$.
Noting that $\langle u''-u', N(u'')\rangle=\langle u''-u', P_{T_uS} N(u'')\rangle$,
by \eqref{limitingmidpoint0} 
we are reduced to showing that
$$
\lim_{\substack{u'\to u\\u'-u''\in\langle\omega\rangle}} 
\frac{ \langle u''-u, P_{T_uS} N(u'') \rangle }{ \langle P_{T_{u''}S} N(u), (\mathrm{d}N_{u''})^{-1} P_{T_{u''}S} N(u) \rangle } =1.
$$
By symmetry, we may replace $u''$ by $u'$ here, so that the objective is to show that 
\begin{equation}\label{e:Goal6/9}
\lim_{\substack{u'\to u\\u'-u''\in\langle\omega\rangle}}
\frac{ \langle u'-u, P_{T_uS} N(u') \rangle }{ \langle P_{T_{u'}S} N(u), (\mathrm{d}N_{u'})^{-1} P_{T_{u'}S} N(u) \rangle } =1.
\end{equation}
To this end we Taylor expand $N(u')$ about $0$ via the parametrisation $u'=(x',\phi(x')) =: \Phi(x')$ to obtain
\begin{align*}
N(u') 
&= N\circ \Phi(x') = N\circ \Phi (0) + \mathrm{d}(N\circ \Phi)_0 x' + O(|x'|^2)\\
&= 
N(u) + 
(\mathrm{d}N)_u\circ (\mathrm{d}\Phi)_0 x' + O(|x'|^2) \\
&= 
N(u) + 
(\mathrm{d}N)_u x' + O(|x'|^2), 
\end{align*}
where we have used that 
$(\mathrm{d}\Phi)_{x'} = \begin{pmatrix} {\rm id}_{\mathbb{R}^{n-1}} & \mathbf{0} \\ 
\nabla_{n-1} \phi(x') & 0\end{pmatrix}$ and $\nabla\phi(0)=0$. 
Thus, in view of the fact that $|x'|= O(|u'-u|)$ we have 
$$
x' = (\mathrm{d}N)_u^{-1} \big( N(u') - N(u) + O(|u'-u|^2) \big). 
$$
The numerator of \eqref{e:Goal6/9} now becomes
\begin{align*}
    \langle u'-u, P_{T_uS} N(u') \rangle
    &=
    \langle P_{T_uS}(u'-u), P_{T_uS} N(u') \rangle \\
    &= 
    \langle x', P_{T_uS} N(u') \rangle \\
    &= 
    \big\langle 
    (\mathrm{d}N)_u^{-1} \big( N(u') - N(u) + O(|u'-u|^2) \big), 
    P_{T_uS} N(u')
    \big\rangle. 
\end{align*}
Note that
$$
P_{T_uS} N(u') = N(u')-N(u) + O(|u'-u|^2),
$$
and so
$$
\langle u'-u, P_{T_uS} N(u') \rangle
    = 
    \big\langle 
    (\mathrm{d}N)_u^{-1} \big( N(u') - N(u) + O(|u'-u|^2) \big), 
    \big( N(u') - N(u) + O(|u'-u|^2) \big) 
    \big\rangle. 
$$
This is now similar to the denominator of \eqref{e:Goal6/9}. 
In fact, 
$$
P_{T_{u'}S} N(u) = - (N(u')-N(u)) + O(|u'-u|^2),
$$
and so
\begin{align*}
    &\frac{ \langle u'-u, P_{T_uS} N(u') \rangle }{ \langle P_{T_{u'}S} N(u), (\mathrm{d}N_{u'})^{-1} P_{T_{u'}S} N(u) \rangle }\\
    &\;\;\;\;\;\;\;\;= 
    \frac{ \big\langle 
    (\mathrm{d}N)_u^{-1} \big( N(u') - N(u) + O(|u'-u|^2) \big), 
    \big( N(u') - N(u) + O(|u'-u|^2) \big) 
    \big\rangle }{ \big\langle 
    (\mathrm{d}N)_{u'}^{-1} \big( N(u') - N(u) + O(|u'-u|^2) \big), 
    \big( N(u') - N(u) + O(|u'-u|^2) \big) 
    \big\rangle }. 
\end{align*}
Further, from \eqref{limitingmidpoint} we have 
$$
N(u') - N(u)
= 
\frac12( N(u') - N(u'') ) + o(|u-u'|), 
$$
and hence,
\begin{align*}
    &\frac{ \langle u'-u, P_{T_uS} N(u') \rangle }{ \langle P_{T_{u'}S} N(u), (\mathrm{d}N_{u'})^{-1} P_{T_{u'}S} N(u) \rangle }\\
    &\;\;\;\;\;\;\;\;= 
    \frac{ \big\langle 
    (\mathrm{d}N)_u^{-1} \big( N(u') - N(u'') + o(|u'-u|) \big), 
    \big( N(u') - N(u'') + o(|u'-u|) \big) 
    \big\rangle }{ \big\langle 
    (\mathrm{d}N)_{u'}^{-1} \big( N(u') - N(u'') + o(|u'-u|) \big), 
    \big( N(u') - N(u'') + o(|u'-u|) \big) 
    \big\rangle }.
\end{align*}
Consequently, if $$\lim_{\substack{u'\to u\\ u'-u''\in\langle\omega\rangle}} \frac{N(u') - N(u'')}{|N(u') - N(u'')|}$$ exists, then  \eqref{e:Goal6/9} follows. 
Here we have also appealed to the fact that
$$
|N(u')-N(u'')| = |(\mathrm{d}N)_u(u'-u'') + O(|u-u'|^2)| 
\gtrsim |u-u'|.
$$
Arguing similarly using Taylor's theorem, we also have 
$$
N(u'')-N(u) = (\mathrm{d}N)_u x'' + O(|x''|^2),  
$$
from which it follows that 
$$
N(u') - N(u'') 
= 
(\mathrm{d}N)_u x' 
-(\mathrm{d}N)_u x'' + O(|x'|^2) + O(|x''|^2) 
=
(\mathrm{d}N)_u ( u'-u'' ) + O(|u-u'|^2), 
$$
and so
$$
    \frac{N(u') - N(u'')}{|N(u') - N(u'')|}
    =   \frac{(\mathrm{d}N)_u ( u'-u'' ) + O(|u-u'|^2)}{|(\mathrm{d}N)_u ( u'-u'' )| + O(|u-u'|^2)}
    =\frac{(\mathrm{d}N)_u ( \omega ) + O(|u-u'|)}{|(\mathrm{d}N)_u ( \omega )| + O(|u-u'|)},
$$
which converges (to $(\mathrm{d}N)_u\omega/|(\mathrm{d}N)_u\omega|$) as
as $u'\rightarrow u$ with $u'-u''\in\langle\omega\rangle$, as required.
\end{proof}

\section{Tomographic constructions}\label{Sect:tomographic}
In this section we show that the explicit geometric Wigner distributions from Section \ref{Sect:general submanifolds} may be constructed \textit{tomographically} from the corresponding extension operators, at least when $n=2$. This is motivated by the tomographic approach to weighted extension inequalities developed in \cite{BN, BNS}.
For the submanifolds $S$ considered in Section \ref{Sect:general submanifolds}, we saw that the natural tomographic transform is the $S$-parametrised X-ray transform $X_Sw(u,v):=Xw(N(u),v)$. Here $X$ denotes the standard X-ray transform and $N$ the Gauss map of $S$. We remark that if the Gauss map is bijective, such as when $S$ is strictly convex and closed, the operator $X_S$ is easily seen to inherit the inversion formula
$$
c_nX_S^*K(-\Delta_v)^{1/2}X_S\psi=\psi
$$
from the classical inversion formula for $X$; here $\psi$ is a suitably well-behaved function and $K(u)$ is the Gaussian curvature of $S$ at a point $u$ (acting multiplicatively). This suggests the following:
\begin{proposition}\label{tomrep} Let $\Phi$ be a smooth bump function on $\mathbb{R}^2$ such that $\Phi(0)=1$, and let $\Phi_\lambda(x)=\Phi(x/\lambda)$ for each $\lambda>0$.
If $S$ is a strictly convex smooth curve in the plane then
$$
\lim_{\lambda\rightarrow\infty}K(u)(-\Delta_v)^{1/2}X_S(\Phi_\lambda|\widehat{g\mathrm{d}\sigma}|^2)(u,v)=W_S(g,g)(u,v)
$$
for all compactly-supported smooth functions $g$ on $S$.
\end{proposition}
\begin{remark}[Phase-space tomographic methods in optics]\label{phasetom}
    This spatial tomographic construction, which in the particular case of the circle is somewhat implicit in \cite{BNS}, appears to be quite different from the \textit{phase-space} tomographic constructions of Wigner distributions that have proved effective in optics. There it is observed that the phase-space X-ray transform applied to the Wigner distribution (referred to as the Radon--Wigner transform) identifies its marginal distributions in all directions, and that these marginals are natural hybrids of the coordinate marginals, involving the fractional Fourier transform.
    The Wigner distribution is then (re)constructed by an application of the classical (left) inverse X-ray transform; see for example \cite{Bert, Al}.
\end{remark}
\begin{remark}
The cut-off $\Phi_\lambda$ is included in the statement of Proposition \ref{tomrep} as $X_S(|\widehat{g\mathrm{d}\sigma}|^2)$ is not in general defined for $g\in L^2(S)$ (unless there is a suitable transversality property satisfied -- see \cite{BNS}). This may already be seen when $S=\mathbb{S}^1$ and $g\equiv 1$, as then $|\widehat{g\mathrm{d}\sigma}(x)|^2$ is comparable to $(1+|x|)^{-1}$ on sufficiently large portions of $\mathbb{R}^2$. 
\end{remark}
\begin{proof}[Proof of Proposition \ref{tomrep}]
    A routine (distributional) argument, using the well-known fact that $$\mathcal{F}_v(Xf)(\omega,\xi)=\widehat{f}(\xi),\;\;\xi\in\langle\omega\rangle^\perp,$$
    reveals that
    \begin{eqnarray}\label{tomrep1}
    \begin{aligned} 
        (-\Delta_v)^{1/2}X_S( \Phi_\lambda|\widehat{g\mathrm{d}\sigma}|^2)(u,v)&
        =\int_{T_u S}e^{2\pi i\xi\cdot v}|\xi|\widehat{\Phi}_\lambda\ast (g\mathrm{d}\sigma)*(\widetilde{g\mathrm{d}\sigma})(\xi)\mathrm{d}\xi.
      \end{aligned}
    \end{eqnarray}
    In order to take the limit as $\lambda\to \infty$ it suffices, by the dominated convergence theorem, to show that 
    \begin{equation}\label{uniformbound}
\sup_{\lambda\geq 1}|\xi| |\widehat{\Phi}_\lambda|\ast (g\mathrm{d}\sigma)*(\widetilde{g\mathrm{d}\sigma})(\xi)\lesssim(1+|\xi|)^{-N}
 \end{equation}
for some sufficiently large $N\in\mathbb{N}$.
    This may be seen by first appealing to the strict convexity of $S$, along with the assumed properties of $g$, to show that $(g\mathrm{d}\sigma)*(\widetilde{g\mathrm{d}\sigma})(\xi)\lesssim |\xi|^{-1}\mathbbm{1}_B(\xi)$ for some ball $B\subset\mathbb{R}^2$; see \cite[Section 2]{DO} for the appropriate detailed computations. The estimate \eqref{uniformbound} then follows using the rapid decay of $\widehat{\Phi}$. Taking this limit, it follows that
    \begin{align*}
    \lim_{\lambda\to\infty}(-\Delta_v)^{1/2}&X_S( \Phi_\lambda|\widehat{g\mathrm{d}\sigma}|^2)(u,v)
        \\&=\int_{T_u S}e^{2\pi i\xi\cdot v}|\xi| (g\mathrm{d}\sigma)*(\widetilde{g\mathrm{d}\sigma})(\xi)\mathrm{d}\xi
    \\&=\int_S\int_Sg(u')\overline{g(u'')}e^{2\pi i(u'-u'')\cdot v}|u'-u''|\delta((u'-u'')\cdot N(u))\mathrm{d}\sigma(u'')\mathrm{d}\sigma(u').
    \end{align*}
    Now, for fixed $u,u'$ the function
    $u''\mapsto (u'-u'')\cdot N(u)$ vanishes if and only if either $u''=u'$ or $u''=R_u u'$, as defined in Section \ref{Sect:general submanifolds}, and so it remains to establish the formula 
    \begin{equation}\label{mass}
    \int_S|u'-u''|\delta((u'-u'')\cdot N(u))\mathrm{d}\sigma(u'')=\frac{|u'-R_u u'|}{|N(u)\wedge N(R_u u')|}
    \end{equation}
    whenever $u'\not= u$; see Remark \ref{interpreting J}.
    Making the change of variables $u'''=u''-R_u u'$ (we stress that $u''$ is the variable of integration in \eqref{mass} rather than a simplified notation for $R_{u}u'$), and using $\mathcal{H}^1$ to denote 1-dimensional Hausdorff measure in the plane, we have that
    \begin{eqnarray*}\begin{aligned}
    \int_S|u'-u''|\delta((u'-u'')&\cdot N(u))\mathrm{d}\sigma(u'')\\&=\int_{S-\{R_u u'\}}|u'-R_u u'-u'''|\delta(u'''\cdot N(u))\mathrm{d}\mathcal{H}^1(u''')\\
    &=|u'-R_u u'|\lim_{\varepsilon\rightarrow 0}\frac{1}{2\varepsilon}\mathcal{H}^1\left(\{u'''\in (S-\{R_u u'\}):|u'''\cdot N(u)|<\varepsilon\}\right),
    \end{aligned}
    \end{eqnarray*}
    from which \eqref{mass} follows from the smoothness of $S$ by elementary geometric considerations.
\end{proof}

\begin{remark}[Stein's inequality as a lower bound on the X-ray transform]
Stein's inequality \eqref{Steinvgen} may of course be interpreted as a certain \textit{lower bound} on the X-ray transform $X_S$. Here we make some contextual remarks relating to this in the setting of the paraboloid, where the corresponding inequality \eqref{Steinpara} takes the form 
    \begin{eqnarray}\label{dada}
    \begin{aligned}
\int_{\mathbb{R}^d\times\mathbb{R}}|u(x,t)|^2w(x,t)\mathrm{d}x\mathrm{d}t\lesssim
\int_{\mathbb{R}^d}\|\rho^*w(\cdot,v)\|_{L^\infty(\mathbb{R}^d)}|\widehat{u}_0(v)|^2\mathrm{d}v,
\end{aligned}
\end{eqnarray}
recalling the caveat in Remark \ref{Remark:failureStrichartz}.
Somewhat similar-looking lower bounds may be obtained from the \textit{adjoint Loomis--Whitney inequality} introduced in \cite{BT}. Arguing as in \cite[Section 8]{BT} it follows that
\begin{equation}\label{bt}
C(|\widehat{u}_0|^2)\|w\|_{L^p_{x,t}}\leq\left(\int_{\mathbb{R}^d}\|\rho^*w(\cdot,v)\|_{L^q(\mathbb{R}^d)}^r|\widehat{u}_0(v)|^2\mathrm{d}v\right)^{1/r}
\end{equation}
whenever $w\geq 0$, $0<p,q\leq 1$, $r>0$ and $\tfrac{1}{d+1}\left(\tfrac{1}{q}-1\right)=\tfrac{1}{d}\left(\tfrac{1}{p}-1\right)$. Here
$$
C(|\widehat{u}_0|^2):=\left(\int_{(\mathbb{R}^d)^{d+1}}\left|\det\left(\begin{array}{ccc}
1& \cdots  & 1\\
2v_1 &
\cdots & 2v_{d+1}\\
\end{array}\right)\right|^{\frac{(d+1)r}{d}\left(\frac{1}{p}-1\right)}|\widehat{u}_0(v_1)|^2\cdots|\widehat{u}_0(v_{d+1})|^2\mathrm{d}v\right)^{\frac{1}{(d+1)r}}.
$$
Of course \eqref{bt}, while superficially similar, is numerologically very different from \eqref{dada}, and also phenomenologically: $L^p$ norms below $L^1$ reflect spread rather than concentration. In particular, raising  
\eqref{bt} to the $r$th power, setting $r=q$ and taking a limit as $p\rightarrow 0$ one obtains
\begin{eqnarray}\label{vis stein}
\begin{aligned}
\Biggl(\int_{(\mathbb{R}^d)^{d+1}}\left|\det\left(\begin{array}{ccc}
1& \cdots  & 1\\
2v_1 &
\cdots & 2v_{d+1}\\
\end{array}\right)\right|&|\widehat{u}_0(v_1)|^2\cdots|\widehat{u}_0(v_{d+1})|^2\mathrm{d}v\Biggr)^{\frac{1}{d+1}}|\supp w|^{\frac{d}{d+1}}\\&\leq \int_{\mathbb{R}^d}|\supp \rho^*w(\cdot,v)||\widehat{u}_0(v)|^2\mathrm{d}v.
\end{aligned}
\end{eqnarray}
It was observed in \cite{BI} (see also \cite{BNS}) that the left-hand side of \eqref{vis stein} (and the expression $C(|\widehat{u}_0|^2)$ in general) has a space-time formulation in terms of $u$, emphasising further the parallels with \eqref{dada}. 
The factor $|\supp\rho^*w(\cdot,v)|$ is a measure of the ``visibility" of $w$ in the space-time direction $(-2v,1)$, making \eqref{vis stein} a certain visibility version of \eqref{dada}.
Similar remarks may be made for more general surfaces $S$ and are left to the interested reader.
\end{remark}

\section{Applications to a variant of Flandrin's conjecture}\label{Sect:Flandrin}
The phase-space integral formula \eqref{dmv identity} exposes a formal similarity between the parabolic Mizohata--Takeuchi inequality \eqref{MTpara} (or its local substitute \eqref{MTparaeps}) 
and a variant of a conjecture of Flandrin \cite{Flan} from time-frequency analysis. This conjecture, which was formulated in \cite{Duy}, states that
\begin{equation}\label{flandrin}
\iint_{K}W(u_0,u_0)(x,v)\mathrm{d}x\mathrm{d}v\lesssim\|u_0\|_2^2
\end{equation}
uniformly over all convex subsets $K$ of $\mathbb{R}^d\times\mathbb{R}^d$. This is a weakened form of the original conjecture that was made with constant $1$, following a recent counterexample in \cite{Duy}; we refer to \cite{Lerner} for further discussion, along with a number of supporting results.

In this section we show that the basic methods of this paper are effective towards \eqref{flandrin} by establishing a version of it in the plane involving an arbitrarily small loss in terms of the Lebesgue measure of $K$. We then show how  \eqref{flandrin} implies the parabolic Mizohata--Takeuchi inequality \eqref{MTpara} for a special class of weights.
\begin{theorem}\label{thm:flan} 
For each $\varepsilon>0$ there exists a constant $C_\varepsilon<\infty$ such that
\begin{equation}\label{Flandrinloss}
\iint_K W(u_0,u_0)(x,v)\mathrm{d}x\mathrm{d}v\leq C_\varepsilon |K|^{\varepsilon}\|u_0\|_2^2
\end{equation}
for all convex subsets $K$ of $\mathbb{R}^2$.
\end{theorem}
\begin{proof}
Arguing as in Section \ref{Sect:para}, and indeed Sections \ref{Sect:Helm} and \ref{Sect:general submanifolds}, by the Cauchy--Schwarz inequality and the duality of the homogeneous Sobolev spaces $\dot{H}^s$ and $\dot{H}^{-s}$, we have 
\begin{equation}\label{halfway}
\iint_K W(u_0,u_0)(x,v)\mathrm{d}x\mathrm{d}v\leq \int_{\pi_2(K)}\|W(u_0,u_0)(\cdot,v)\|_{\dot{H}^{-s}_x}\|\mathbbm{1}_K(\cdot,v)\|_{\dot{H}^s_x}\mathrm{d}v,
\end{equation}
for each $s<\frac{1}{2}$, where $\pi_2(K)\subseteq\mathbb{R}$  is the projection of $K$ onto the $v$-axis. We now compute both of these Sobolev norms explicitly.

To compute the $\dot{H}^s_x$ norm, we fix $v$ and observe that by the convexity of $K$, $$\mathbbm{1}_K(\cdot,v)=\mathbbm{1}_{[a,b]}$$ almost everywhere for some real numbers $a, b$. Since $$|\widehat{\mathbbm{1}}_{[a,b]}(\xi)|=\left|\frac{\sin(\pi(b-a)\xi)}{\pi\xi}\right|,$$
\begin{eqnarray*}
    \begin{aligned}
\|\mathbbm{1}_K(\cdot,v)\|_{\dot{H}^s_x}^2=\int_{\mathbb{R}}|\xi|^{2s}\left(\frac{\sin(\pi(b-a)\xi)}{\pi\xi}\right)^2\mathrm{d}\xi=(b-a)^{1-2s}\int_{\mathbb{R}}|\xi|^{2s}\left(\frac{\sin(\pi\xi)}{\pi\xi}\right)^2\mathrm{d}\xi\leq c_s\diam_1(K)^{1-2s},
    \end{aligned}
\end{eqnarray*}
with finite constant $c_s$ since $s<\frac{1}{2}$. Here $\diam_1(K)$ is the diameter of $K$ in the first coordinate direction.

To compute the $\dot{H}^{-s}_x$ norm we argue as in Section \ref{Sect:para}, and indeed Sections \ref{Sect:Helm} and \ref{Sect:general submanifolds}, to write
$$
\|W(u_0,u_0)(\cdot, v)\|_{\dot{H}^{-s}_x}=I_{2s}(|\widehat{u}_0|^2,|\widehat{u}_0|^2)(v)^{1/2},
$$
where 
$
I_s
$
is given by \eqref{KSdefhom}.
We estimate this term further by applying the weak-type estimate 
\begin{equation}\label{e:KSCopy}
    \| I_{s}(g,g)\|_{L^{q,\infty}(\mathbb{R})} 
    \lesssim 
    \| g\|_{L^1(\mathbb{R})}^2,
\end{equation}
from \cite{KS} (see also \cite{GK}), which holds whenever $s\in (0,1)$ and $\frac1q = 1+s$.
In particular, given $\varepsilon >0$ and writing $s_\varepsilon = \frac{1}2-\varepsilon$, we have
$$
\| I_{2s_\varepsilon} (g,g)^{1/2} \|_{L^{q_\varepsilon,\infty}(\mathbb{R})} 
= 
\| I_{2s_\varepsilon} (g,g) \|_{ L^{q_\varepsilon /2 , \infty} (\mathbb{R}) }^{1/2}
\le C_\varepsilon 
\|g\|_1,\quad {q_\varepsilon}:= \frac{1}{1 - \varepsilon}.
$$
With this in mind, we apply the Lorentz--H\"older inequality in \eqref{halfway} to write
\begin{align*}
    \iint_K W(u_0,u_0)(x,v) \mathrm{d}x\mathrm{d}v 
    &\leq
    \big\|  
    I_{2s_\varepsilon} ( |\widehat{u}_0|^2, |\widehat{u}_0|^2 )^{1/2} 
    \big\|_{L^{q_\varepsilon,\infty}(\mathbb{R})} 
    \big\| \| \mathbbm{1}_K(x,v) \|_{\dot{H}^{s_\varepsilon}_x} \big\|_{ L^{q_\varepsilon',1}( \pi_2(K) ) }, 
\end{align*}
where $\pi_2(K)$ is the projection of $K$ onto the $v$-axis.
Consequently,
\begin{align*}
    \iint_K W(u_0,u_0)(x,v) \mathrm{d}x\mathrm{d}v 
    &\le 
    C_\varepsilon
    \| |\widehat{u}_0|^2 \|_1 
    \big\| \| \mathbbm{1}_K(x,v) \|_{\dot{H}^{s_\varepsilon}_x} \big\|_{ L^{q_\varepsilon',1}( \pi_2(K) ) }
    \\
    &\le 
    C_\varepsilon
    \| |\widehat{u}_0|^2 \|_1 
    \big\| \| \mathbbm{1}_K(x,v) \|_{\dot{H}^{s_\varepsilon}_x} \big\|_{ L^{\infty }} 
    |\pi_2(K)|^{\frac1{q_\varepsilon'}} \\
    &\le 
    C_\varepsilon c_{s_\varepsilon}^{\frac{1}{2}}\diam_1(K)^{\frac{1-2s_\varepsilon}{2}}|\pi_2(K)|^{\frac1{q_\varepsilon'}} \| u_0\|_2^2. 
\end{align*}
It remains to observe that 
$$
\frac{1-2s_\varepsilon}{2} = \frac1{q_\varepsilon'} = \varepsilon, 
$$
and appeal to the fact that $\diam_1(K)$ is comparable to the average diameter $|K|/|\pi_2(K)|$ uniformly over all convex bodies $K$ by an application of Brunn's theorem.
\end{proof}
\begin{remark}[Higher dimensions]
Our proof of Theorem \ref{thm:flan} does not extend to higher dimensions, at least readily. This may already be seen if $K$ is the Euclidean unit ball in $\mathbb{R}^{2d}$, since its $d$-dimensional sections, also being Euclidean balls, fail to belong to $\dot{H}^s$ whenever $s\geq 1/2$; see \cite{Stein}. Evidently, a routine extension of our argument would require such control for all $s<d/2$. For further discussion of Sobolev norms of indicator functions we refer to \cite{Faraco}.
\end{remark}
\begin{remark}[Inequalities of Flandrin type for surface-carried Wigner distributions]
Our proof of Theorem \ref{thm:flan} reveals that the convexity hypothesis on $K$ may be weakened to the requirement that the sections $\{x\in\mathbb{R}:(x,v)\in K\}$ are intervals for each $v\in\mathbb{R}$, provided we replace the measure of $K$ with the diameter of $K$ in \eqref{Flandrinloss}. As such our argument should extend to Flandrin-type inequalities of the form
$$
\iint_K W_S(g,g)\lesssim \|g\|_{L^2(S)}^2
$$
for the surface-carried Wigner distributions $W_S$ of Section \ref{Sect:general submanifolds}, on the assumption that $K\subseteq TS$ is such that $\{v\in T_uS:(u,v)\in K\}$ is an interval for each $u\in S$. This would require a weak-type addition to Theorem \ref{l:EndpointKS}, analogous to Theorem 1(b) in \cite{KS}, and would introduce some dependence on the curvature quotient $Q(S)$.
\end{remark}
We conclude this section by establishing a simple direct connection between the parabolic Mizohata--Takeuchi inequality \eqref{MTpara} 
and the Flandrin-type inequality \eqref{flandrin}, although with one caveat: that the support condition on the right hand side of the parabolic Mizohata--Takeuchi inequality 
is dropped. 
\begin{proposition}\label{equivflan}
    If the Flandrin-type conjecture \eqref{flandrin} is true then the undirected Mizohata--Takeuchi inequality
    \begin{equation}\label{PSparaundirected}
\int_{\mathbb{R}^d\times\mathbb{R}}|u(x,t)|^2w(x,t)\mathrm{d}x\mathrm{d}t\lesssim \|\rho^*w\|_\infty\|u_0\|_2^2
\end{equation}
holds for space-time weight functions $w$ that are concave in the spatial variable.
\end{proposition}
\begin{proof}
    We begin by observing that if $w$ is a concave function in the spatial variable then $\rho^*w$ is a concave function. This is immediate since whenever $(x_\lambda,v_\lambda)=\lambda(x_1,v_1)+(1-\lambda)(x_2,v_2)\in\mathbb{R}^d\times\mathbb{R}^d$,
    \begin{eqnarray*}
    \begin{aligned}
\rho^*w(x_\lambda,v_\lambda)&=\int_{\mathbb{R}}w(\lambda(x_1-2tv_1)+(1-\lambda)(x_2-2tv_2),t)\mathrm{d}t\\&\geq 
\int_{\mathbb{R}}\left(\lambda w(x_1-2tv_1,t)+(1-\lambda)w(x_2-2tv_2,t)\right)\mathrm{d}t\\
&=
\lambda\rho^*w(x_1,v_1)+(1-\lambda)\rho^*w(x_2,v_2)
\end{aligned}
\end{eqnarray*}
for all $0<\lambda<1$.
Applying the layer-cake representation,
    \begin{equation}\label{representation}\rho^*w(x,v)=\int_0^{\|\rho^*w\|_\infty} \mathbbm{1}_{K(s)}(x,v)\mathrm{d}s,
    \end{equation}
    where 
    $K(s)=\{(x,v)\in\mathbb{R}^d\times\mathbb{R}^d: \rho^*w(x,v)\geq s\}$, it follows from the convexity of $K(s)$ for each $s$, Fubini's theorem, and the conjectural inequality \eqref{flandrin} that 
    \begin{eqnarray*}
        \begin{aligned}
        \int_{\mathbb{R}^d\times\mathbb{R}}|u(x,t)|^2w(x,t)\mathrm{d}x\mathrm{d}t=&
    \int_{\mathbb{R}^d\times\mathbb{R}^d}W(u_0,u_0)(x,v)\rho^*w(x,v)\mathrm{d}x\mathrm{d}v\\&=\int_0^{\|\rho^*w\|_\infty}\left(\iint_{K(s)}W(u_0,u_0)(x,v)\mathrm{d}x\mathrm{d}v\right)\mathrm{d}s\\
    &\lesssim\|\rho^* w\|_\infty\|u_0\|_2^2.
    \end{aligned}
    \end{eqnarray*}
\end{proof}
\begin{remark}
If instead of applying the conjectural \eqref{flandrin} one applies the established \eqref{Flandrinloss} in the proof of Proposition \ref{equivflan}, an application of Chebyshev's inequality reveals that
 \begin{eqnarray*}
        \begin{aligned}
        \int_{\mathbb{R}^d\times\mathbb{R}}|u(x,t)|^2w(x,t)\mathrm{d}x\mathrm{d}t&\leq C_\varepsilon\|u_0\|_2^2\int_0^{\|\rho^*w\|_\infty}|K(s)|^\varepsilon\mathrm{d}s
    \leq \frac{C_\varepsilon}{1-p\varepsilon} \|\rho^*w\|_p^{p\varepsilon}\|\rho^*w\|_\infty^{1-p\varepsilon}\|u_0\|_2^2
    \end{aligned}
    \end{eqnarray*}
   for $0<p<1/\varepsilon$. This might be interpreted as a certain $\varepsilon$-loss form of \eqref{PSparaundirected}. We thank one of the reviewers for suggesting such an observation.
    \end{remark}
\begin{remark}[Connections with maximally modulated singular integrals]
Our proof of Theorem \ref{thm:flan} hints at a connection between the Flandrin-type conjecture \eqref{flandrin} and another natural question in modern harmonic analysis. Specifically, for subsets $K$ of $\mathbb{R}\times\mathbb{R}$ whose vertical sections are intervals (and hence for convex $K$), a routine calculation reveals that
\begin{equation}\label{nicered}
\iint_K W(u_0,u_0)\lesssim\|H_*(u_0,u_0)\|_{L^1(\mathbb{R})},
\end{equation}
where $$
H_*(f_1,f_2)(x):=\sup_{\lambda\in\mathbb{R}}\left|\int_{\mathbb{R}}f_1\left(x+\frac{y}{2}\right)f_2\left(x-\frac{y}{2}\right)e^{i\lambda y}\frac{\mathrm{d}y}{y}\right|
$$
is the \textit{maximally-modulated bilinear Hilbert transform}. The Flandrin-type conjecture \eqref{flandrin} would therefore follow from the bound
\begin{equation}\label{maxmodconj}
\|H_*(f_1,f_2)\|_{L^1(\mathbb{R})}\lesssim\|f_1\|_{L^2(\mathbb{R})}\|f_2\|_{L^2(\mathbb{R})}.
\end{equation} 
The operator $H_*$ is a natural (bi-sublinear) analogue of the classical Carleson maximal operator.
Tools from time-frequency analysis have proved very effective in the study of various related maximally-modulated singular integral operators (such as in \cite{MTT1} and \cite{LM}) following the celebrated work of Lacey and Thiele \cite{LT1,LT2} on the boundedness properties of the bilinear Hilbert transform. However, as far as we are aware no nontrivial bounds for the operator $H_{*}$ are known. We note that the bound \eqref{maxmodconj} was established for certain ``non-resonant perturbations" of $H_*$ in \cite{benea}.
\end{remark}
\section{Questions}\label{Sect:tombounds}
Here we collect a number of questions, some concrete and some more speculative.
\begin{question}[Strengthening the parabolic Sobolev--Mizohata--Takeuchi inequality]
For nonnegative weights $w$, can one strengthen \eqref{SMTlocpara} to
$$\int_{\mathbb{R}^d\times\mathbb{R}}|u(x,t)|^2w(x,t)\mathrm{d}x\mathrm{d}t\lesssim\sup_{v\in\supp(\widehat{u}_0)}\|\rho^* w(\cdot,v)\|_{\dot{H}_x^s}\|u_0\|_2^2,$$ as suggested by \eqref{MTpara}? 
\end{question}
\begin{question}[Tomographic constructions of Wigner distributions in higher dimensions]
In Section \ref{Sect:tomographic} we saw that geometric Wigner distributions may be constructed tomographically from $|\widehat{g\mathrm{d}\sigma}|^2$ when $n=2$ using the X-ray transform. Might there be a similar tomographic construction of a Wigner distribution that functions in all dimensions, perhaps involving the Radon transform? 
\end{question}
\begin{question}[Fractional Stein and Mizohata--Takeuchi inequalities]
Are there interesting fractional forms of \eqref{Steinpara} or \eqref{MTpara} suggested by considering an oblique phase-space marginal of the Wigner distribution in place of \eqref{classical marginal}? See Remark \ref{phasetom} on phase-space tomography. 
    \end{question}
\begin{question}[A Flandrin-type inequality with an $\varepsilon$-loss in higher dimensions]
May the statement of Theorem \ref{thm:flan} be extended to dimensions $d>1$?
\end{question}

\end{document}